\documentclass[reqno]{amsart}
\usepackage{amsmath}
\usepackage{amssymb}
\usepackage{amsthm}
\usepackage{graphicx}
\usepackage{color}

\newtheorem{theorem}{Theorem}[section]
\newtheorem{proposition}[theorem]{Proposition}

\newtheorem{lemma}[theorem]{Lemma}
\theoremstyle{definition}
\newtheorem{definition}[theorem]{Definition}
\newtheorem{example}[theorem]{Example}

\newtheorem{remark}[theorem]{Remark}
\newtheorem{notation}[theorem]{Notation}

\newcommand{\oline}[1]{\mathbin{\overline{#1}}}
\newcommand{\uline}[1]{\mathbin{\underline{#1}}}

\newcounter{nofigs}
\renewenvironment{figure}[0]{\refstepcounter{nofigs}}
{\centerline{Figure \thenofigs}\medskip}

\newcounter{samfigs}
\newenvironment{samfigure}[4]{\refstepcounter{samfigs}\label{#2}
\centerline{\begin{tabular}{c}\includegraphics[#4]{#1}\\
Figure \thesamfigs #3\end{tabular}}}{}

\newcommand{\blangle}{{\Big\langle}\hspace{-2.1mm} {\Big\langle}\hspace{-2.1mm} {\Big\langle}\hspace{-2.1mm} {\Big\langle} \hspace{-2.1mm} {\Big\langle} \hspace{-2.1mm} {\Big\langle} \hspace{-2.1mm} {\Big\langle} \hspace{-2.1mm} {\Big\langle} \hspace{-2.1mm} {\Big\langle}}

\newcommand{\brangle}{{\Big\rangle}\hspace{-2.1mm} {\Big\rangle}\hspace{-2.1mm} {\Big\rangle}\hspace{-2.1mm} {\Big\rangle} \hspace{-2.1mm} {\Big\rangle} \hspace{-2.1mm} {\Big\rangle} \hspace{-2.1mm} {\Big\rangle} \hspace{-2.1mm} {\Big\rangle} \hspace{-2.1mm} {\Big\rangle}}

\newcommand{\blsq}{{\Big[}\hspace{-1.6mm}{\bf \Big[}\hspace{-1.6mm}{\Big[}\hspace{-1.6mm}{\Big[}\hspace{-1.6mm}{\Big[}\hspace{-1.6mm} {\Big[}\hspace{-1.6mm} {\Big[} \hspace{-1.6mm} {\Big[} \hspace{-1.6mm} {\Big[} \hspace{-1.6mm} {\Big[}} 

\newcommand{\brsq}{{\Big]}\hspace{-1.6mm}{\Big]}\hspace{-1.6mm}{\Big]}\hspace{-1.6mm}{\Big]}\hspace{-1.6mm}{\Big]}\hspace{-1.6mm} {\Big]}\hspace{-1.6mm} {\Big]}\hspace{-1.6mm} {\Big]}  \hspace{-1.6mm} {\Big]} \hspace{-1.6mm} {\Big]}}

\begin{document}



\title[Local biquandles and Niebrzydowski's tribracket theory]{Local biquandles and\\ Niebrzydowski's tribracket theory }

\author[S.~Nelson]{Sam Nelson} 
\address{Department of Mathematical Sciences, Claremont McKenna College, Claremont, CA 91711, USA}
\email{Sam.Nelson@cmc.edu}

\author[K.~Oshiro]{Kanako Oshiro}
\address{Department of Information and Communication Sciences, Sophia University, Tokyo 102-8554, Japan}
\email{oshirok@sophia.ac.jp}

\author[N.~Oyamaguchi]{Natsumi Oyamaguchi}
\address{Department of Teacher Education, Shumei University, Chiba 276-0003, Japan}
\email{p-oyamaguchi@mailg.shumei-u.ac.jp}



\date{\today}

\maketitle

\begin{abstract}
We introduce a new algebraic structure called \textit{local 
biquandles} and show how colorings of oriented classical link diagrams
and of broken surface diagrams are related to tribracket colorings. We 
define a (co)homology theory for local biquandles and show that it is isomorphic 
to Niebrzydowski's tribracket (co)homology. 
This implies that Niebrzydowski's (co)homology theory can be interpreted 
similary as biqandle (co)homology theory.
Moreover through the isomorphism between two cohomology groups, 
we show that Niebrzydowski's cocycle invariants and local biquandle cocycle 
invariants are the same. 
\end{abstract}

\section*{Introduction}

Invariants of knots and knotted surfaces defined in terms of colorings by
algebraic structures have a long history, including colorings by algebraic
structures such as groups, quandles, biquandles and more \cite{ElhamdadiNelson}.
Enhancements, called cocycle invariants,  of these invariants using cocycles 
in cohomology theories
associated to the coloring structures were popularized in the late 1990s
in papers such as \cite{CarterJelsovskyKamadaLangfordSaito03} and have been 
a topic of much study ever since.

In \cite{Niebrzydowski0}, colorings of the planar complement of a 
knot diagram by algebraic structures now known as \textit{(knot-theoretic) 
ternary quasigroups} were introduced and used to define invariants of knots.
In \cite{Niebrzydowski1,Niebrzydowski2},  a (co)homology theory for these structures
was introduced and used to enhance the coloring invariants.

In \cite{NeedellNelson16}, an algebraic structure known as \textit{biquasile} 
was introduced by the first author and used to 
define oriented link invariants via colorings of certain graphs obtained
from oriented link diagrams. These invariants were enhanced with Boltzmann
weights in \cite{ChoiNeedellNelson17} and used to distinguish orientable 
surface-links in \cite{KimNelson}. Biquasile colorings can be understood 
in terms of ternary quasigroup colorings, and these Boltzmann weights can 
be understood as enhancement by cocycles in ternary quasigroup cohomology.

In more recent papers by the first author, knot-theoretic ternary quasigroup 
structures have been studied and generalized in terms of ternary operations 
called \textit{Niebrzydowski tribrackets}. In \cite{PicoNelson}, tribracket
coloring invariants are extended to the case of oriented virtual links. In 
\cite{GravesNelsonTamagawa} tribracket colorings are extended to $Y$-oriented 
trivalent spatial graphs and handlebody-links, and in \cite{NeedellNelsonShi}, 
enhancements of tribracket colorings by structures called 
\textit{tribracket modules} are defined.

In this paper, we introduce a new algebraic structure called \textit{local 
biquandles} and show how colorings of oriented classical link diagrams
and of broken surface diagrams are related to tribracket colorings. We 
define a (co)homology theory for local biquandles and show that it is isomorphic 
to Niebrzydowski's tribracket (co)homology. 
This implies that Niebrzydowski's (co)homology theory can be interpreted 
similary as biqandle (co)homology theory since our local biquandle (co)homology theory 
is analogous to biquandle (co)homology theory.
Moreover through the isomorphism between two cohomology groups, 
we show that Niebrzydowski's cocycle invariants and local biquandle cocycle 
invariants are the same. We provide examples of cocycles
and computation of these cocycle invariants.
 
The paper is organized as follows: In Section~\ref{Preliminaries}, we review the definitions of links, surface-links and tribrackets,  and we introduce local biquandles. 
In Section~\ref{Local biquandle homology groups and cocycle invariants}, we define local biquandle (co)homology groups, colorings using local biquandles and cocycle invariants.
 In Section~\ref{Niebrzydowski's work}, we summarize Niebrzydowski's work shown in \cite{Niebrzydowski0, Niebrzydowski1,Niebrzydowski2}, that is, we review Niebrzydowski's (co)homology groups, colorings using tribrackets and cocycle invariants. 
In Section~\ref{Correspondence between our work and Niebrzydowski's work}, our main results are stated and proved, that is,  we show that local biquandle (co)homology groups are isomorphic to Niebrzydowski's ones, and  local biquandle cocycle invariants are the same as Niebrzydowski's ones. 
We provide some examples in Section~\ref{Examples}.

\section{Preliminaries} \label{Preliminaries}
\subsection{Knots, links, connected diagrams}
Throughout this paper, a knot/link  means an oriented classical knot/link. 

A knot diagram is always connected.  
A link diagram with at least two components is said to be {\it connected} if every component intersects another  component.
It is known that between two connected  diagrams $D$ and $D'$ that represent the same link, there exists a finite sequence of connected diagrams and oriented Reidemeister moves that transforms $D$ to $D'$, i.e., there exists 
\[
D=D_0 \overset{R_0} {\longrightarrow} D_1 \overset{R_1} {\longrightarrow}\cdots \overset{R_{i-1}} {\longrightarrow}D_i \overset{R_i} {\longrightarrow} \cdots   \overset{R_{n-1}} {\longrightarrow} D_n =D',
\]  
where for each $i\in \{0,1,\ldots ,n\}$, $R_i$ is an oriented Reidemeister move, and $D_i$ is a connected diagram of a link.

For a diagram $D$, we remove a small neighborhood of each crossing, and then, we call each connected component a {\it semi-arc} of $D$.  
In this paper, for a link diagram $D$, $\mathcal{SA}(D)$ means the set of semi-arcs of $D$ and $\mathcal{R}(D)$ means  the set of connected regions of $\mathbb R^2\setminus D$.

For each semi-arc $x$ of a connected diagram $D$, we set two parallel copies  $x^{(+)}$ and $x^{(-)}$  of $x$ as depicted in Figure~\ref{semi-arc}, where we note that $x^{(+)}$ (resp. $x^{(-)}$) is in the right-side (resp. left-side) of $x$.
The {\it $2$-parallel  $\widetilde{D}$} of $D$ is the gathering of the two copies for all semi-arcs of $D$, that is, $\widetilde{D}=\bigsqcup \mathcal{SA}(\widetilde{D})$, see Figure~\ref{semi-arc}, where throughout this paper, $\mathcal{SA}(\widetilde{D})$ means the set $\{x^{(\varepsilon)} ~|~ x\in \mathcal{SA}(D), \varepsilon \in \{+, -\}\}$. 
\begin{figure}
  \begin{center}
    \includegraphics[clip,width=8cm]{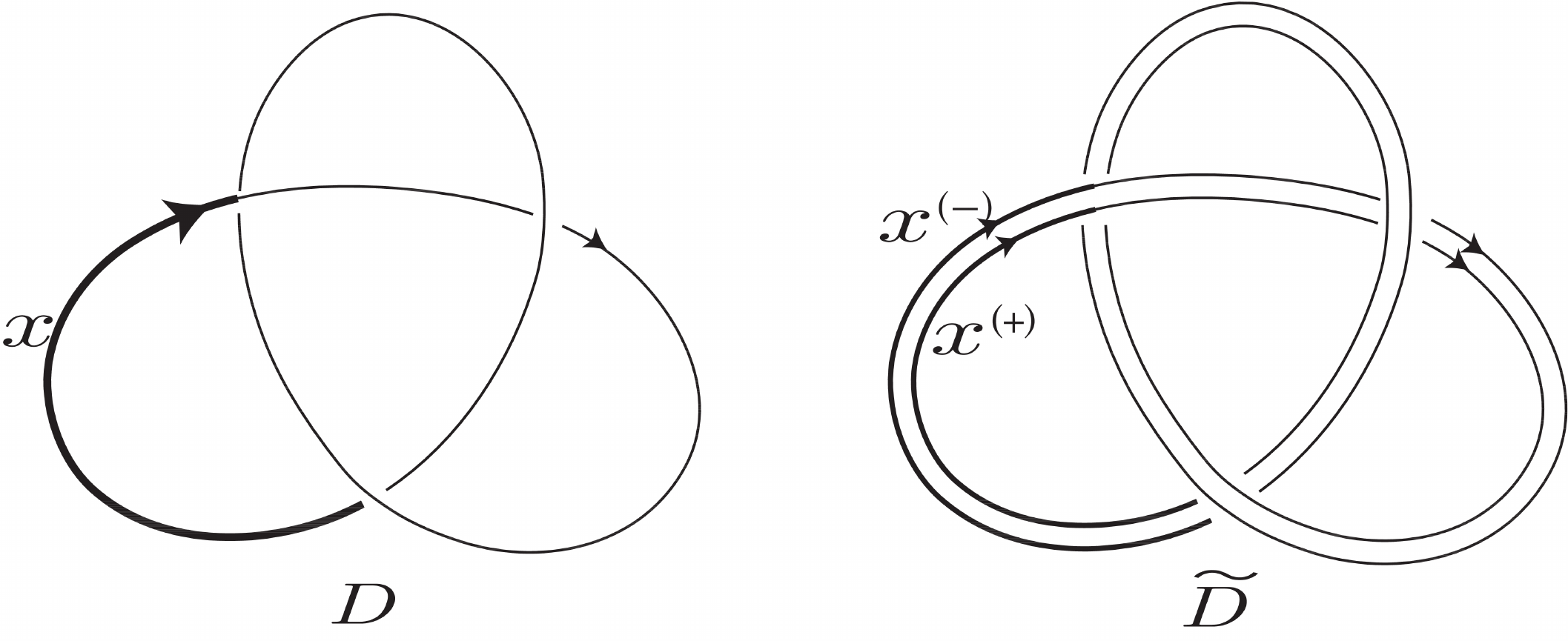}
    \label{semi-arc}
  \end{center}
\end{figure}


\subsection{Surface-knots, surface-links, connected diagrams}
A {\it surface-knot} is an oriented closed surface locally flatly embedded in $\mathbb R^4$. A {\it surface-link} is a disjoint union of surface-knots.  We note that every surface-knot is a surface-link. 
Two surface-links are said to be {\it equivalent} if they can be deformed into each other through an isotopy of $\mathbb R^4$.

A {\it diagram} of a surface-link is its image by a regular projection, from $\mathbb R^4$ to $\mathbb R^3$, equipped with the height information for each double point curve, where the height information is represented by removing small  neighborhoods of lower double point curves. Then a diagram is composed of  four kinds of local pictures depicted in Figure~\ref{multiplepoint}, and the indicated points are called a {\it regular point}, a {\it double point}, a {\it triple  point} and a {\it branch point}, respectively.  
It is known that two surface-link diagrams represent the same surface-link if and only if they are related by a finite sequence of Roseman moves, see \cite{Roseman} for details.
A surface-knot diagram is always connected.  
A surface-link diagram with at least two components is said to be {\it connected} if every component intersects another  component.
It is known that between two connected  diagrams $D$ and $D'$ that represent the same surface-link, there exists a finite sequence of connected diagrams and oriented Roseman moves that transforms $D$ to $D'$.
\begin{figure}
  \begin{center}
    \includegraphics[clip,width=10cm]{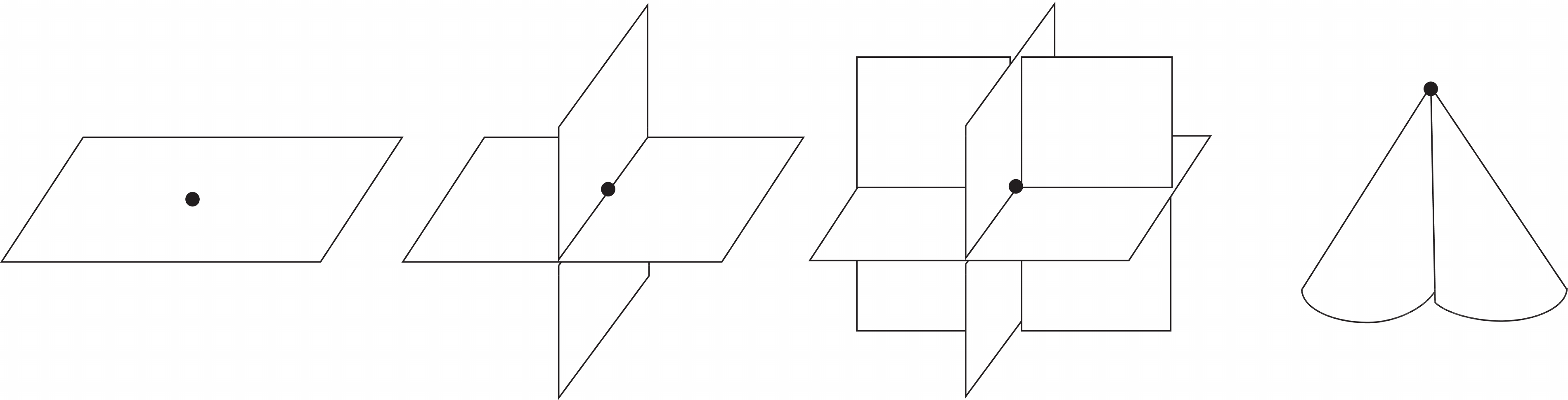}
    \label{multiplepoint}
  \end{center}
\end{figure}

For a surface-link diagram $D$, we remove small  neighborhoods of double point curves, and then, we call each connected component a {\it semi-sheet} of $D$.  
In this paper, for a surface-link diagram $D$, $\mathcal{SS}(D)$ means the set of semi-sheets of $D$ and $\mathcal{R}(D)$ means  the set of connected regions of $\mathbb R^3 \setminus D$.
For a semi-sheet $x$ of a surface-link diagram $D$, we assign a normal orientation $n_x$ to $x$ to satisfy that 
the triple $(o_1, o_2, n_x)$ of the orientation $(o_1, o_2)$ of $D$ and $n_x$ coincides with the right-handed orientation of $\mathbb R^3$, and thus, we represent the orientation of $D$.      

\subsection{Tribrackets}

\begin{definition}\label{def:horizontal}
 A {\it knot-theoretic horizontal-ternary-quasigroup} is a pair of  a   set $X$  and a ternary operation $[\, ]: X^3 \to X; (a,b,c) \mapsto [a,b,c]$ satisfying the following property:
\begin{enumerate}
\item[] \hspace{-0.5cm}($\mathcal{H}$1) For any $a,b,c\in X$, 
\begin{itemize}
\item[(i)] there exists a unique $d_1\in X$ such that $[a,b,d_1]=c$, 
\item[(ii)] there exists a unique $d_2\in X$ such that $[a,d_2,b]=c$, 
\item[(iii)] there exists a unique $d_3 \in X$ such that $[d_3,a,b]=c$.
\end{itemize}
\item[] \hspace{-0.5cm}($\mathcal{H}$2) For any $a,b,c,d \in X$, it holds that 
\[
\begin{array}{l}
[b,[a,b,c],[a,b,d]] = [c,[a,b,c],[a,c,d]]=[d,[a,b,d],[a,c,d]].
\end{array} 
\]
\end{enumerate}
We call the operation $[\,]$  a {\it horizontal-tribracket}.
\end{definition}
The axioms of a knot-theoretic horizontal-ternary-quasigroup $(X, [\,])$ are obtained from the oriented Reidemeister moves of link diagrams, which is observed when we consider an assignment of an element of $X$ to each region of a link diagram satisfying the condition depicted in Figure~\ref{coloring2}.
See Figure~\ref{RmoveIII} for the correspondence between the Reidemeister move of type III and the axiom ($\mathcal{H}$2) of a knot-theoretic horizontal-ternary-quasigroup $(X, [\,])$.
\begin{figure}
  \begin{center}
    \includegraphics[clip,width=7.0cm]{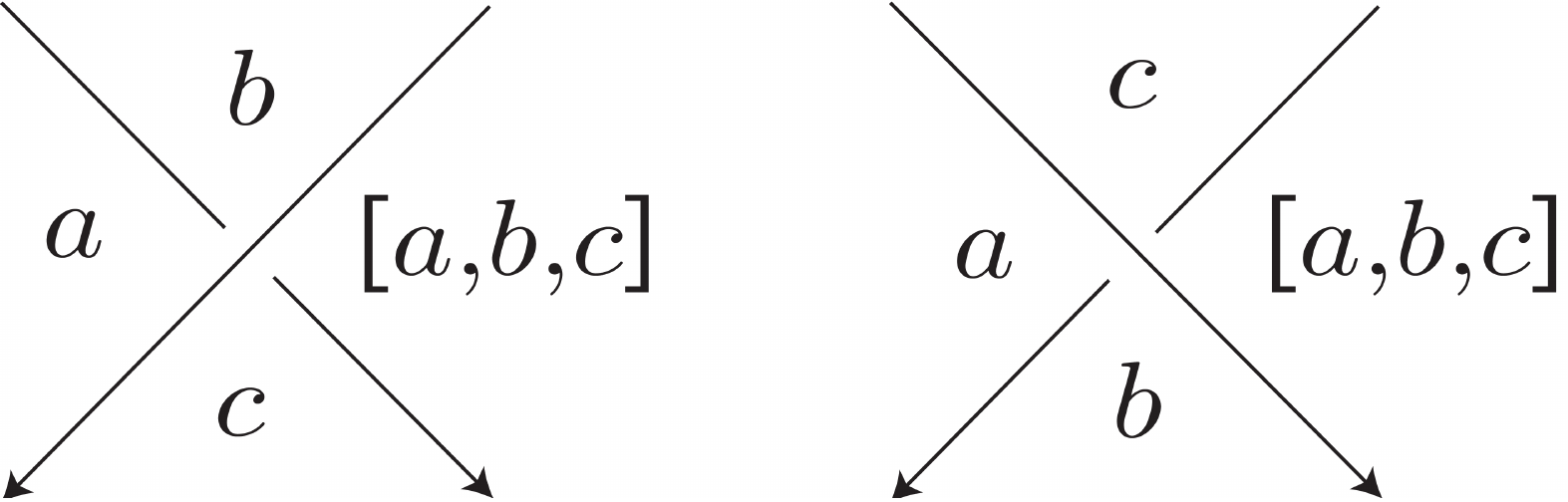}
    \label{coloring2}
  \end{center}
\end{figure}
\begin{figure}
  \begin{center}
    \includegraphics[clip,width=6.0cm]{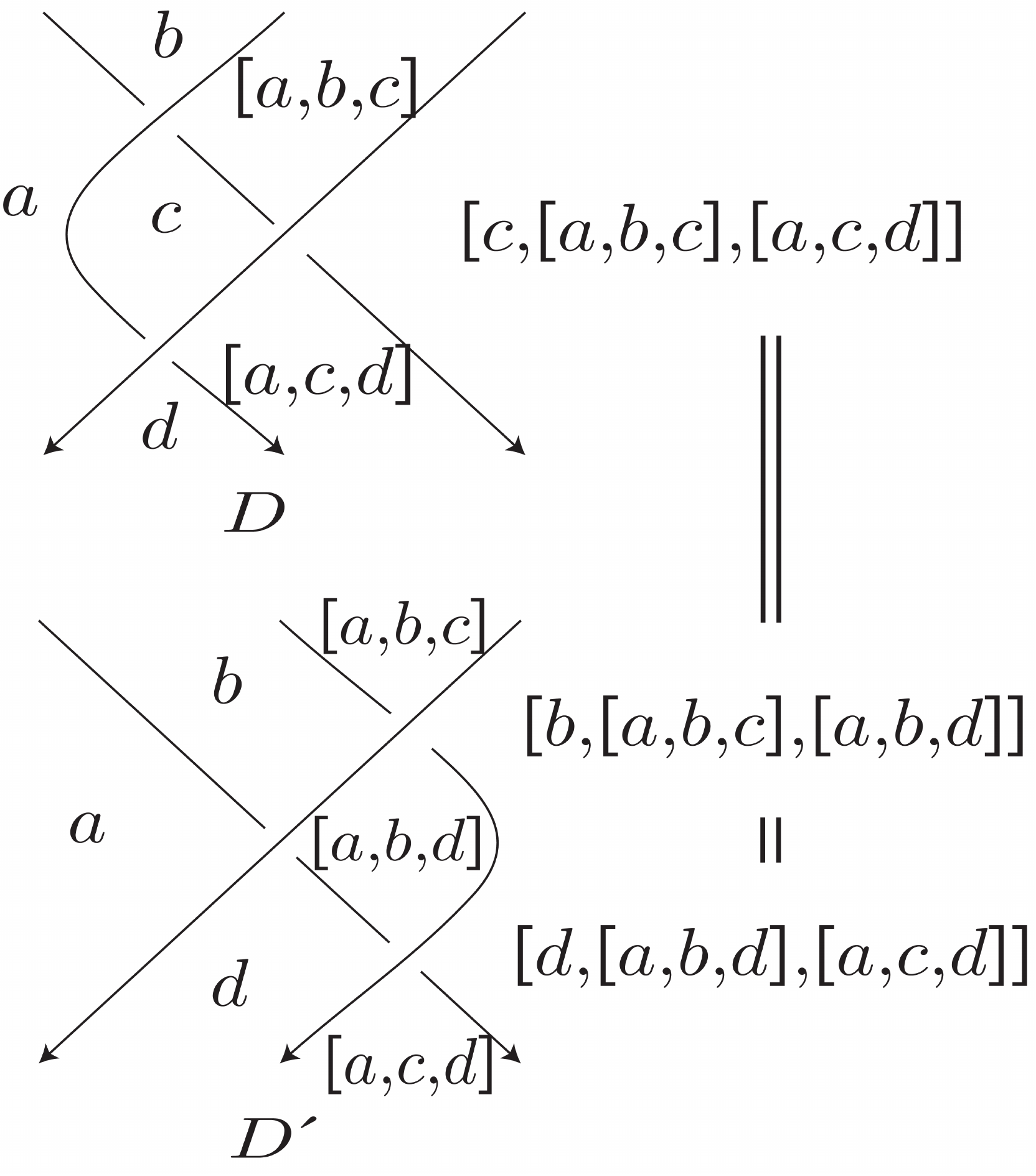}
    \label{RmoveIII}
  \end{center}
\end{figure}

\begin{definition} \label{s-ternary operation}
 A {\it knot-theoretic vertical-ternary-quasigroup} is a pair of  a   set $X$  and a ternary operation $\langle\, \rangle: X^3 \to X; (a,b,c) \mapsto \langle a,b,c \rangle $ satisfying the following property:
\begin{itemize}
\item[] \hspace{-0.5cm}($\mathcal{V}$1) For any $a,b,c\in X$, 
\begin{itemize}
\item[(i)] there exists a unique $d_1\in X$ such that $\langle a,b,d_1\rangle=c$,
\item[(ii)] there exists a unique $d_2\in X$ such that $\langle a, d_2, b \rangle=c$,
\item[(iii)] there exists a unique $d_3\in X$ such that $\langle d_3,a,b \rangle=c$.
\end{itemize}
\item[] \hspace{-0.5cm}($\mathcal{V}$2)  For any $a,b,c,d \in X$, it holds that 
\begin{itemize}
\item[(i)] $\big\langle a , \langle a,b,c\rangle ,  \langle \langle  a, b, c\rangle ,  c, d \rangle \big \rangle  =\big\langle  a, b, \langle  b, c, d \rangle \big \rangle $ and  
\item[(ii)] $\big\langle \langle a, b, c \rangle, c, d \big \rangle =\big\langle \langle  a, b, \langle b, c, d\rangle \rangle, \langle b, c, d \rangle ,  d \big\rangle.$
\end{itemize}
\end{itemize}
We call the operation $\langle\, \rangle$  a {\it vertical-tribracket}. 
\end{definition}
The axioms of a vertical-tribracket are obtained from the oriented Reidemeister moves of link diagrams, which is observed when we consider  an assignment of an element of $X$ to each region of a link diagram satisfying the condition depicted in Figure~\ref{coloring3} (see Definition~\ref{def:regioncoloring1}).
\begin{figure}
  \begin{center}
    \includegraphics[clip,width=7.0cm]{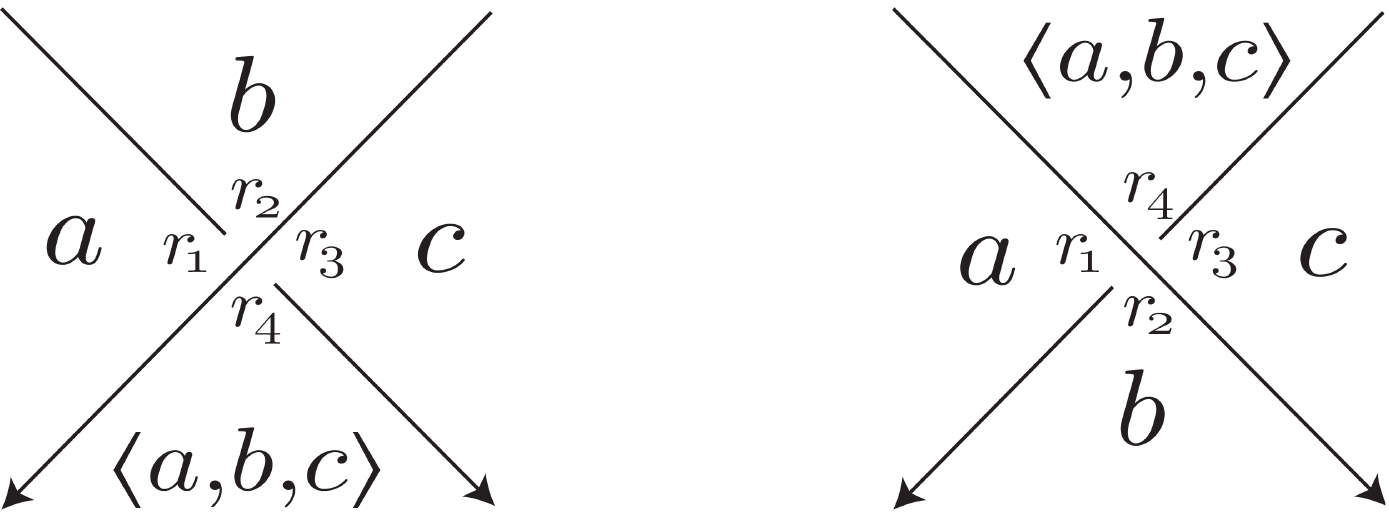}
    \label{coloring3}
  \end{center}
\end{figure}

\begin{remark}\label{rem:correspondence}
Horizontal- and vertical-tribrackets are induced from each other as follows:
Each vertical-tribracket $\langle\, \rangle$ satisfies the condition 
\begin{itemize}
\item[] \hspace{-0.5cm}($\mathcal{V}$1)-(i) For any $a,b,c \in X$, there exists a unique $d\in X$ such that $\langle a, b, d \rangle =c$, 
\end{itemize}
and hence,  by defining a horizontal-tribracket $[\,]: X^3\to X$ by $ (a,b,c) \mapsto d(=:[a,b,c])$, $[\, ]$ is induced from $\langle\, \rangle$. We call $[\, ]$ the {\it corresponding horizontal-bracket} of $\langle\, \rangle$. 
On the other hand, each horizontal-tribracket $[\,]$ satisfies the condition 
\begin{itemize}
\item[] \hspace{-0.5cm}($\mathcal{H}$1)-(i) For any $a,b,c \in X$, there exists a unique $d\in X$ such that $[a,b,d]=c$, 
\end{itemize}
and hence, by defining a vertical-tribracket $\langle\, \rangle: X^3\to X$ by $ (a,b,c) \mapsto d(=:\langle a,b,c \rangle)$,  $\langle\, \rangle$ is induced from $[\, ]$. We call $\langle\, \rangle$ the {\it corresponding vertical-bracket} of $[\, ]$.
\end{remark}

Next lemma follows from Remark~\ref{rem:correspondence}, see also Figure~\ref{tribracketsfomula1}.
\begin{lemma}\label{lem:tribrackets1}
 For a   set $X$, let $[\, ]$ be a horizontal-tribracket on $X$ and $\langle\, \rangle$
the corresponding vertical-tribracket of $[\,]$. Then for any $a,b,c,d\in X$, we have 
\[
\mbox{\rm (1)}~~c=\langle a,b,[a,b,c]\rangle\quad \mathrm{and}\quad \mbox{\rm (2)}~~
d=[a,b,\langle a,b,d\rangle].\]
\end{lemma}
\begin{figure}
  \begin{center}
    \includegraphics[width=5cm]{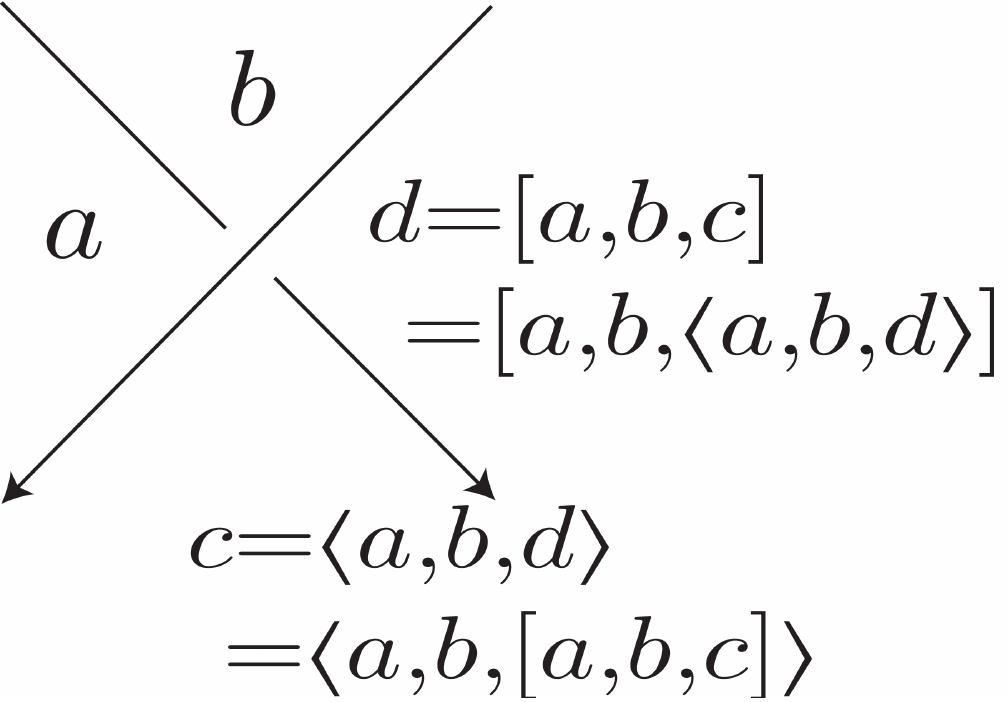}
    \label{tribracketsfomula1}
  \end{center}
\end{figure}

\begin{lemma}\label{lem:tribrackets2}
For a   set $X$, let $[\, ]$ be a horizontal-tribracket on $X$ and $\langle\, \rangle$
the corresponding vertical-tribracket of $[\,]$. Then for any $a,b,c,d\in X$, we have 
\begin{itemize}
\item[(1)]
\begin{eqnarray*}
{}\big[b,d,[a,b,c]\big] &  = & \big[\langle a,b,d\rangle, d, [a,\langle a,b,d\rangle, c]\big]\\
& = &\big[c,[a,b,c],[a,\langle a,b,d\rangle, c]\big],\\
\end{eqnarray*}
\item[(2)]
\begin{eqnarray*}
{}\big[a,\langle a,b,d\rangle, c\big] & = &  \big\langle c,[a,b,c], [b,d,[a,b,c]]\big\rangle\\
 & = & \big\langle \langle a,b,d\rangle, d,[b,d,[a,b,c]]\big\rangle.
\end{eqnarray*}
\end{itemize}
\end{lemma}
\begin{proof}
(1) \  We have 
\[
\begin{array}{ll}
\big[c,[a,b,c],[a,\langle a,b,d\rangle, c]\big]\\[3pt]
= \big[\langle a,b,d\rangle, [a,b,\langle a,b,d\rangle], [a,\langle a,b,d\rangle, c]\big]\\[3pt]
= \big[\langle a,b,d\rangle, d, [a,\langle a,b,d\rangle, c]\big],\\
\end{array}
\]
where the first equality follows from the second equality  of ($\mathcal{H}$2) of Definition~\ref{def:horizontal} and the second equality follows from (2) of Lemma~\ref{lem:tribrackets1}. 
We leave the proof of the other equality of (1) to the reader, see also  Figure~\ref{tribracketsfomula2}.

(2) \  We have 
\[
\begin{array}{ll}
\big\langle \langle a,b,d\rangle, d,[b,d,[a,b,c]]\big\rangle\\[3pt]
= \big\langle \langle a,b,\langle b,d,[b,d,[a,b,c]]\rangle \rangle, \langle b,d,[b,d,[a,b,c]]\rangle,[b,d,[a,b,c]] \big\rangle\\[3pt]
=\big\langle \langle a,b,[a,b,c]\rangle,[a,b,c], [b,d,[a,b,c]]\big\rangle\\[3pt]
=\big\langle c,[a,b,c], [b,d,[a,b,c]]\big\rangle,\\
\end{array}
\]
where the first equality follows from (ii) of ($\mathcal{V}$2) of Definition~\ref{s-ternary operation} and the second and third equalities follow from (1) of Lemma~\ref{lem:tribrackets1}. We leave the proof of the other equality of (2) to the reader, see also  Figure~\ref{tribracketsfomula2}.

\end{proof}
\begin{figure}
  \begin{center}
    \includegraphics[width=13.0cm]{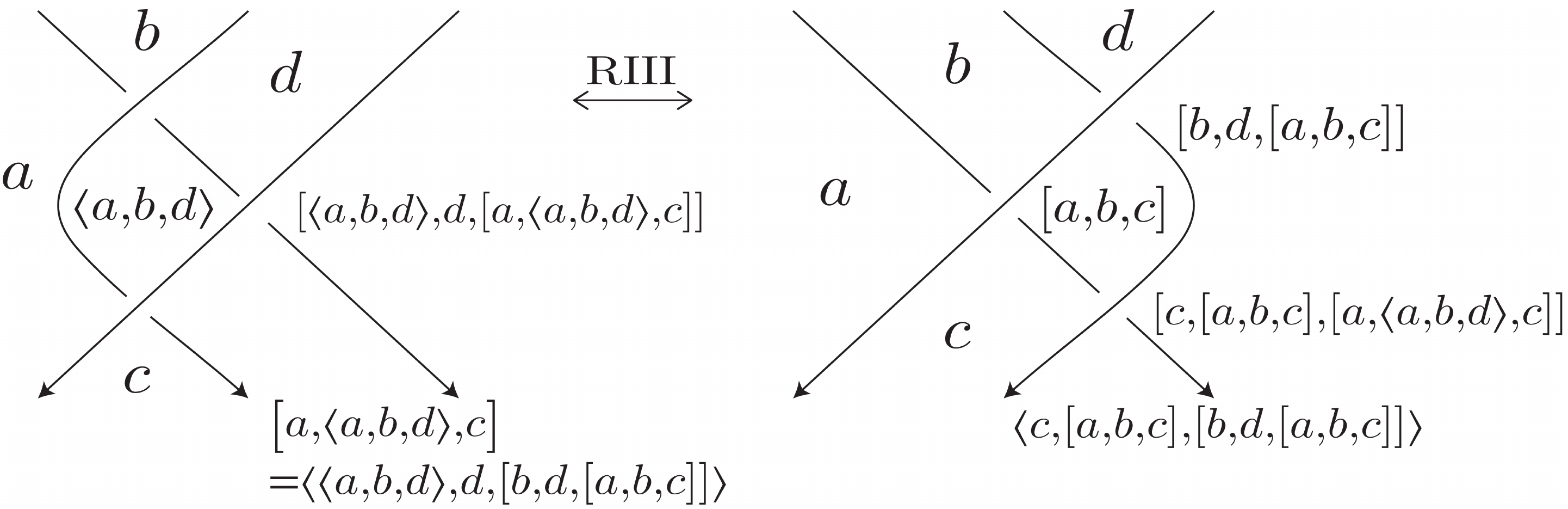}
    \label{tribracketsfomula2}
  \end{center}
\end{figure}

\begin{remark}
In \cite{Niebrzydowski1}, Niebrzydowski defined 
a knot-theoretic ternary quasigroup $(X, T)$ with the left division $\mathcal{L}$, the middle division $\mathcal{M}$ and the right  devision $\mathcal{R}$.
An horizontal-tribracket $[\,]$ in this paper coincides with a right division $\mathcal{R}$ in \cite{Niebrzydowski1}.
A vertical-tribracket $\langle\, \rangle$ in this paper coincides with his tribracket $T$ in \cite{Niebrzydowski1}.
In \cite{NeedellNelson16, ChoiNeedellNelson17}, binary operations related to vertical-tribrackets are defined. 
\end{remark}

\subsection{Local biquandles}
\begin{definition}\label{def:localbiquandle}
A {\it local biquandle} is a triple $(X,  \{\uline{\star}_a\}_{a\in X}, \{\oline{\star}_a\}_{a\in X}  )$ of a set $X$ and two families of operations $\uline{\star}_a, \oline{\star}_a: (\{a\} \times X)^2 \to X^2$ 
satisfying the following property: 
\begin{itemize}
\item[] \hspace{-0.5cm}($\mathcal{L}$1) For any $a,b,c \in X$, 
\begin{itemize}
\item[(i)] the first component of the result of $(a, b) \uline{\star}_a (a,c)$ is $c$,
\item[(ii)] the first component of the result of $(a, b) \oline{\star}_a (a,c)$ is $c$,
\item[(iii)] the second component of the result of $(a, b) \uline{\star}_a (a,c)$ coincides with that of the result of  $(a, c) \oline{\star}_a (a,b)$.
\end{itemize}
\item[] \hspace{-0.5cm}($\mathcal{L}$2)
\begin{itemize}
\item[(i)] For any $a, b\in  X$, the map $\uline{\star}_a(a, b) :\{a\}\times X \to \{b\}\times X$ sending $(a,c)$ to $(a,c)\uline{\star}_a(a,b)$ is bijective.
\item[(ii)]
For any $a, b\in X$, the map $\oline{\star}_a(a,b):\{a\}\times X \to \{b\}\times  X$ sending $(a,c)$ to $(a,c)\oline{\star}_a(a,b)$ is bijective.
\item[(iii)]
The map $S: \bigcup_{a\in X}(\{a\} \times X)^2\to \bigcup_{d\in X} (X\times \{d\})^2$ defined by $S\big((a,b),(a,c)\big)=\big((a,c)\oline{\star}_a(a,b),(a,b)\uline{\star}_a(a,c) \big)$ is bijective.
\end{itemize}
\item[] \hspace{-0.5cm}($\mathcal{L}$3)
For any $a,b,c\in X$, it holds that
\begin{itemize}
\item[(i)] $\big((a,b)\uline{\star}_a(a,c)\big)\uline{\star}_c\big((a,d)\oline{\star}_a(a,c)\big)=\big((a,b)\uline{\star}_a(a,d)\big)\uline{\star}_d\big((a,c)\uline{\star}_a(a,d)\big)$,
\item[(ii)] 
$\big((a,b)\oline{\star}_a(a,c)\big)\uline{\star}_c\big((a,d)\oline{\star}_a(a,c)\big)=\big( (a,b)\uline{\star}_a(a,d)\big)\oline{\star}_d\big((a,c)\uline{\star}_a(a,d)\big)$,
\item[(iii)] $\big((a,b)\oline{\star}_a(a,c)\big)\oline{\star}_c\big((a,d)\uline{\star}_a(a,c)\big)=\big((a,b)\oline{\star}_a(a,d)\big)\oline{\star}_d\big((a,c)\oline{\star}_a(a,d)\big)$.
\end{itemize}
\end{itemize}

In this paper, for simplicity, we often omit the subscript by $a$ as $\uline{\star}=\uline{\star}_a$, $\oline{\star}=\oline{\star}_a$,  $\{\uline{\star}\}=\{\uline{\star}_a\}_{a\in X}$ and $\{\oline{\star}\}=\{\oline{\star}_a\}_{a\in X}$ unless it causes confusion. 
\end{definition}
\noindent We note that from ($\mathcal{L}$1), we can obtain that
\begin{itemize}
\item[] \hspace{-0.5cm}($\mathcal{L}$1') for any $a, b\in X$, it holds that $(a,b)\uline{\star}_a (a,b) =(a,b)\oline{\star }_a (a,b)$.
\end{itemize}
The property ($\mathcal{L}$1') and the axioms ($\mathcal{L}$2) and ($\mathcal{L}$3) are analogous to the axioms of a biquandle, see \cite{FennRourkeSanderson92, KR02} for the definition of a biquandle.

\begin{proposition}\label{prop:localbqandhorizontaltqg}
\begin{itemize}
\item[(1)] Given a local biquandle $(X,  \{\uline{\star}_a\}_{a\in X}, \{\oline{\star}_a\}_{a\in X}  )$, we have the knot-theoretic horizontal-ternary-quasigroup $(X, [\, ])$ whose horizontal-tribracket $[\,]$ sends $(a,b,c)$ to the second component of the result of $(a, b) \uline{\star}_a (a,c)$ (or $(a, c) \oline{\star}_a (a,b)$).
\item[(2)] Given a knot-theoretic horizontal-ternary-quasigroup $(X, [\, ])$, we have the local biquandle $(X,  \{\uline{\star}_a\}_{a\in X}, \{\oline{\star}_a\}_{a\in X}  )$ whose operations $\uline{\star}_a$ and $\oline{\star}_a$ are defined by 
\[
\begin{array}{l}
(a, b) \uline{\star}_a (a,c) = (c, [a,b,c]), \mbox{ and }\\[5pt]
(a,b) \oline{\star}_a (a,c) = (c,[a,c,b]). 
\end{array}
\]
\item[(3)] The correspondences of (1) and (2) of this proposition are inverse of each other.
\end{itemize}
\end{proposition}
\begin{proof}
(1) Suppose that we have a local biquandle $(X,  \{\uline{\star}_a\}_{a\in X}, \{\oline{\star}_a\}_{a\in X}  )$. We show that the map $[\, ]: X^3 \to X$ sending $(a,b,c)$ to the second component of the result of $(a, b) \uline{\star}_a (a,c)$  (or $(a, c) \oline{\star}_a (a,b)$) satisfies the axioms of horizontal-tribrackets in Definition~\ref{def:horizontal}.

For $a,b,c\in X$,  
since the map $\uline{\star}_a(a,b): \{a\} \times X \to \{b\} \times X $ is bijective by the axiom ($\mathcal{L}$2)-(i) of Definition~\ref{def:localbiquandle}, there exists a unique $d\in X$ such that $(b,c)=(a,d) \uline{\star}_a (a,b)= (b, [a,d,b])$, that is, $[a,d,b] =c$, which satisfies the axiom ($\mathcal{H}$1)-(ii) of Definition~\ref{def:horizontal}. Similarly, the axioms ($\mathcal{L}$2)-(ii) and ($\mathcal{L}$2)-(iii) of Definition~\ref{def:localbiquandle} lead to the axioms ($\mathcal{H}$1)-(i) and  ($\mathcal{H}$1)-(iii) of Definition~\ref{def:horizontal}, respectively.

For $a,b,c, d \in X$, we have 
\[
\begin{array}{ll}
\big([a,c,d], [c, [a,b,c], [a,c,d]]\big)\\[3pt]
=\big((a,b)\uline{\star}_a(a,c)\big)\uline{\star}_c\big((a,d)\oline{\star}_a(a,c)\big)\\[3pt]
=\big((a,b)\uline{\star}_a(a,d)\big)\uline{\star}_d\big((a,c)\uline{\star}_a(a,d)\big)\\[3pt]
=\big([a,c,d], [d,[a,b,d],[a,c,d]]\big)\\
\end{array}
\]
by the axiom ($\mathcal{L}$3)-(i) of Definition~\ref{def:localbiquandle}, 
\[
\begin{array}{ll}
\big([a,b,d], [b, [a,b,c], [a,b,d]]\big)\\[3pt]
=\big((a,c)\oline{\star}_a(a,b)\big)\uline{\star}_b\big((a,d)\oline{\star}_a(a,b)\big)\\[3pt]
=\big( (a,c)\uline{\star}_a(a,d)\big)\oline{\star}_d\big((a,b)\uline{\star}_a(a,d)\big)\\[3pt]
=\big([a,b,d], [d,[a,b,d],[a,c,d]]\big)\\
\end{array}
\]
by the axiom ($\mathcal{L}$3)-(ii) of Definition~\ref{def:localbiquandle}, and
\[
\begin{array}{ll}
\big([a,b,c], [c, [a,b,c], [a,c,d]]\big)\\[3pt]
=\big((a,d)\oline{\star}_a(a,c)\big)\oline{\star}_c\big((a,b)\uline{\star}_a(a,c)\big)\\[3pt]
=\big((a,d)\oline{\star}_a(a,b)\big)\oline{\star}_b\big((a,c)\oline{\star}_a(a,b)\big)\\[3pt]
=\big([a,b,c], [b,[a,b,c],[a,b,d]]\big)\\
\end{array}
\]
by the axiom ($\mathcal{L}$3)-(iii) of Definition~\ref{def:localbiquandle}.
Thus the axiom ($\mathcal{H}$2) of Definition~\ref{def:horizontal} can be obtained from the equality of the second components. We leave the proof of  (2) and (3) to the reader. 
\end{proof}
For a local biquandle $(X, \{\uline{\star}\} , \{\oline{\star}\})$, we call $(X,[\,])$ in (1) of Proposition~\ref{prop:localbqandhorizontaltqg} the {\it corresponding knot-theoretic horizontal-ternary-quasigroup of $(X, \{\uline{\star}\} , \{\oline{\star}\})$}. For a knot-theoretic horizontal-ternary-quasigroup $(X,[\,])$, we call $(X, \{\uline{\star}\} , \{\oline{\star}\})$ in (2) of Proposition~\ref{prop:localbqandhorizontaltqg} the {\it corresponding local biquandle of $(X,[\,])$}.

\section{Local biquandle homology groups and cocycle invariants}\label{Local biquandle homology groups and cocycle invariants}

\subsection{Local biquandle homology groups}
Let $(X, \{\uline{\star}\} , \{\oline{\star}\})$ be a local biquandle, and $(X,[\,])$ the corresponding knot-theoretic horizontal-ternary-quasigroup of $(X, \{\uline{\star}\} , \{\oline{\star}\})$.
Let $n \in \mathbb Z$.
Let $C_n(X)$ be the free $\mathbb Z$-module generated by the elements of 
\[
\bigcup_{a\in X} (\{a\} \times X)^n =\big\{\big( (a,b_1), (a,b_2), \ldots , (a,b_n) \big) ~|~ a, b_1, \ldots , b_n \in X \big\}
\]
if $n\geq 1$, and $C_n(X)=0$ otherwise.
We define a homomorphism $\partial_n : C_n (X) \to C_{n-1} (X)$ by 
\begin{align}
&\partial_n \Big( \big( (a,b_1), \ldots , (a,b_n) \big) \Big) = \sum_{i=1}^{n} (-1)^i \big\{ \big( (a,b_1), \ldots, \widehat{(a, b_i)}, \ldots  , (a,b_n) \big) \notag \\
&- \big( (a,b_1)\uline{\star} (a, b_i), \ldots, (a,b_{i-1})\uline{\star} (a, b_i),    (a,b_{i+1})\oline{\star} (a, b_i) ,\ldots  , (a,b_n) \oline{\star} (a, b_i)\big) \big\} \notag \\
 &= \sum_{i=1}^{n} (-1)^i \big\{ \big( (a,b_1), \ldots, \widehat{(a, b_i)}, \ldots  , (a,b_n) \big) \notag \\
&- \big( (b_i, [a, b_1, b_i] ), \ldots, (b_i, [a, b_{i-1}, b_i]),    (b_i, [a, b_i, b_{i+1}]) ,\ldots  , (b_i, [a, b_i, b_{n}])\big) \big\} \notag 
\end{align}
if $n\geq 2$, and $\partial_n=0$ otherwise, where $\widehat{x}$ means that the component $x$ is removed.
 \begin{lemma}\label{lem:chain complex}
$C_*(X)=\{C_n(X), \partial_n\}_{n\in \mathbb Z}$ is a chain complex, i.e., for any $n \in \mathbb Z$, $\partial_n\circ \partial_{n+1} =0$ holds.
\end{lemma}
\begin{proof}
This is shown by a direct calculation as in the case of biquandle chain complexes, see \cite{CarterElhamdadiSaito04, CenicerosElhamdadiGreenNelson14}.
\end{proof}

Let $D_n(X)$ be a submodule of $C_n(X)$ which is generated by the elements of 
\[
\Big\{\big( (a,b_1),  \ldots , (a,b_n) \big) \in \bigcup_{a\in X
}  (\{a\} \times X)^n ~\Big|~ \mbox{ $b_i =b_{i+1}$ for some $i\in \{1, \ldots , n-1 \} $  }  \Big\}.
\]
\begin{lemma}
$D_*(X)=\{D_n(X), \partial_n\}_{n\in \mathbb Z}$ is a subchain complex of $C_*(X)$, i.e., for any $n \in \mathbb Z$, $\partial_n (D_n(X)) \subset D_{n-1} (X)$ holds.
\end{lemma}
\begin{proof}
This is shown by a direct calculation as in the case of biquandle subchain complexes, see \cite{CarterElhamdadiSaito04, CenicerosElhamdadiGreenNelson14}.
\end{proof}
\noindent Therefore the chain complex $$C_*^{\rm LB} (X)=\{C_n^{\rm LB}(X):=C_n(X)/D_n(X), \partial_n^{\rm LB}:=\partial_n\}_{n\in \mathbb Z}$$ is induced.
We call the homology group $H_n^{\rm LB} (X)$ of $C_*^{\rm LB} (X)$ the \textit{$n$th local  biquandle homology group} of $(X, \{\uline{\star}\} , \{\oline{\star}\})$.

For an abelian group $A$, we define the chain and cochain complexes by 
\[
\begin{array}{l}
C_n^{\rm LB}(X; A)=C_n^{\rm LB}(X) \otimes A, \quad \partial_n^{\rm LB} \otimes {\rm id}  \mbox{ and }\\[5pt]
C_{\rm LB}^n(X; A) ={\rm Hom}(C_n^{\rm LB}(X); A), \quad \delta^n_{\rm LB} \mbox{ s.t. }\delta^n_{\rm LB}(f)=f \circ \partial_{n+1}^{\rm LB}.
\end{array}
\]
Let $C_\ast^{\rm LB}(X; A)=\{C_n^{\rm LB}(X; A), \partial_n^{\rm LB}\otimes {\rm id}\}_{n\in \mathbb Z}$ and $C_{\rm LB}^\ast(X; A)=\{C_{\rm LB}^n(X; A), \delta^n_{\rm LB}\}_{n\in \mathbb Z}$. The \textit{nth homology group} $H_n^{\rm LB}(X; A)$ \textit{and nth cohomology group} $H^n_{\rm LB}(X; A)$ of $(X, \{\uline{\star}\} , \{\oline{\star}\})$ with coefficient group $A$ are defined by
\[
H_n^{\rm LB}(X; A)=H_n(C_\ast^{\rm LB}(X; A)) \qquad {\rm and} \qquad H_{\rm LB}^n(X; A)=H^n(C^\ast_{\rm LB}(X; A)).
\]
Note that we omit the coefficient group  $A$ if $A=\mathbb Z$ as usual. 

We remark that these definitions are anologous to those in biquandle (co)homology theory.

\subsection{Semi-arc colorings of link diagrams, cocycle invariants}\label{subsection:linkinvariant}
Let $(X, \{\uline{\star}\} , \{\oline{\star}\})$ be a local biquandle, and $(X,[\,])$ the corresponding knot-theoretic horizontal-ternary-quasigroup of $(X, \{\uline{\star}\} , \{\oline{\star}\})$. Let $D$ be a connected diagram of a link.
\begin{definition}
A \textit{semi-arc $X^2$-coloring} of $D$ is a map $C: \mathcal{SA}(D) \to X^2$ satisfying the following condition:
\begin{itemize}
\item For a crossing composed of under-semi-arcs $u_1, u_2$ and over-semi-arcs $o_1, o_2$ as depicted in Figure~\ref{coloring1}, let $C(u_1)=(a_1,b), C(o_1)=(a_2, c)$. Then 
\begin{itemize}
\item $a_1=a_2$,  
\item $C(u_2) = C(u_1) \uline{\star} C(o_1)= (a,b) \uline{\star} (a,c) = (c, [a,b,c])$, and 
\item $C(o_2) = C(o_1) \oline{\star} C(u_1)=(a,c) \oline{\star} (a,b) = (b, [a,b,c])$
\end{itemize}
hold,  where $a= a_1 =a_2$, see Figure~\ref{coloring1}.
\begin{figure}
  \begin{center}
    \includegraphics[clip,width=10.0cm]{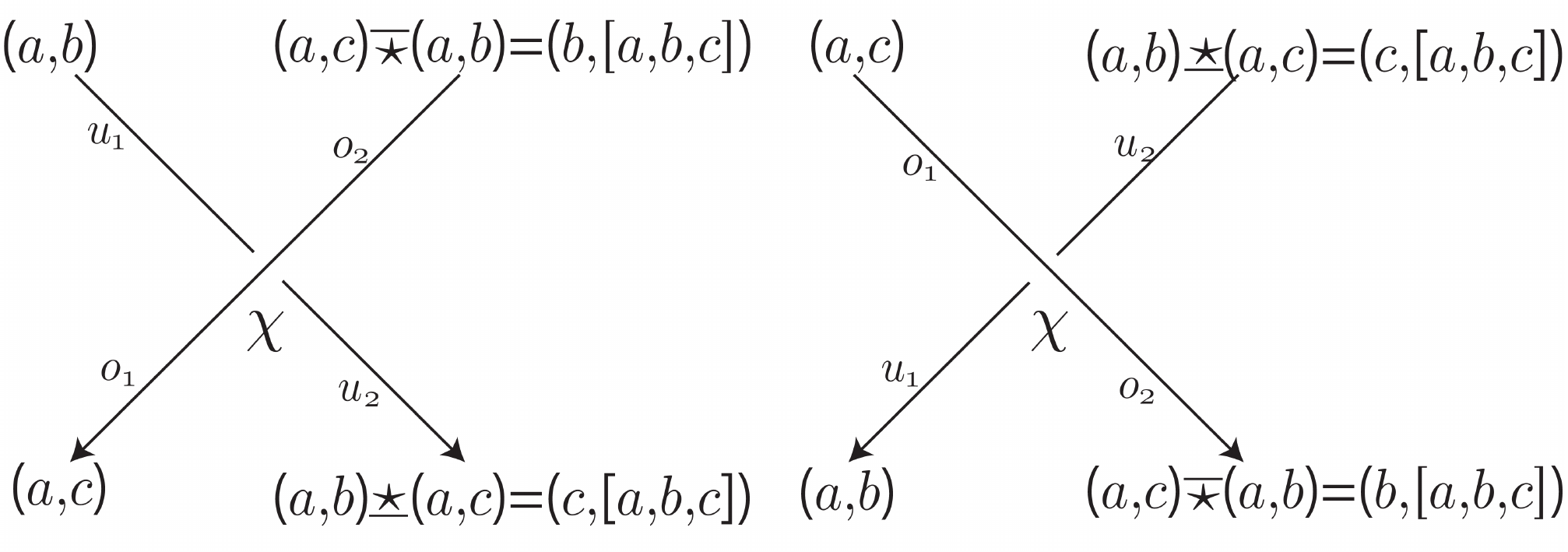}
    \label{coloring1}
  \end{center}
\end{figure}
\end{itemize}
We denote by ${\rm Col}_{X^2}^{\rm SA} (D) $ the set of semi-arc $X^2$-colorings of $D$. We call $C(x)$ for a semi-arc $x$ the {\it color} of $x$. 
We call a pair $(D,C)$ of a diagram $D$ and a semi-arc $X^2$-coloring $C$ of $D$ a {\it semi-arc $X^2$-colored diagram}.
\end{definition}

\begin{remark}\label{Rem:coloringtranslation}
As shown in Figure~\ref{coloring5}, a semi-arc $X^2$-coloring $C$ of a connected diagram $D$ induces a semi-arc $X$-coloring $\widetilde{C}$ of the $2$-parallel $\widetilde{D}$, that is,  a map $\widetilde{C}: \mathcal{SA}(\widetilde{D}) \to X$ is obtained from $C$ by setting $\widetilde{C}(x^{(+)})=a$ and $\widetilde{C}(x^{(-)})=b$  for each semi-arc $x$ of $D$ such that  $C(x)=(a,b)$.  
Moreover, the semi-arc $X$-coloring $\widetilde{C}$ of $\widetilde{D}$ induces a region $X$-coloring $\overline{C}$ of $D$, that is, a map $\overline{C}: \mathcal{R}(D) \to X$ is obtained  from $\widetilde{C}$ by setting $\overline{C}(r) = \widetilde{C}(x^{(\varepsilon)})$ for a region $r\in \mathcal{R}(D)$ and a semi-arc $x^{(\varepsilon)} \in \mathcal{SA}(\widetilde{D})$ that is on the boundary of $r$, where this construction is well-defined only when we use a connected diagram. 
We also note that the induced region coloring satisfies the condition depicted in Figures~\ref{coloring2} and \ref{coloring3}, see also Definition~\ref{def:regioncoloring1}.
Thus we have a map $T: {\rm Col}_{X^2}^{\rm SA} (D) \to {\rm Col}_{X}^{\rm R}(D); C \mapsto \overline{C}$, which we call the {\it translation map  from ${\rm Col}_{X^2}^{\rm SA} (D) $ to ${\rm Col}_{X}^{\rm R}(D)$}, where ${\rm Col}_{X}^{\rm R}(D)$ means the set of region $X$-colorings of $D$ that satisfy the condition depicted in Figures~\ref{coloring2}  and \ref{coloring3}.
For  a semi-arc $X^2$-coloring $C$,  we call $\overline{C}=T(C)$ the {\it corresponding region  $X$-coloring} of $C$ through $T$.  
 
Since the translation map $T$ is invertible, we have the inverse translation map $T^{-1}:  {\rm Col}_{X}^{\rm R}(D) \to {\rm Col}_{X^2}^{\rm SA} (D)$, and 
for a region $X$-coloring $C$, we call $\overline{C}=T^{-1} (C)$ the {\it corresponding semi-arc $X^2$-coloring } of $C$ through $T^{-1}$.   
\begin{figure}
  \begin{center}
    \includegraphics[clip,width=12cm]{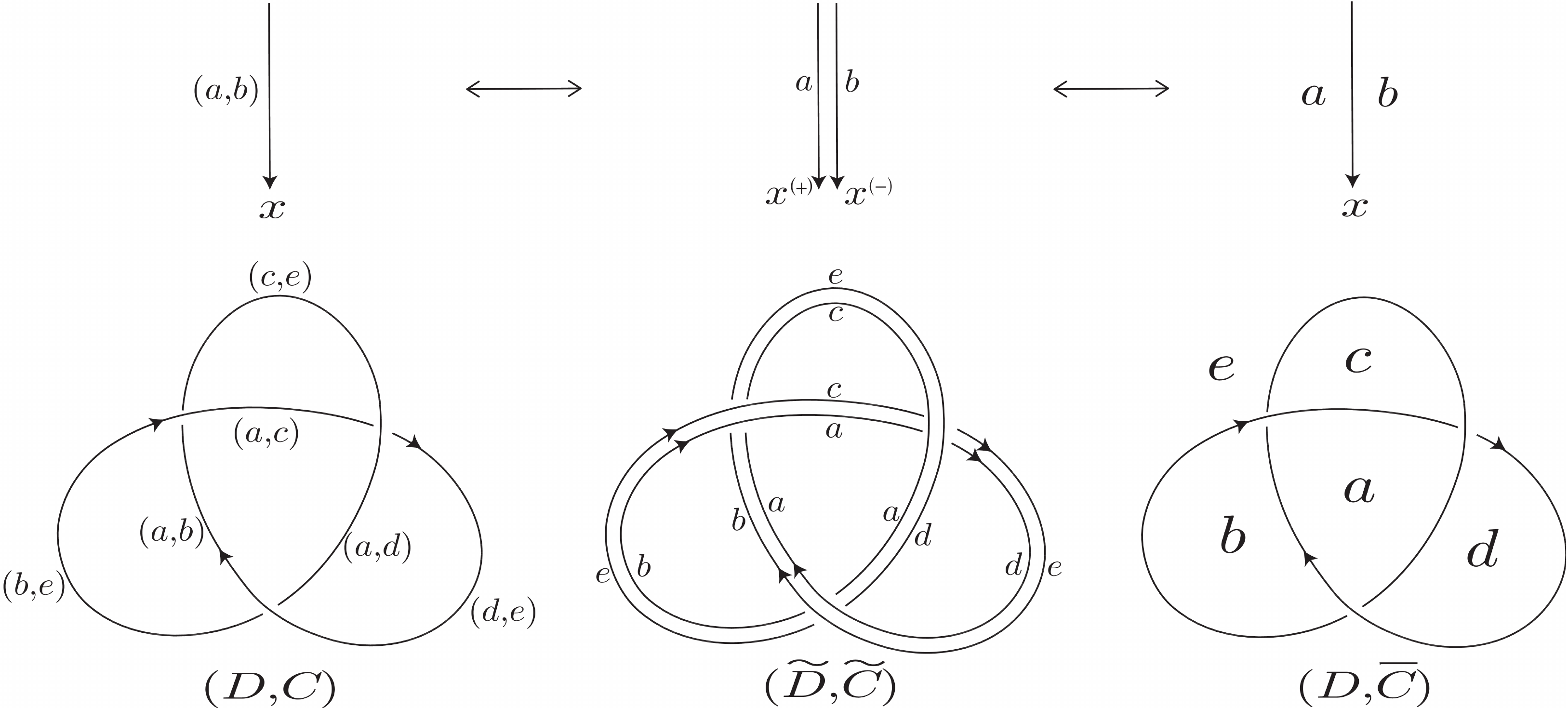}
    \label{coloring5}
  \end{center}
\end{figure}
\end{remark}

\begin{proposition}\label{prop:coloring1}
Let $D$ and $D'$ be connected diagrams of links. 
If $D$ and $D'$ represent the same link, then there exists a bijection between  
${\rm Col}_{X^2}^{\rm SA} (D) $ and ${\rm Col}_{X^2}^{\rm SA} (D')  $.
\end{proposition}
\begin{proof}
Let $D$ and $D'$  be connected diagrams such that $D'$ is obtained from $D$ by a single Reidemeister move. 
Let $E$ be a $2$-disk in which the move is applied. 
Let $C$ be a semi-arc $X^2$-coloring of $D$. 
We define a semi-arc $X^2$-coloring $C'$ of $D'$, corresponding to $C$, by $C'(e) = C(e)$ for a
semi-arc $e$ appearing in the outside of $E$. Then the colors of the semi-arcs appearing in $E$,
by $C'$, are uniquely determined, see Figure~\ref{RmoveII}. Thus we have a bijection ${\rm Col}_{X^2}^{\rm SA} (D) \to {\rm Col}_{X^2}^{\rm SA} (D') $ that maps $C$ to $C'$.
\begin{figure}
  \begin{center}
    \includegraphics[clip,width=12cm]{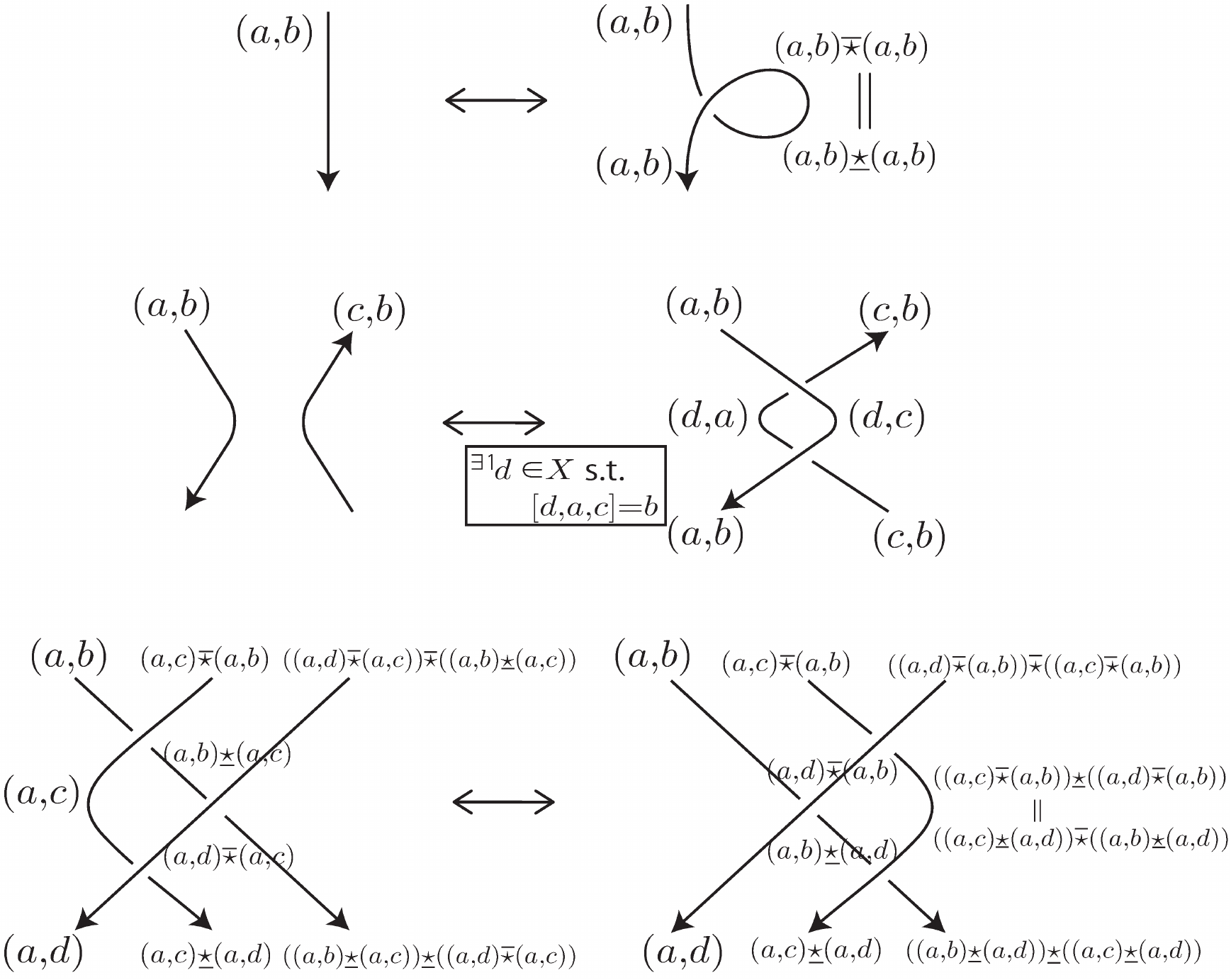}
    \label{RmoveII}
  \end{center}
\end{figure}
\end{proof}

Next, we show how to obtain a cocycle invariant by using the semi-arc $X^2$-colorings of a connected diagram.   

Let $C$ be a semi-arc $X^2$-coloring of $D$.
We define the local chain 
$w(D, C; \chi) \in C^{\rm LB}_2 (X)$ at each crossing $\chi$ by 
$$w(D, C; \chi) ={\rm sign}(\chi) \big((a,b), (a,c)\big)$$ when  $C(u_1)=(a,b)$ and $C(o_1)=(a,c)$, where $u_1$ and $o_1$ are the under-semi-arc and the over-semi-arc  as depicted in Figure~\ref{coloring1}.
We define a chain by 
$$\displaystyle W(D, C)=\sum_{\chi \in \{\mbox{\small crossings of $D$}\}} w(D, C; \chi) \in C^{\rm LB}_2 (X).$$ 

\begin{lemma}
The chain $W(D, C)$ is a  $2$-cycle of $C_\ast ^{\rm LB}(X)$. 
\end{lemma}

\begin{proof}
This can be shown similarly as in the case of biquandles. 
More precisely,  for each positive crossing $\chi$ as depicted in the left of Figure~\ref{coloring1}, we have 
\begin{align*}
\partial_2^{\rm LB} \Big(w(D, C; \chi)\Big) &=\partial_2^{\rm LB} \Big(\big((a, b), (a, c)\big)\Big) \notag\\
&=-(a, c)+(a, c)\oline{\star}(a, b)+(a, b)-(a, b)\uline{\star}(a, c) \notag\\
&=-(a, c)+(b, [a, b, c])+(a, b)-(c, [a, b, c]),
\end{align*}
the positive (resp. negative) terms of which can be assigned to the incoming (resp. outgoing) semi-arcs around $\chi$ so as that for each semi-arc around $\chi$ the color of the semi-arc coincides with  the assigned pair.
For a negative crossing, the same thing holds.
This ensures that the two ends of each semi-arc have the same pair of elements of $X$, but with opposite signs. Thus, we have $\partial_2^{\rm LB}(W(D, C))=0$.
\end{proof}

Let $A$ be an abelian group. For a $2$-cocycle $\theta \in C^2_{\rm LB}(X; A)$, we define
\[
\begin{array}{l}
\mathcal{H}(D)=\{[W(D, C)] \in H^{\rm LB}_2(X) \ | \ C \in {\rm Col}_{X^2}^{\rm SA} (D) \}, and \\[5pt]
\Phi_{\theta}(D)=\{\theta(W(D, C)) \in A \ | \ C \in {\rm Col}_{X^2}^{\rm SA} (D) \}
\end{array}
\]
as multisets. Then we have the following theorem:
\begin{theorem}
$\mathcal{H}(D)$ and $\Phi_{\theta}(D)$ are invariants of $L$.
\end{theorem}

\begin{proof}
It is sufficient to show that $[W(D, C)] \in H^{\rm LB}_2(X)$ is unchanged under each semi-arc $X^2$-colored Reidemeister move. Here we show this property for the case of a Reidemeister move of type III. We leave the proof of the other cases to the reader.

\begin{figure}
  \begin{center}
    \includegraphics[clip,width=11cm]{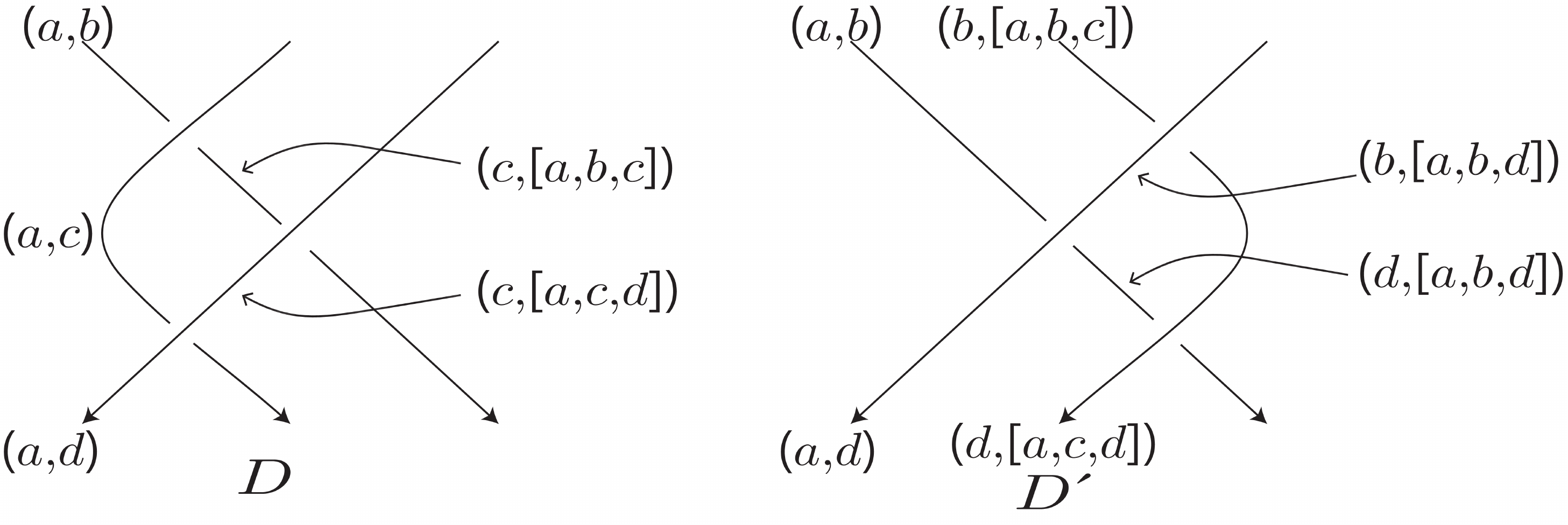}
    \label{RmoveIII2}
  \end{center}
\end{figure}
Let $(D, C)$ and $(D', C')$ be semi-arc $X^2$-colored, connected diagrams of $L$ that differ by the Reidemeister move of type III shown in Figure~\ref{RmoveIII2}. We then have
\begin{eqnarray*}
&&W(D, C) \\
&&=\big((a, b), (a, c)\big)+\big((a, c), (a, d)\big)+\big((c, [a, b, c]), (c, [a, c, d])\big)+\mathcal{C}\\
&&=\big((a, b), (a, c)\big)+\big((a, c), (a, d)\big)+\big((a, b)\uline{\star}(a, c), (a, d)\oline{\star}(a, c)\big)+\mathcal{C},
\end{eqnarray*}
\begin{eqnarray*}
&&W(D', C')\\
&&=\big((a, b), (a, d)\big)+\big((b, [a, b, c]), (b, [a, b, d])\big)+\big((d, [a, b, d]), (d, [a, c, d])\big)+\mathcal{C}\\
&&=\big((a, b), (a, d)\big)+\big((a, c)\oline{\star}(a, b), (a, d)\oline{\star}(a, b)\big)+\big((a, b)\uline{\star}(a, d), (a, c)\uline{\star}(a, d)\big)+\mathcal{C},
\end{eqnarray*}
for some chain $\mathcal{C}$ in $C_2^{\rm LB}(X)$. On the other hand,
\begin{eqnarray*}
\partial_3^{\rm LB}\Big(\big((a, b), (a, c), (a, d)\big)\Big)&=&-\big((a, c), (a, d)\big)+\big((a, c)\oline{\star}(a, b), (a, d)\oline{\star}(a, b)\big)\\
&&+\big((a, b), (a, d)\big)-\big((a, b)\uline{\star}(a, c), (a, d)\oline{\star}(a, c)\big)\\
&&-\big((a, b), (a, c)\big)+\big((a, b)\uline{\star}(a, d), (a, c)\uline{\star}(a, d)\big),
\end{eqnarray*}
which implies that 
$$W(D', C')-W(D, C)=\partial_3^{\rm LB}\Big(\big((a, b), (a, c), (a, d)\big)\Big) \in \rm{Im} \, \partial_3^{\rm LB}.$$ Therefore we have $[W(D, C)]=[W(D', C')] \in H^{\rm LB}_2(X)$.
\end{proof}

\begin{proposition}
For cohomologous $2$-cocycles $\theta, \theta' \in C^2_{\rm LB}(X; A)$, we have $\Phi_\theta(D)=\Phi_{\theta^{'}}(D)$.
\end{proposition}

\begin{proof}
Since $\theta$ and $\theta'$ are cohomologous, we have $\theta - \theta' \in {\rm Im}  \ \delta^1$, and hence, there exists  a homomorphism $\phi : C_1^{\rm LB}(X) \to A$ such that $\theta - \theta'=\delta^1_{\rm LB}(\phi )=\phi \circ \partial_2^{\rm LB}$. Then for each semi-arc $X^2$-coloring $C$ of $D$, we have
\begin{align*}
\theta(W(D, C)) - \theta'(W(D, C))&=(\theta - \theta')(W(D, C))\\
&=(\phi \circ \partial_2^{\rm LB}) (W(D, C))\\
&=\phi ( \partial_2^{\rm LB} (W(D, C))\\
&=0.
\end{align*}
Thus we have  $\Phi_\theta(D)=\Phi_{\theta^{'}}(D)$.
\end{proof}

\subsection{Semi-sheet colorings of surface-link diagrams, cocycle invariants}\label{subsection:surfaceinvariant}
For surface-links, we can also define semi-sheet $X^2$-colorings for diagrams and cocycle invariants. 
In this subsection, we briefly show these definitions and similar  properties.

Let $(X, \{\uline{\star}\} , \{\oline{\star}\})$ be a local biquandle, and $(X,[\,])$ the corresponding knot-theoretic horizontal-ternary-quasigroup of $(X, \{\uline{\star}\} , \{\oline{\star}\})$. Let $D$ be a connected diagram of a surface-link.
\begin{definition}
A \textit{semi-sheet $X^2$-coloring} of $D$ is a map $C: \mathcal{SS}(D) \to X^2$ satisfying the following condition:
\begin{itemize}
\item For a double point curve composed of under-semi-sheets $u_1, u_2$ and over-semi-sheets $o_1, o_2$ as depicted in Figure~\ref{doublepoint1}, let $C(u_1)=(a_1,b), C(o_1)=(a_2, c)$. Then 
\begin{itemize}
\item $a_1=a_2$,  
\item $C(u_2) = C(u_1) \uline{\star} C(o_1)= (a,b) \uline{\star} (a,c) = (c, [a,b,c])$, and 
\item $C(o_2) = C(o_1) \oline{\star} C(u_1)=(a,c) \oline{\star} (a,b) = (b, [a,b,c])$
\end{itemize}
hold,  where $a= a_1 =a_2$, see Figure~\ref{doublepoint1}.
\begin{figure}
  \begin{center}
    \includegraphics[clip,width=5.0cm]{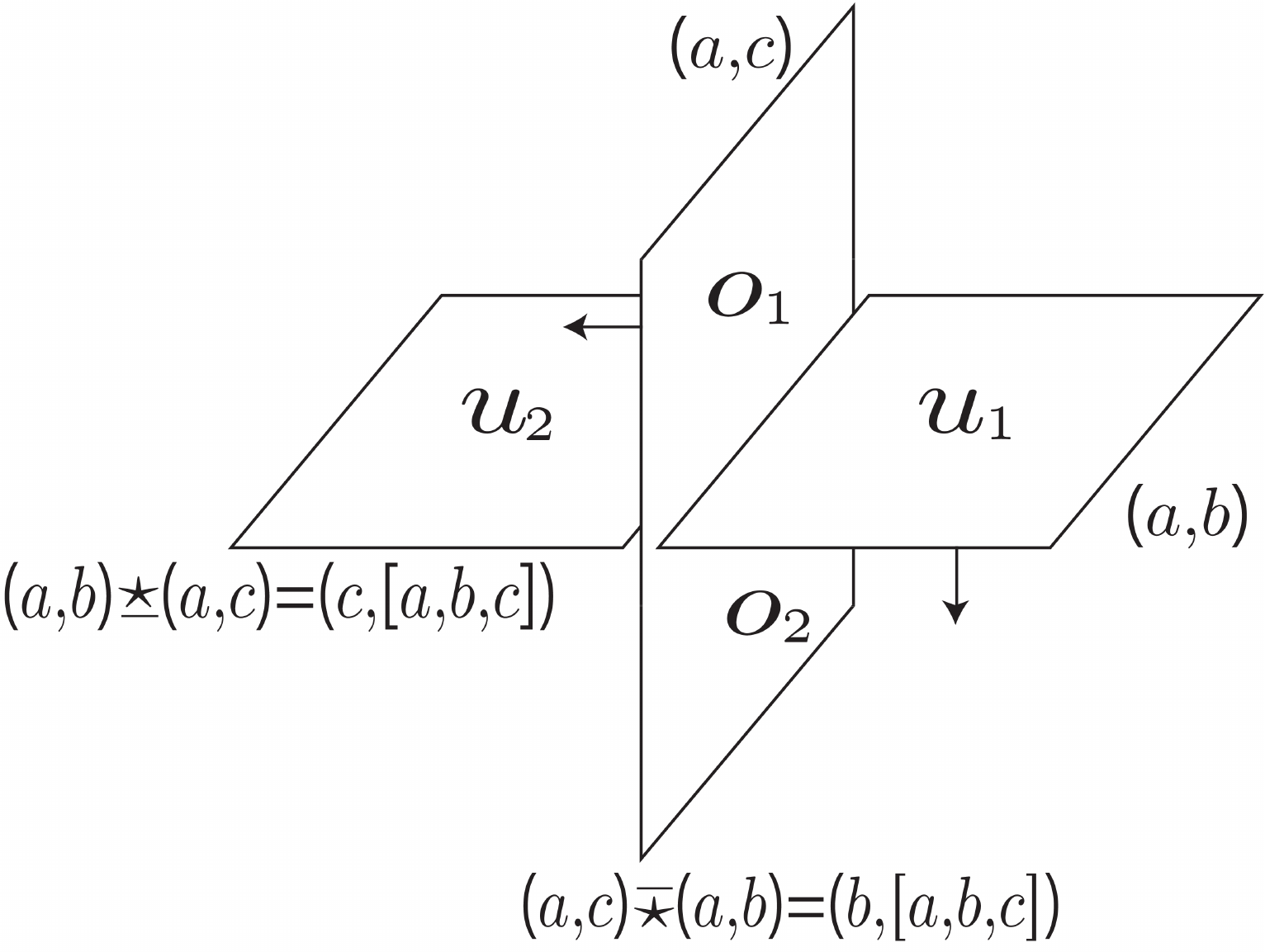}
    \label{doublepoint1}
  \end{center}
\end{figure}
\end{itemize}
We denote by ${\rm Col}_{X^2}^{\rm SS} (D) $ the set of semi-sheet $X^2$-colorings of $D$. We call $C(x)$ for a semi-sheet $x$ the {\it color} of $x$. We call a pair $(D,C)$ of a diagram $D$ and a semi-sheet $X^2$-coloring $C$ of $D$ a {\it semi-sheet $X^2$-colored diagram}.
\end{definition}
\begin{remark}\label{Rem:coloringtranslation2}
As shown in Remark~\ref{Rem:coloringtranslation} and Figure~\ref{doublepoint2},  we have the  bijective {\it translation maps}  $T: {\rm Col}_{X^2}^{\rm SS} (D) \to {\rm Col}_{X}^{\rm R}(D); C \mapsto \overline{C}$ and $T^{-1}: {\rm Col}_{X}^{\rm R}(D) \to {\rm Col}_{X^2}^{\rm SS} (D); \overline{C}\mapsto C $, where ${\rm Col}_{X}^{\rm R}(D)$ means the set of region colorings of $D$ that satisfy the condition depicted in the right of Figure~\ref{doublepoint2} and Figure~\ref{doublepoint3}.

For a semi-sheet $X^2$-coloring $C$, we call $\overline{C}=T (C)$ the {\it corresponding region $X$-coloring } of $C$ through $T$.   
For a region $X$-coloring $C$, we call $\overline{C}=T^{-1} (C)$ the {\it corresponding semi-sheet $X^2$-coloring } of $C$ through $T^{-1}$.   

\begin{figure}
  \begin{center}
    \includegraphics[clip,width=9.0cm]{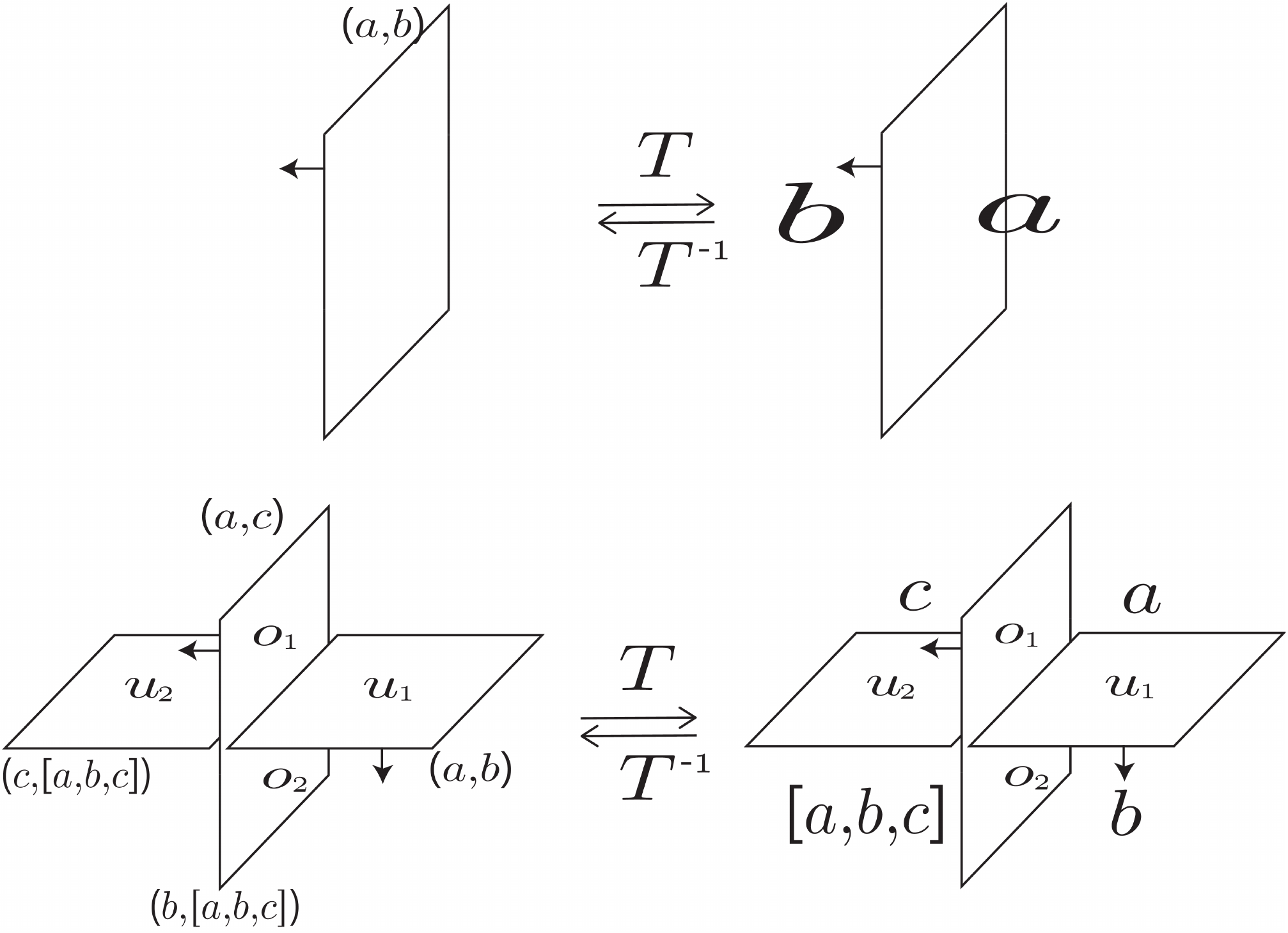}
    \label{doublepoint2}
  \end{center}
\end{figure}
\end{remark}

\begin{proposition}\label{prop:coloring2}
Let $D$ and $D'$ be connected diagrams of surface-links. 
If $D$ and $D'$ represent the same surface-link, then there exists a bijection between  
${\rm Col}_{X^2}^{\rm SS} (D) $ and ${\rm Col}_{X^2}^{\rm SS} (D')  $.
\end{proposition}


Next, we show how to obtain a cocycle invariant by using semi-sheet $X^2$-colorings of a connected diagram.   

Let $C$ be a semi-sheet $X^2$-coloring of $D$.
We define the local chain 
$w(D, C; \tau) \in C^{\rm LB}_3 (X)$ at each triple point $\tau$ by 
$$w(D, C; \tau) ={\rm sign}(\tau) \big((a,b), (a,c), (a,d)\big)$$ when  $C(b_1)=(a,b), C(m_1)=(a,c)$ and $C(t_1)=(a,d)$, where $b_1$, $m_1$ and $t_1$ are the bottom-, middle- and top-semi-sheet  as depicted in Figure~\ref{triplepoint1}.
\begin{figure}
  \begin{center}
    \includegraphics[clip,width=10cm]{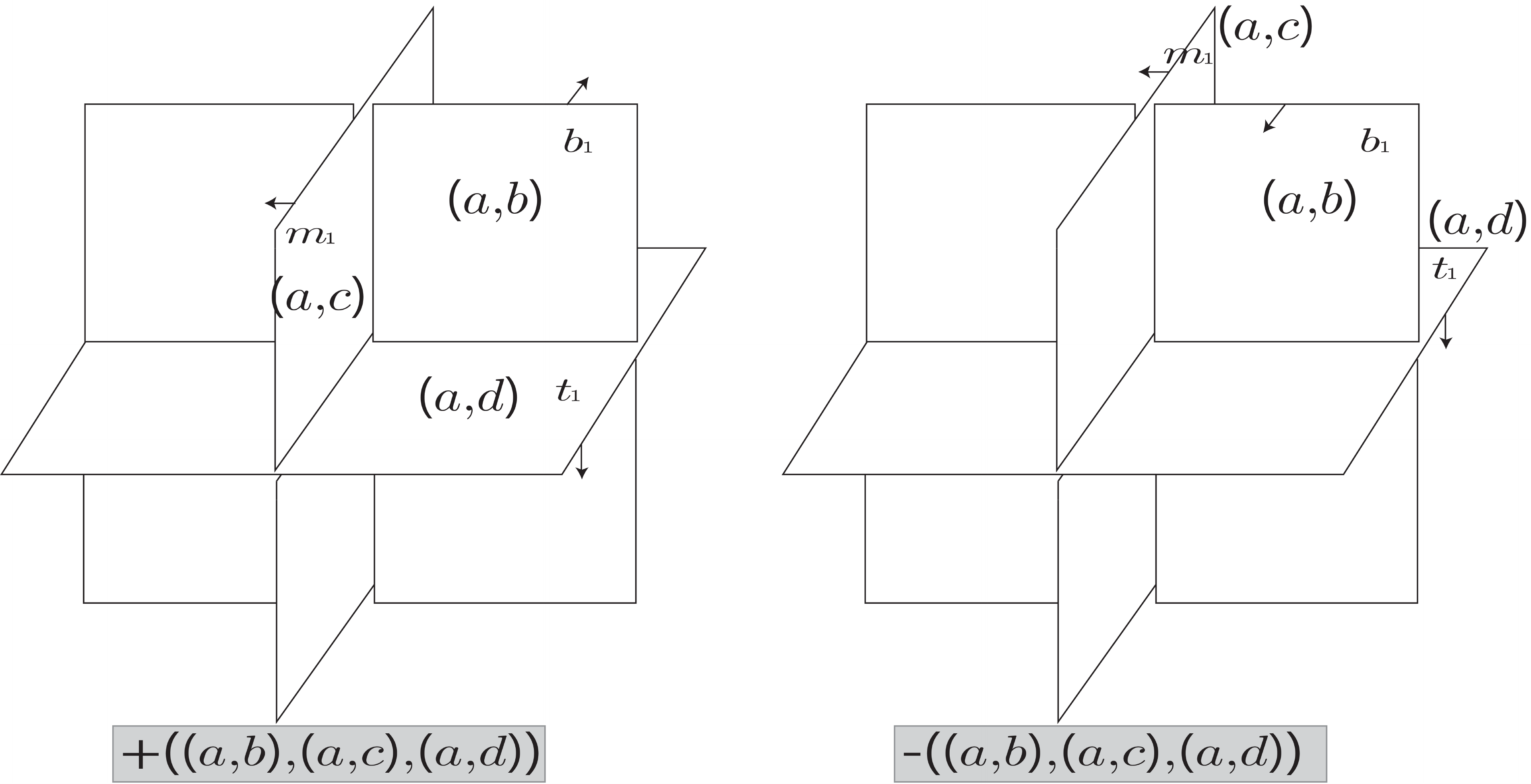}
    \label{triplepoint1}
  \end{center}
\end{figure}
Let  $A$ be an abelian group. We define a chain by 
$$\displaystyle W(D, C)=\sum_{\tau \in \{\mbox{\small triple points of $D$}\}} w(D, C; \tau) \in C^{\rm LB}_3 (X).$$ 

\begin{lemma}
The chain $W(D, C)$ is a  $3$-cycle of $C_* ^{\rm LB}(X)$. 
\end{lemma}

For a $3$-cocycle $\theta \in C^3_{\rm LB}(X;A)$, we define
\[
\begin{array}{l}
\mathcal{H}(D)=\{[W(D, C)] \in H^{\rm LB}_3(X) \ | \ C \in {\rm Col}_{X^2}^{\rm SS} (D) \}, and \\[5pt]
\Phi_{\theta}(D)=\{\theta(W(D, C)) \in A \ | \ C \in {\rm Col}_{X^2}^{\rm SS} (D) \}
\end{array}
\]
as multisets. Then we have the following theorem:
\begin{theorem}
$\mathcal{H}(D)$ and $\Phi_{\theta}(D)$ are invariants of $F$.
\end{theorem}

\begin{proposition}
For cohomologous $3$-cocycles $\theta, \theta' \in C^3_{\rm LB}(X; A)$, we have $\Phi_\theta(D)=\Phi_{\theta^{'}}(D)$.
\end{proposition}

\section{Niebrzydowski's work}
\label{Niebrzydowski's work}
In \cite{Niebrzydowski1} (see also \cite{Niebrzydowski2}), Niebrzydowski studied  knot-theoretic vertical-tribracket $\langle\, \rangle$ (with some conditions)  to give a region coloring  of a link diagram, and he gave a homology theory  for  knot-theoretic ternary quasigroups. 
In this section, we review his idea with our notation.

\subsection{Homology groups}\label{subsec:NHom}
Let $(X, \langle\, \rangle )$ is a knot-theoretic vertical-ternary-quasigroup.
Let $n \in \mathbb Z$.
Let $C_n^{\rm Nie}(X)$ be the free $\mathbb Z$-module generated by the elements of $X^{n+2}$
if $n\geq 0$, and $C_n^{\rm Nie}(X)=0$ otherwise.
We define a homomorphism $\partial_n^{\rm Nie} : C_n^{\rm Nie} (X) \to C_{n-1}^{\rm Nie} (X)$ by 
\begin{align}
&\partial_n^{\rm Nie} \big( ( a_0, a_1, \ldots , a_{n+1} ) \big)  \notag \\
&= \sum_{i=0}^{n} (-1)^i \Big\{ \big( y_{(i,1)}, y_{(i,2)}, \ldots , y_{(i,i)}, a_{i+1}, a_{i+2}, \ldots , a_{n+1} \big)\notag  \\
&\hspace{2cm}- \big( a_0, a_1, \ldots , a_i, y_{(i,i+1)},y_{(i,i+2)},\ldots , y_{(i,n)} \big) \Big\} \notag 
\end{align}
if $n> 0$, and $\partial_n^{\rm Nie}=0$ otherwise, where 
\[
y_{(i,j)} = \left\{
\begin{array}{ll}
\langle a_{j-1},  a_j,  y_{(i,j+1)}\rangle  & (j<i),\\[5pt]
\langle a_{j-1},  a_j,  a_{j+1} \rangle & (j=i, i+1),\\[5pt]
\langle y_{(i,j-1)},  a_j,  a_{j+1} \rangle & (j>i+1).
\end{array}
\right.
\]
\begin{example} For $a,b,c,d \in X$, 
\[
\begin{array}{rcl}
\partial_1^{\rm Nie}\big((a,b,c)\big) &=& (-1)^0\{ (b,c)-(a, \langle a,b,c\rangle)\}+(-1)^1\{ (\langle a,b,c\rangle,c)-(a, b)\},\\[3pt]
&=&(b,c)-(a, \langle a,b,c\rangle)-(\langle a,b,c\rangle,c)+(a, b), \\[3pt]
\partial_2^{\rm Nie}\big((a,b,c,d)\big) &=& (-1)^0\{ (b,c,d)-(a, \langle a,b,c\rangle, \langle \langle a,b,c\rangle ,c,d \rangle)\}\\[3pt]
&&+(-1)^1\{ (\langle a,b,c\rangle,c,d)-(a, b, \langle b,c,d\rangle )\}\\[3pt]
&&+(-1)^2\{(\langle a,b , \langle b,c,d \rangle \rangle, \langle b,c,d \rangle, d)-(a,b,c)\} \\[3pt]
&=& (b,c,d)-(a, \langle a,b,c\rangle, \langle \langle a,b,c\rangle ,c,d \rangle)- (\langle a,b,c\rangle,c,d)\\[3pt]
&&+(a, b, \langle b,c,d\rangle )+(\langle a,b , \langle b,c,d \rangle \rangle, \langle b,c,d \rangle, d)-(a,b,c).\\ 
\end{array}
\]
\end{example}
\begin{remark}
We modified the definitions of  the chain group $C_{-1}^{\rm Nie}(X)$ and the boundary map $\partial_0^{\rm Nie}$, while Niebrzydowski defined $C_0^{\rm Nie}(X)$ to be the free $\mathbb Z$-module generated by the elements of $X$, and 
$\partial_0^{\rm Nie}$ by $\partial_0^{\rm Nie}\big((a_0,a_1)\big) = (a_1)-(a_0)$. This does not affect the $2$- or $3$-cocycle conditions and his cocycle invariants defined in Subsections~\ref{subsec: Region colorings of link diagrams, cocycle invariants} and  \ref{subsec: Region colorings of surface-link diagrams, cocycle invariants}.
\end{remark}
\begin{lemma}{\rm (\cite{Niebrzydowski1})}
$C_*^{\rm Nie}(X)=\{C_n^{\rm Nie}(X), \partial_n^{\rm Nie}\}_{n\in \mathbb Z}$ is a chain complex, i.e., for any $n \in \mathbb Z$, $\partial_n^{\rm Nie}\circ \partial_{n+1}^{\rm Nie} =0$ holds.
\end{lemma}
Let $D_n^{\rm Nie}(X)$ be the submodule of $C_n^{\rm Nie}(X)$ that is generated by the elements of 
\[
\big\{ (a_0, a_1,\ldots , a_{n+1}) \in X^{n+2} ~|~ \mbox{ $\langle a_{j-1},  a_j,  a_{j+1} \rangle = a_j$ for some $j\in \{1, \ldots , n \} $  }  \big\}.
\]
\begin{lemma}{\rm (\cite{Niebrzydowski1})} \label{lemma:degenerate(N)}
$D_*^{\rm Nie}(X)=\{D_n^{\rm Nie}(X), \partial_n^{\rm Nie}\}_{n\in \mathbb Z}$ is a subchain complex of  $C_*^{\rm Nie}(X)$, i.e., for any $n \in \mathbb Z$, $\partial_n (D_n^{\rm Nie}(X)) \subset D_{n-1}^{\rm Nie} (X)$ holds.
\end{lemma}
\noindent Therefore the  chain complex 
$$C_*^{\rm N} (X)=\{C_n^{\rm N}(X):=C_n^{\rm Nie}(X)/D_n^{\rm Nie}(X), \partial_n^{\rm N} := \partial_n^{\rm Nie} \}_{n\in \mathbb Z}$$
is induced.
The homology group $H_n^{\rm N}  (X)$ of $C_*^{\rm N}  (X)$  is called the \textit{$n$th knot-theoretic ternary quasigroup homology} of $(X, \langle\, \rangle )$.

For an abelian group $A$, we define the chain and cochain complexes by 
\[
\begin{array}{l}
C_n^{\rm N} (X; A)=C_n^{\rm N}(X) \otimes A, \quad \partial_n^{\rm N} \otimes {\rm id}  \mbox{ and }\\[5pt]
C_{\rm N}^n(X; A) ={\rm Hom}(C_n^{\rm N}(X); A), \quad \delta^n_{\rm N} \mbox{ s.t. }\delta^n_{\rm N}(f)=f \circ \partial_{n+1}^{\rm N}.
\end{array}
\]
Let $C_\ast^{\rm N}(X; A)=\{C_n^{\rm N}(X; A), \partial_n^{\rm N}\otimes {\rm id}\}_{n\in \mathbb Z}$ and $C_{\rm N}^\ast(X; A)=\{C_{\rm N}^n(X; A), \delta^n_{\rm N}\}_{n\in \mathbb Z}$. 
The $n$th homology group $H_n^{\rm N}(X; A)$ and $n$th cohomology group $H^n_{\rm N}(X; A)$ of  $(X, \langle\, \rangle )$  with coefficient group $A$ are defined by
\[
H_n^{\rm N}(X; A)=H_n(C_\ast^{\rm N}(X; A)) \qquad {\rm and} \qquad H_{\rm N}^n(X; A)=H^n(C^\ast_{\rm N}(X; A)).
\]
Note that we omit the coefficient group  $A$ if $A=\mathbb Z$ as usual.

\subsection{Region colorings of link diagrams, cocycle invariants}\label{subsec: Region colorings of link diagrams, cocycle invariants}

Let $(X, \langle\, \rangle )$ be a knot-theoretic vertical-ternary-quasigroup.
Let $D$ be a diagram of a link.
\begin{definition}\label{def:regioncoloring1}
A \textit{region $X$-coloring} of $D$ is a map $C: \mathcal{R}(D) \to X$ satisfying the following condition:
\begin{itemize}
\item For a crossing with regions $r_1$, $r_2$, $r_3$ and $r_4$ near it as depicted in Figure~\ref{coloring3}, let $C(r_1)=a$, $C(r_2)=b$ and $C(r_3)= c$. Then $C(r_4) = \langle a,b,c \rangle $ holds, see Figure~\ref{coloring3}.
\end{itemize}
We denote by ${\rm Col}_{X}^{\rm R} (D) $ the set of region $X$-colorings of $D$. 
We call $C(r)$ for a  region $r$ the {\it color} of $r$. 
We call a pair $(D,C)$ of a diagram $D$ and a region $X$-coloring $C$ of $D$ a {\it region $X$-colored diagram}.
\end{definition}
\begin{proposition}{\rm (\cite{Niebrzydowski1})}
Let $D$ and $D'$ be diagrams of links. 
If $D$ and $D'$ represent the same links, then there exists a bijection between  
${\rm Col}_{X}^{\rm R} (D) $ and ${\rm Col}_{X}^{\rm R} (D')  $.
\end{proposition}

Next, we show how to obtain a cocycle invariant by using region $X$-colorings of a diagram.   

Let $C$ be a region $X$-coloring of $D$.
We define the local chain 
$w^{\rm N}(D, C; \chi) \in C^{\rm N}_1 (X)$ at each crossing $\chi$ by 
$$w^{\rm N}(D, C; \chi) ={\rm sign}(\chi) \big(a,b,c\big)$$ when  $C(r_1)=a, C(r_2)=b$ and $C(r_3)=c$, where $r_1$, $r_2$ and $r_3$ are the regions near $\chi$  as depicted in Figure~\ref{coloring3}.
We define a chain by 
$$\displaystyle W^{\rm N}(D, C)=\sum_{\chi \in \{\mbox{\small crossings of $D$}\}} w^{\rm N}(D, C; \chi) \in C^{\rm N}_1 (X).$$ 

\begin{lemma}{\rm (\cite{Niebrzydowski1})}
The chain $W^{\rm N}(D, C)$ is a $1$-cycle of $C_* ^{\rm N}(X)$. 
\end{lemma}

Let $A$ be  an abelian group. For a $1$-cocycle $\theta \in C^1_{\rm N}(X; A)$, we define
\[
\begin{array}{l}
\mathcal{H}^{\rm N}(D)=\{[W^{\rm N}(D, C)] \in H^{\rm N}_1(X) \ | \ C \in {\rm Col}_{X}^{\rm R} (D) \}, and \\[5pt]
\Phi_{\theta}^{\rm N}(D)=\{\theta(W^{\rm N}(D, C)) \in A \ | \ C \in {\rm Col}_{X}^{\rm R} (D) \}
\end{array}
\]
as multisets. 
\begin{theorem}{\rm (\cite{Niebrzydowski1})}
$\mathcal{H}^{\rm N}(D)$ and $\Phi_{\theta}^{\rm N}(D)$ are invariants of $L$.
\end{theorem}

\subsection{Region colorings of surface-link diagrams, cocycle invariants}\label{subsec: Region colorings of surface-link diagrams, cocycle invariants}

Let $(X, \langle\, \rangle )$ be a knot-theoretic vertical-ternary-quasigroup.
Let $D$ be a diagram of a surface-link.
\begin{definition}\label{def:regioncoloring2}
A \textit{region $X$-coloring} of $D$ is a map $C: \mathcal{R}(D) \to X$ satisfying the following condition:
\begin{itemize}
\item For a double point curve with regions $r_1$, $r_2$, $r_3$ and $r_4$ near it as depicted in Figure~\ref{doublepoint3}, let $C(r_1)=a$, $C(r_2)=b$ and $C(r_3)= c$. Then $C(r_4) = \langle a,b,c \rangle $ holds, see Figure~\ref{doublepoint3}.
\begin{figure}
  \begin{center}
    \includegraphics[clip,width=5.0cm]{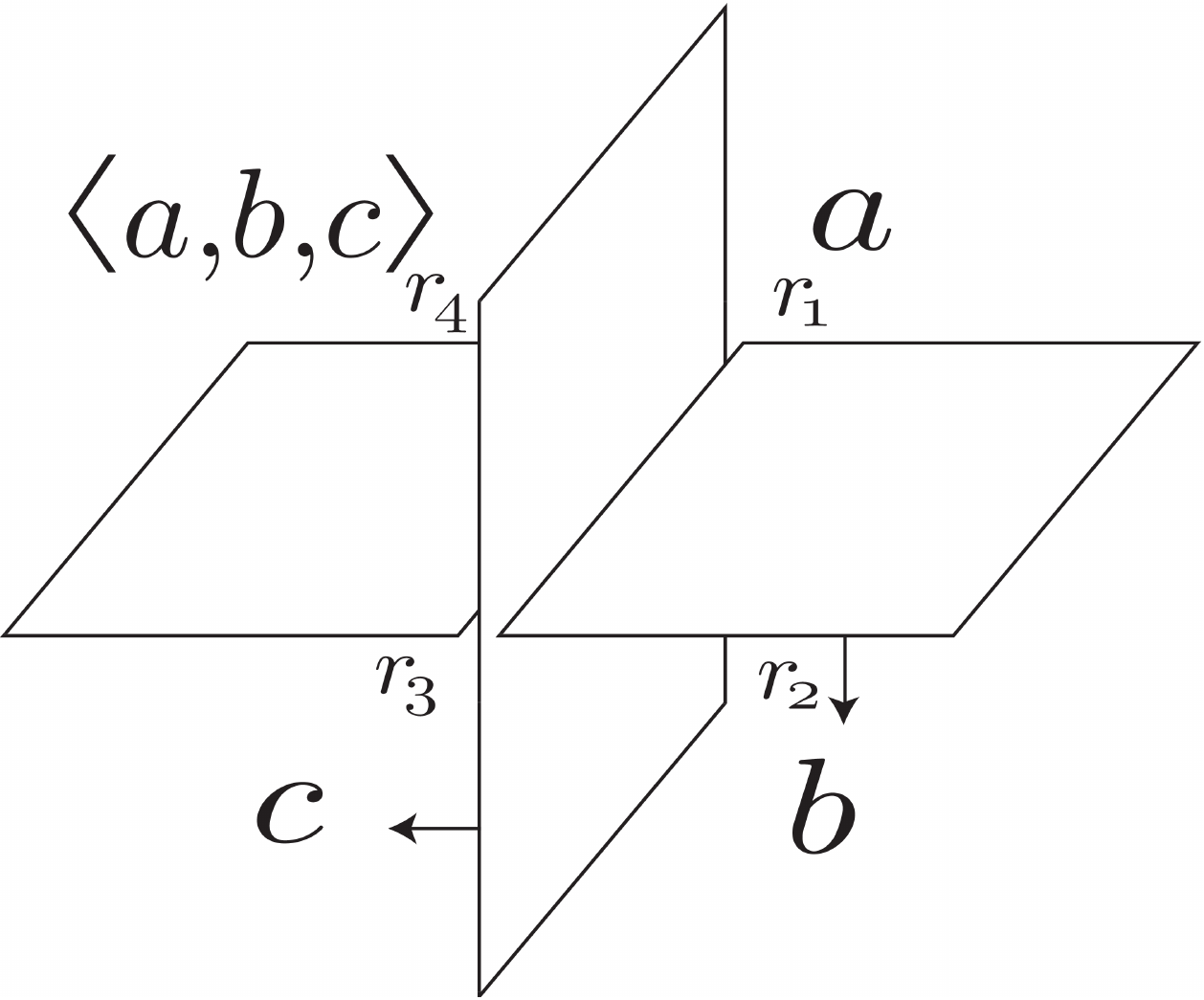}
    \label{doublepoint3}
  \end{center}
\end{figure}
\end{itemize}
We denote by ${\rm Col}_{X}^{\rm R} (D) $ the set of region $X$-colorings of $D$. 
We call $C(r)$ for a  region $r$ the {\it color} of $r$. 
We call a pair $(D,C)$ of a diagram $D$ and a region $X$-coloring $C$ of $D$ a {\it region $X$-colored diagram}.
\end{definition}
\begin{proposition}{\rm (\cite{Niebrzydowski1})}
Let $D$ and $D'$ be diagrams of surface-links. 
If $D$ and $D'$ represent the same surface-link, then there exists a bijection between  
${\rm Col}_{X}^{\rm R} (D) $ and ${\rm Col}_{X}^{\rm R} (D')  $.
\end{proposition}


Next, we show how to obtain a cocycle invariant by using region $X$-colorings of a diagram.   

Let $C$ be a region $X$-coloring of $D$.
We define the local chain 
$w^{\rm N}(D, C; \tau) \in C^{\rm N}_2 (X)$ at each triple point $\tau$ by 
$$w^{\rm N}(D, C; \tau) ={\rm sign}(\tau) \big(a,b,c,d\big)$$ when  $C(r_1)=a, C(r_2)=b$, $C(r_3)=c$ and $C(r_4)=d$, where $r_1$, $r_2$, $r_3$ and $r_4$ are the regions near $\tau$  as depicted in Figure~\ref{triplepoint2}.
\begin{figure}
  \begin{center}
    \includegraphics[clip,width=10cm]{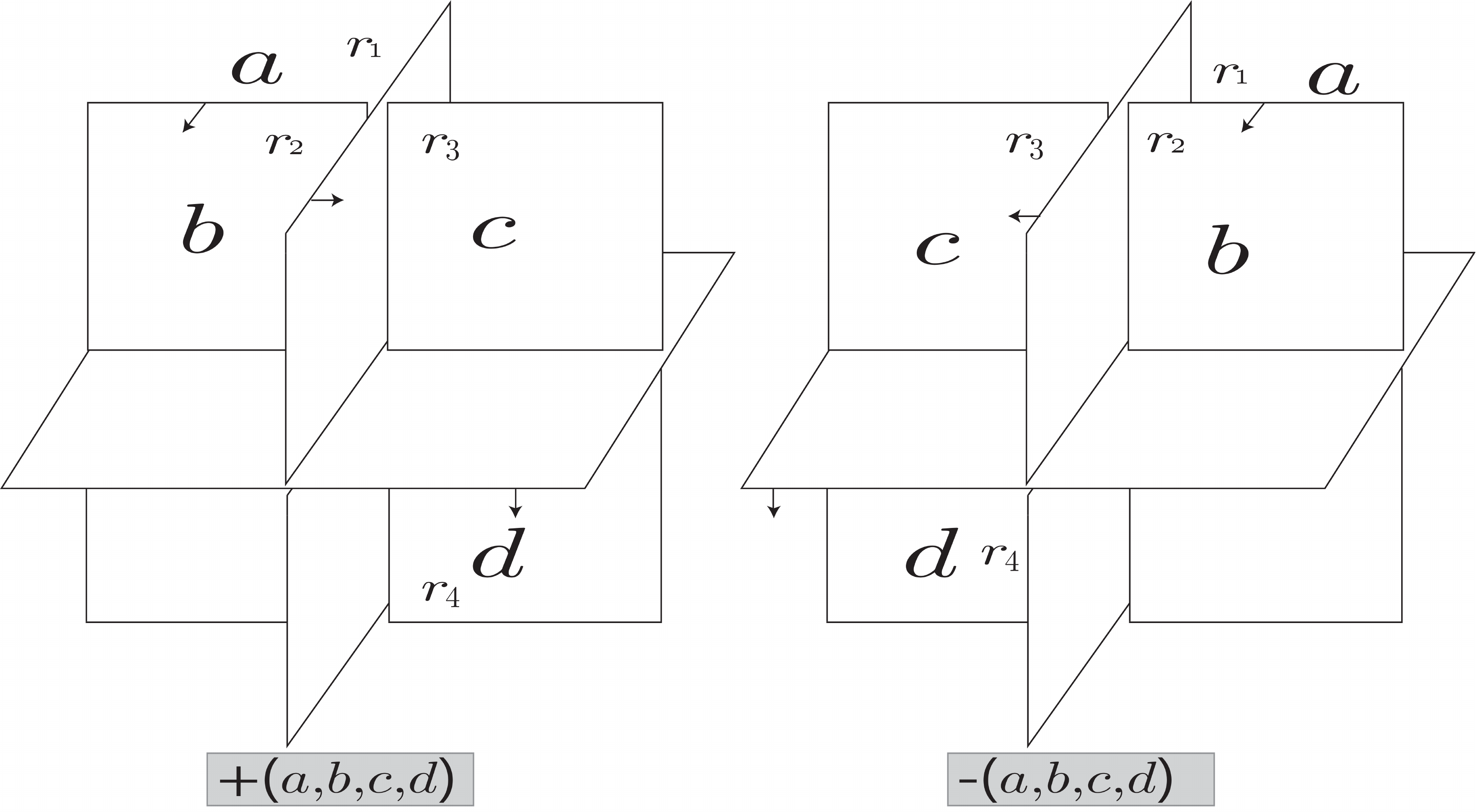}
    \label{triplepoint2}
  \end{center}
\end{figure}
We define a chain by 
$$\displaystyle W^{\rm N}(D, C)=\sum_{\tau \in \{\mbox{\small triple points of $D$}\}} w^{\rm N}(D, C; \tau) \in C^{\rm N}_2 (X).$$ 

\begin{lemma}{\rm (\cite{Niebrzydowski1})}
The chain $W^{\rm N}(D, C)$ is a $2$-cycle of $C_* ^{\rm N}(X)$. 
\end{lemma}

Let  $A$ be  an abelian group. For a $2$-cocycle $\theta \in C^2_{\rm N}(X; A)$, we define
\[
\begin{array}{l}
\mathcal{H}^{\rm N}(D)=\{[W^{\rm N}(D, C)] \in H^{\rm N}_2(X) \ | \ C \in {\rm Col}_{X}^{\rm R} (D) \}, and \\[5pt]
\Phi_{\theta}^{\rm N}(D)=\{\theta(W^{\rm N}(D, C)) \in A \ | \ C \in {\rm Col}_{X}^{\rm R} (D) \}
\end{array}
\]
as multisets. 
\begin{theorem}{\rm (\cite{Niebrzydowski1})}
$\mathcal{H}^{\rm N}(D)$ and $\Phi_{\theta}^{\rm N}(D)$ are invariants of $F$.
\end{theorem}

\section{Correspondence between our work and Niebrzydowski's work}
\label{Correspondence between our work and Niebrzydowski's work}
\subsection{Correspondence between (co)homology groups}\label{subsec:correspondence1}
From now on, $x|_{a\mapsto b}$ means the element $x$ after replacing $a$ in $x$ with $b$.

Let $X$ be  a   set,  $[\, ]$ a horizontal-tribracket on $X$, and $\langle\, \rangle$ the corresponding vertical-tribracket of $[\, ]$. 

Define  a homomorphism 
$\varphi_n: C_n^{\rm LB}(X) \to C_{n-1}^{\rm N}(X)$ by 
\[
\varphi_n \Big(\big((a,b_1), (a, b_2), \ldots , (a ,b_n) \big)\Big) = (z_0, z_1, \ldots , z_n)
\]
if $n\geq 1$, and $\varphi_n=0$ otherwise, where  
\[
\left\{
\begin{array}{ll}
z_0=a,\\
z_1=b_1,\\
z_i=[z_{i-2}, z_{i-1}, z_{i-1}|_{b_{i-1} \mapsto b_i}] & (i\in\{2,\ldots, n \}).
\end{array}
\right.
\]
Define a homomorphism  
$\psi_n:  C_{n-1}^{\rm N}(X) \to C_{n}^{\rm LB}(X)$ by
\[
\psi_n \big((a, a_1, a_2, \ldots , a_n)\big) = \big((a,w_1), (a,w_2), \ldots , (a, w_n)\big)
\]
if $n\geq 1$, and $\psi_n=0$ otherwise, where  
\[
\left\{
\begin{array}{ll}
w_1=a_1,\\
w_2=\langle a,a_1,a_2\rangle,\\
w_i= w_{i-1} |_{a_{i-1} \mapsto \langle a_{i-2}, a_{i-1}, a_i \rangle}& (i\in \{3,\ldots , n\}).
\end{array}
\right.
\]
\begin{example} For $a,b,c,d \in X$, 
\[
\begin{array}{rcl}
\varphi_2 \Big(\big((a,b), (a,c) \big)\Big) &=& \big(a,b, [a,b,c]\big), \\
\varphi_3 \Big(\big((a,b), (a,c), (a,d) \big)\Big) &=& \big(a,b, [a,b,c], [b,[a,b,c],[a,b,d]]\big),\\
\psi_2 \Big(\big(a,b,c \big)\Big) &=& \big((a,b), (a, \langle a,b,c\rangle)\big),\\
\psi_3 \Big(\big(a,b,c ,d \big)\Big) &=&\big((a,b), (a, \langle a,b,c\rangle), (a, \langle a,b,\langle b,c,d\rangle \rangle)\big). \\
\end{array}
\]
\end{example}

We postpone the proof of the fact that $\varphi_n$ and $\psi_n$ are chain maps and the inverses of each other
to Subsection~\ref{subsec:chainmaps}.
This implies that  for each $n\in\mathbb Z$, we have the induced isomorphism $\varphi^*_n: H_n^{\rm LB}(X) \to H_{n-1}^{\rm N}(X)$ defined by 
\[
\varphi^*_n \Big(\Big[\big((a,b_1), (a, b_2), \ldots , (a ,b_n) \big)\Big]\Big) = \Big[\varphi_n\Big(\big((a,b_1), (a, b_2), \ldots , (a ,b_n) \big)\Big)\Big],
\]
and the inverse isomorphism $\psi^*_n:  H_{n-1}^{\rm N}(X) \to H_{n}^{\rm LB}(X)$ defined by  
\[
\psi^*_n \Big(\Big[(a, a_1, a_2, \ldots , a_n)\Big]\Big) = \Big[ \psi_n \Big( (a, a_1, a_2, \ldots , a_n) \Big) \Big].
\]

Let $A$ be an abelian group.  
The bijective chain map $\varphi_n$ induces the bijective chain map  $\varphi_n\otimes {\rm id}: C_n^{\rm LB}(X;A) \to C_{n-1}^{\rm N}(X;A)$, and hence, we have an isomorphism $\varphi^*_n \otimes {\rm id}: H_n^{\rm LB}(X;A) \to H_{n-1}^{\rm N}(X;A)$.
The bijective chain map $\varphi_n$ induces the bijective cochain map $\varphi^n:  C^{n-1}_{\rm N}(X;A) \to C^n_{\rm LB}(X;A)$ defined by $\varphi^n (f) = f\circ \varphi_n$, and hence, we have an isomorphism 
${\varphi}_*^n:  H^{n-1}_{\rm N}(X;A) \to H^n_{\rm LB}(X;A)$. 
Thus we have the following theorem. 
\begin{theorem}\label{mainthm1}
For any $n\in \mathbb Z$, we have 
\[
H_n^{\rm LB}(X; A) \cong H_{n-1}^{\rm N}(X; A) \mbox{ and }
H^n_{\rm LB}(X; A) \cong H^{n-1}_{\rm N}(X; A).
\]

\end{theorem}

\subsection{Correspondence between cocycle invariants of links}

 Let $\varphi_2: C_2^{\rm LB}(X) \to C_1^{\rm N}(X)$ be the chain map defined in Subsection~\ref{subsec:correspondence1},  and $\psi_2$  the inverse chain map of $\varphi_2$.

Suppose that we have exactly the same situation as in  Subsection~\ref{subsection:linkinvariant}. 
Let $\langle\, \rangle: X^3\to X$ be the corresponding vertical-tribracket of the horizontal-tribracket $[\, ]$, and $\overline{C}$  the corresponding region $X$-coloring of the semi-arc $X^2$-coloring $C$ through the translation map $T$ defined in Remark~\ref{Rem:coloringtranslation}.

\begin{figure}
  \begin{center}
    \includegraphics[clip,width=8cm]{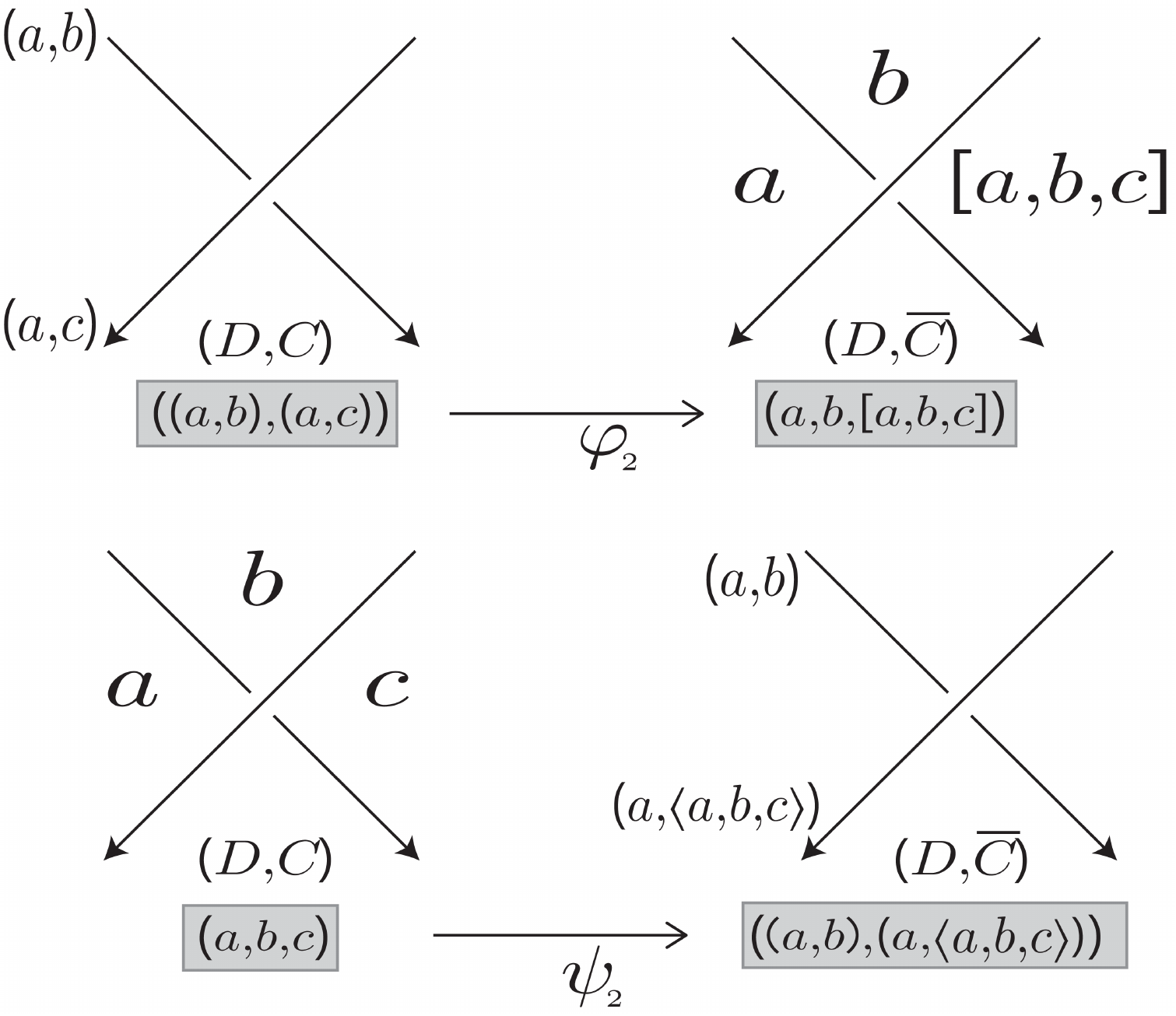}
    \label{doublepoint4}
  \end{center}
\end{figure}
At a crossing  $\chi$ of $D$ as depicted in the upper of Figure~\ref{doublepoint4}, we have 
\[
\begin{array}{ll}
w^{\rm N}(D, \overline{C}; \chi) 
&={\rm sign}(\chi) \big(a,b,[a,b,c]\big)\\[3pt]
&=\varphi_2 \Big({\rm sign}(\chi)\big((a,b),(a,c)\big)\Big)\\[3pt]
&=\varphi_2\big(w(D, C; \chi)\big),
\end{array}
\]
which implies that $W^{\rm N}(D, \overline{C}) =\varphi_2\big(W(D, C)\big)$. Thus we have 
\[
\begin{array}{ll}
\mathcal{H}^{\rm N}(D)&=\Big\{\big[W^{\rm N}(D, \overline{C})\big]  ~\Big|~ \ \overline{C} \in {\rm Col}_{X}^{\rm R} (D) \Big\}\\[4pt]
&=\Big\{\varphi_2^{\ast}\big(\big[W(D, C)\big]\big) ~\Big|~ \ C \in {\rm Col}_{X^2}^{\rm SA} (D) \Big\}\\[4pt]
&=\varphi_2^{\ast}\big(\mathcal{H}(D)\big).
\end{array}
\]

For the 2-cocycle $\theta: C^{\rm LB}_2(X) \to A$, set $\overline{\theta}: C^{\rm N}_1(X) \to A$ by $\overline{\theta}=\theta \circ \psi_2$, which is a $1$-cocycle of $C_{\rm N}^{\ast}(X)$. Then we have
\[
\begin{array}{ll}
\overline{\theta}\big(w^{\rm N}(D, \overline{C}; \chi)\big)&= \theta \circ \psi_2 \Big(\varphi_2\big(w(D, C; \chi)\big)\Big)\\[3pt]
&=\theta \big(w(D, C; \chi)\big),
\end{array}
\]
which implies $\overline{\theta}\big(W^{\rm N}(D, \overline{C})\big)=\theta(W(D, C))$. 
Thus we have $\Phi_{\overline{\theta}}^{\rm N}(D) =\Phi_{\theta}(D)$. 
Therefore, the Niebrzydowski's invariants  $\mathcal{H}^{\rm N}(L)$ and   $\Phi_{\overline{\theta}}^{\rm N}(L)$ can be obtained from our invariants $\mathcal{H}(L)$ and $\Phi_{\theta}(L)$, respectively.

Conversely, suppose that we have exactly  the same situation as in  Subsection~\ref{subsec: Region colorings of link diagrams, cocycle invariants}. 
We may assume that a diagram $D$ is connected. 
Let $[\,]: X^3 \to X$ be the corresponding horizontal-tribracket of the vertical-tribracket $\langle\, \rangle$, and $\overline{C}$  the corresponding semi-arc $X^2$-coloring of the region $X$-coloring $C$ through the translation map $T^{-1}$ defined in Remark~\ref{Rem:coloringtranslation}. 

At a crossing  $\chi$ of $D$ depicted in the lower of Figure~\ref{doublepoint4}, we have 
\[
\begin{array}{ll}
w(D, \overline{C}; \chi) 
&={\rm sign}(\chi)\big((a,b),(a,\langle a,b,c\rangle)\big)\\[3pt]
&=\psi_2 \big({\rm sign}(\chi) (a,b,c)\big)\\[3pt]
&=\psi_2\big(w^{\rm N}(D, C; \chi)\big),
\end{array}
\]
which implies that $W(D, \overline{C}) =\psi_2\big(W^{\rm N}(D, C)\big)$. Thus we have 
\[
\begin{array}{ll}
\mathcal{H}(D)&=\Big\{\big[W(D, \overline{C})\big]  ~\Big|~ \ \overline{C} \in {\rm Col}_{X^2}^{\rm SA} (D) \Big\}\\[4pt]
&=\Big\{\psi_2^{\ast}\big(\big[W^{\rm N}(D, C)\big]\big) ~\Big|~ \ C \in {\rm Col}_{X}^{\rm R} (D) \Big\}\\[4pt]
&=\psi_2^{\ast}\big(\mathcal{H}^{\rm N}(D)\big).
\end{array}
\]

For the 1-cocycle $\theta: C^{\rm N}_1(X) \to A$, set $\overline{\theta}: C^{\rm LB}_2(X) \to A$ by $\overline{\theta}=\theta \circ \varphi_2$, which is a 2-cocycle of $C_{\rm LB}^{\ast}(X)$. Then we have
\[
\begin{array}{ll}
\overline{\theta}\big(w(D, \overline{C}; \chi)\big)&=\theta \circ \varphi_2 \Big( \psi_2\big( w^{\rm N}(D, C; \chi)\big)\Big)\\[3pt]
&=\theta \big( w^{\rm N}(D, C; \chi)\big),
\end{array}
\]
which implies $\overline{\theta}\big(W(D, \overline{C})\big)=\theta\big(W^{\rm N}(D, C)\big)$. 
Thus we have $\Phi_{\overline{\theta}}(D) =\Phi_{\theta}^{\rm N}(D)$. 
Therefore, our invariants $\mathcal{H}(L)$ and $\Phi_{\overline{\theta}}(L)$ can be obtained from the Niebrzydowski's invariants $\mathcal{H}^{\rm N}(L)$  and $\Phi_{\theta}^{\rm N}(L)$, respectively.

As a consequence,  we have the following theorem.
\begin{theorem}
Let $X$ be  a   set,  and let $[\, ]$ and $\langle\, \rangle$ be  a horizontal- and a vertical-tribracket on $X$, respectively, that are corresponding each other.
Let $\theta$ and $\theta'$ be a $2$-cocycle of $C^{*}_{\rm LB} (X)$ and a $1$-cocycle of $C^*_{N}(X)$, respectively, that are corresponding each other through the isomorphisms $\varphi^2_*$ and  $\psi^2_*$. 
Let $L$ be a link.
Then we have 

\[
\varphi_2^* \big(\mathcal{H}(L)\big) = \mathcal{H}^{\rm N}(L) \mbox{ and } \psi_2^* \big(\mathcal{H}^{\rm N}(L)\big) = \mathcal{H}(L),
\]
and 
\[
\Phi_{\theta}(L)=\Phi_{\theta'}^{\rm N}(L).
\]

\end{theorem}

\subsection{Correspondence between cocycle invariants of surface-links}

 Let $\varphi_3: C_3^{\rm LB}(X) \to C_2^{\rm N}(X)$ be the chain map defined in Subsection~\ref{subsec:correspondence1},  and $\psi_3$  the inverse chain map of $\varphi_3$.

Suppose that we have exactly the same situation as in  Subsection~\ref{subsection:surfaceinvariant}.
Let $\langle\, \rangle: X^3\to X$ be the corresponding vertical-tribracket of the horizontal-tribracket $[\, ]$, and $\overline{C}$  the corresponding region $X$-coloring of the semi-sheet $X^2$-coloring $C$ through the translation map $T$ defined in Remark~\ref{Rem:coloringtranslation2}.

\begin{figure}
  \begin{center}
    \includegraphics[clip,width=12cm]{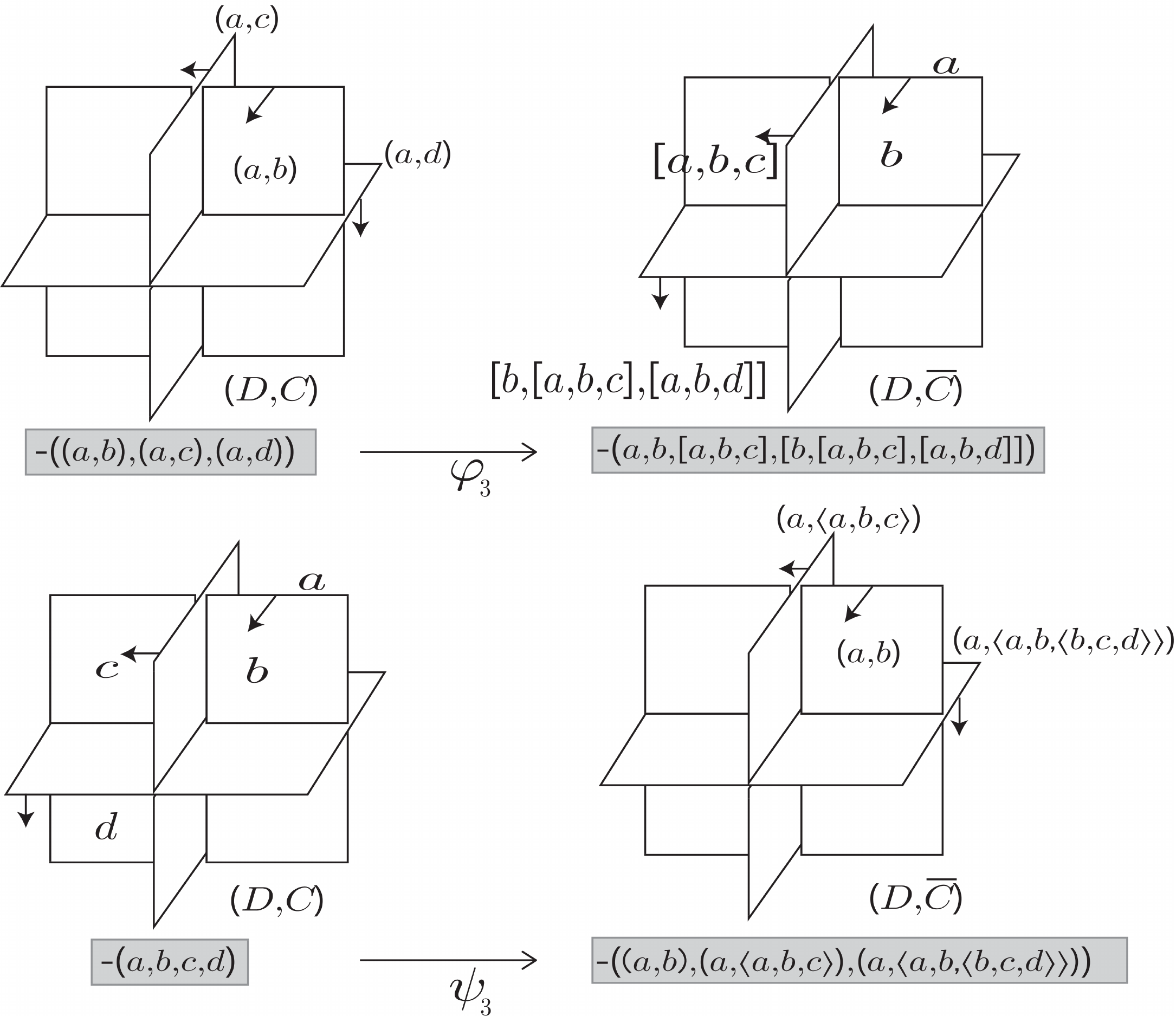}
    \label{triplepoint3}
  \end{center}
\end{figure}
At a triple point  $\tau$ of $D$ as depicted in the upper of Figure~\ref{triplepoint3}, we have 
\[
\begin{array}{ll}
w^{\rm N}(D, \overline{C};  \tau) 
&={\rm sign}(\chi) \Big(a,b,[a,b,c],\big[b,[a,b,c],[a,b,d]\big]\Big) \\[3pt]
&=\varphi_3 \Big({\rm sign}(\chi)\big((a,b),(a,c), (a,d)\big)\Big)\\[3pt]
&=\varphi_3\big(w(D, C; \tau)\big),
\end{array}
\]
which implies that $W^{\rm N}(D, \overline{C}) =\varphi_3\big(W(D, C)\big)$. 
Thus we have 
\[
\begin{array}{ll}
\mathcal{H}^{\rm N}(D)&=\Big\{\big[W^{\rm N}(D, \overline{C})\big]  ~\Big|~ \ \overline{C} \in {\rm Col}_{X}^{\rm R} (D) \Big\}\\[4pt]
&=\Big\{\varphi_3^{\ast}\big(\big[W(D, C)\big]\big) ~\Big|~ \ C \in {\rm Col}_{X^2}^{\rm SS} (D) \Big\}\\[4pt]
&=\varphi_3^{\ast}\big(\mathcal{H}(D)\big).
\end{array}
\]

For the 3-cocycle $\theta: C^{\rm LB}_3(X) \to A$, set $\overline{\theta}: C^{\rm N}_2(X) \to A$ by $\overline{\theta}=\theta \circ \psi_3$, which is a $2$-cocycle of $C_{\rm N}^{\ast}(X)$. Then we have
\[
\begin{array}{ll}
\overline{\theta}\big(w^{\rm N}(D, \overline{C}; \tau)\big)&=\theta \circ \psi_3 \Big(\varphi_3\big(w(D, C; \tau)\big) \Big)\\[3pt]
&=\theta \big(w(D, C; \tau)\big),
\end{array}
\]
which implies $\overline{\theta}\big(W^{\rm N}(D, \overline{C})\big)=\theta(W(D, C))$. 
Thus we have $\Phi_{\overline{\theta}}^{\rm N}(D) =\Phi_{\theta}(D)$. 
Therefore, the Niebrzydowski's invariants $\mathcal{H}^{\rm N}(F)$ and  $\Phi_{\overline{\theta}}^{\rm N}(F)$ can be obtained from our invariants $\mathcal{H}(F)$ and $\Phi_{\theta}(F)$, respectively.

Conversely, suppose that we have exactly  the same situation as in  Subsection~\ref{subsec: Region colorings of surface-link diagrams, cocycle invariants}.
We may assume that a diagram $D$ is connected. 
Let $[\,]: X^3 \to X$ be the corresponding horizontal-tribracket of the vertical-tribracket $\langle\, \rangle$, and $\overline{C}$  the corresponding semi-sheet $X^2$-coloring of the region $X$-coloring $C$ through the translation map $T^{-1}$ defined in Remark~\ref{Rem:coloringtranslation2}. 

At a triple point  $\tau$ of $D$ depicted in the lower of Figure~\ref{triplepoint3}, we have 
\[
\begin{array}{ll}
w(D, \overline{C}; \tau) 
&={\rm sign}(\tau)\Big(\big(a,b\big),\big(a,\langle a,b,c\rangle\big),\big(a,\langle a,b,\langle b,c,d\rangle \rangle \big)\Big)\\[3pt]
&=\psi_3 \big({\rm sign}(\tau) (a,b,c,d)\big)\\[3pt]
&=\psi_3\big(w^{\rm N}(D, C; \tau)\big),
\end{array}
\]
which implies that $W(D, \overline{C}) =\psi_3\big(W^{\rm N}(D, C)\big)$. Thus we have 
\[
\begin{array}{ll}
\mathcal{H}(D)&=\Big\{\big[W(D, \overline{C})\big]  ~\Big|~ \ \overline{C} \in {\rm Col}_{X^2}^{\rm SS} (D) \Big\}\\[4pt]
&=\Big\{\psi_3^{\ast}\big(\big[W^{\rm N}(D, C)\big]\big) ~\Big|~ \ C \in {\rm Col}_{X}^{\rm R} (D) \Big\}\\[4pt]
&=\psi_3^{\ast}\big(\mathcal{H}^{\rm N}(D)\big).
\end{array}
\]

For the 2-cocycle $\theta: C^{\rm N}_2(X) \to A$, set $\overline{\theta}: C^{\rm LB}_3(X) \to A$ by $\overline{\theta}=\theta \circ \varphi_3$, which is a 3-cocycle of $C_{\rm LB}^{\ast}(X)$. Then we have
\[
\begin{array}{ll}
\overline{\theta}\big(w(D, \overline{C}; \tau)\big)&=\theta \circ \varphi_3 \Big(\psi_3\big(w^{\rm N}(D, C; \tau)\big)\Big)\\[3pt]
&=\theta \big( w^{\rm N}(D, C; \tau)\big),
\end{array}
\]
which implies $\overline{\theta}\big(W(D, \overline{C})\big)=\theta\big(W^{\rm N}(D, C)\big)$. 
Thus we have $\Phi_{\overline{\theta}}(D) =\Phi_{\theta}^{\rm N}(D)$. 
Therefore, our invariants $\mathcal{H}(L)$ and $\Phi_{\overline{\theta}}(L)$ can be obtained from the Niebrzydowski's invariants $\mathcal{H}^{\rm N}(L)$ and $\Phi_{\theta}^{\rm N}(L)$, respectively.

As a consequence,  we have the following theorem.
\begin{theorem}
Let $X$ be  a   set,  and let $[\, ]$ and $\langle\, \rangle$ be  a horizontal- and a vertical-tribracket on $X$, respectively, that are corresponding each other.
Let $\theta$ and $\theta'$ be a $3$-cocycle of $C^{*}_{\rm LB} (X)$ and a $2$-cocycle of $C^*_{N}(X)$, respectively, that are corresponding each other through the isomorphisms $\varphi^3_*$ and  $\psi^3_*$. 
Let $F$ be a surface-link.
Then we have 
\[
\varphi_3^* \big(\mathcal{H}(F)\big) = \mathcal{H}^{\rm N}(F) \mbox{ and } \psi_3^* \big(\mathcal{H}^{\rm N}(F)\big) = \mathcal{H}(F),
\]
and 
\[
\Phi_{\theta}(F)=\Phi_{\theta'}^{\rm N}(F).
\]

\end{theorem}


\subsection{Chain maps $\varphi_n$ and $\psi_n$}\label{subsec:chainmaps}
In this subsection, we will discuss the homomorphisms $\varphi_n$ and $\psi_n$ defined in Subsection~\ref{subsec:correspondence1}. 
In particular, we show that they are chain maps between the chain complexes $C_*^{\rm LB} (X)$ and $C_*^{\rm N} (X)$, and  that they are the inverses of each other.

\begin{notation} In this subsection, we use the following notations.
\begin{itemize}
\item The bold angle bracket $\blangle a_0, a_1, a_2, \ldots , a_n \brangle $ means 
$$\Bigg\langle a_{0}, a_1, \bigg\langle a_{1}, a_{2}, \Big\langle a_2, a_3, \big\langle \cdots  \big\langle a_{n-3}, a_{n-2}, \langle a_{n-2}, a_{n-1}, a_n\rangle \big\rangle \cdots \big\rangle  \Big\rangle\bigg\rangle\Bigg\rangle.$$


\item The augmented bold square bracket$\blsq a_0, a_1, a_2, \ldots , a_n; b \brsq $ means 
$$\Bigg[ a_{n-1}, a_n, \bigg[ a_{n-2}, a_{n-1}, \Big[\cdots \Big[ a_2, a_3, \big[ a_1, a_2, [a_0, a_1, b] \big] \Big] \cdots \Big]\bigg]\Bigg].$$
\end{itemize}
\end{notation}

Let $X$ be a   set $X$, $[\, ]$ a horizontal-tribracket on $X$, and $\langle\, \rangle$
the corresponding vertical-tribracket of $[\,]$.

\begin{lemma}\label{lem:tribracketfomula}
For nonnegative integers $m, n$ and $i$, let  $a_0, a_1, \ldots, a_{\max\{m, n\}}, b, x_{i+1},$ $x_{i+2}, \ldots , x_{m}, y_{i+1}, \ldots , y_{n}\in X$.
\begin{enumerate}
\renewcommand{\labelenumi}{(\arabic{enumi})}

\item \label{formula1-1} For $n\geq 2$ and $i\leq n-2$, we have 
\[
\begin{array}{l}
\blangle a_0, a_1, \ldots , a_n \brangle = \blangle a_0, a_1, \ldots ,a_{i}, \blangle a_i, a_{i+1}, \ldots , a_n \brangle\brangle.
\end{array}
\]

\item \label{formula1-2} For $n\geq 2$, we have 
\[
\begin{array}{l}
\blangle a_0, a_1, \ldots , a_n \brangle = \blangle a_0, \langle  a_0, a_1, a_2 \rangle, a_2, a_3,\ldots , a_n \brangle.
\end{array}
\]

\item\label{formula1-3}  For $n\geq 1$, we have 
$$\blangle a_0, a_1,\ldots , a_n ,\blsq a_0, a_1,\ldots , a_n; b \brsq\brangle =b. $$

\item \label{formula1-4}  For $m\geq 1$, $n\geq 2$ and $0\leq i \leq \min\{ m-1, n-2\}$, we have 
$$\begin{array}{l}
\blsq a_0, a_1,\ldots , a_{i}, x_{i+1}, \ldots ,x_{m} ; \blangle a_0, a_1,\ldots , a_i, y_{i+1}, \ldots , y_n \brangle \brsq \\[5pt]
=\blsq a_{i},x_{i+1}, \ldots ,x_{m} ; \blangle a_i, y_{i+1}, \ldots , y_n \brangle \brsq.
\end{array}
$$

In particular, 
$$\begin{array}{l}
\blsq a_0, a_1,\ldots , a_{m} ; \blangle a_0, a_1,\ldots , a_n \brangle \brsq \\[5pt]
=\left\{\begin{array}{cl}
\blangle a_m, a_{m+1},\ldots , a_n \brangle &(m<n-1),\\[5pt]
a_n&(m=n-1),\\[5pt]
\blsq a_{n-1}, a_n,\ldots , a_m; a_n \brsq &(m>n-1). \\
\end{array}
\right.
\end{array}
$$

\end{enumerate}
\end{lemma}
\begin{proof}
(\ref{formula1-1}) This  follows easily from  the definition of the bold angle bracket. 

(\ref{formula1-2}) We have 
\[
\begin{array}{l}
\blangle a_0, a_1, \ldots , a_n \brangle \\[5pt]
=\Big\langle a_0, a_1, \big\langle a_1, a_2,  \blangle a_2,  \ldots , a_n \brangle \big\rangle \Big\rangle \\[5pt]
=\Big\langle a_0, \langle a_0 , a_1, a_2\rangle , \big\langle\langle a_0 , a_1, a_2\rangle, a_2,  \blangle a_2,  \ldots , a_n \brangle \big\rangle \Big\rangle \\[5pt]
=\blangle a_0, \langle a_0 , a_1, a_2\rangle , a_2 ,   \blangle a_2,  \ldots , a_n \brangle  \brangle \\[5pt]
= \blangle a_0, \langle  a_0, a_1, a_2 \rangle, a_2, a_3,\ldots , a_n \brangle,
\end{array}
\]
where the first and third equalities follow from the definition of the bold angle bracket, the second equality follows from (i) of ($\mathcal{V}2$) of Definition~\ref{s-ternary operation}, and the last equality follows from (\ref{formula1-1})  of Lemma~\ref{lem:tribracketfomula}.

(\ref{formula1-3})  We prove this formula by induction for $n\geq 1$.
When $n=1$, we have 
$\blangle a_0, a_1 ,\blsq a_0, a_1; b \brsq\brangle=\big\langle  a_0, a_1 ,[ a_0, a_1, b ] \big\rangle =b $ by (1) of Lemma~\ref{lem:tribrackets1}. 
Assuming the formula holds for some $n\geq 1$, we have   
$$
\begin{array}{l}
\blangle a_0, a_1,\ldots , a_{n+1} ,\blsq a_0, a_1,\ldots , a_{n+1}; b \brsq\brangle \\[5pt]
=\blangle a_0, a_1, \ldots , a_{n},  \bigg \langle a_{n}, a_{n+1}, \Big [  a_{n}, a_{n+1},  \blsq a_0, a_1,\ldots , a_{n}; b \brsq \Big ] \bigg\rangle \brangle\\[5pt] 
=\blangle a_0, a_1,\ldots , a_{n} ,\blsq a_0, a_1,\ldots , a_{n}; b \brsq\brangle \\[5pt]
=b,
\end{array} 
$$
where the first equality follows from the definitions, the second equality follows from (1) of Lemma~\ref{lem:tribrackets1}, and the third equality follows from the assumption.

(\ref{formula1-4})
We show the first equality only, and we leave the proof of the other formulas to the reader.  

For $k<i$, we have
$$\begin{array}{l}
\blsq a_k, a_{k+1},\ldots , a_{i}, x_{i+1}, \ldots ,x_{m} ; \blangle a_k, a_{k+1},\ldots , a_i, y_{i+1}, \ldots , y_n \brangle \brsq \\[7pt]
=\Bigg[ x_{m-1} , x_m,\bigg[ \cdots \Big[ a_{k+1},  \bullet_{k+2},  \big[a_k, a_{k+1},  \blangle a_k, a_{k+1},\ldots , a_i, y_{i+1}, \ldots , y_n \brangle  \big] \Big] \cdots \bigg] \Bigg]\\[11pt]
=\Bigg[ x_{m-1} , x_m,\bigg[ \cdots   \Big[ a_{k+1},  \bullet_{k+2}, \big[a_k, a_{k+1},  \big\langle a_k, a_{k+1}, \blangle  a_{k+1},\ldots , a_i, y_{i+1}, \ldots , y_n \brangle \big\rangle  \big] \Big] \cdots \bigg] \Bigg]\\[11pt]
=\Bigg[ x_{m-1} , x_m,\bigg[ \cdots   \Big[ a_{k+1},  \bullet_{k+2}, \blangle  a_{k+1}, \ldots , a_i, y_{i+1}, \ldots , y_n \brangle  \Big]  \cdots \bigg] \Bigg]\\[11pt]
=\blsq a_{k+1},\ldots , a_{i}, x_{i+1}, \ldots ,x_{m} ; \blangle  a_{k+1},\ldots , a_i, y_{i+1}, \ldots , y_n \brangle \brsq , 
\end{array}
$$
where the first, second and fourth equalities follow from the definitions, and for the third equality, we applied (2) of Lemma~\ref{lem:tribrackets1} for the innermost square bracket and the outermost angle bracket, 
and where $\bullet=\left\{ 
\begin{array}{ll}
a & (k+2\leq i),\\
x & (\mbox{otherwise}).
\end{array}
\right.
$
By repeating the above operation, we have 
$$\begin{array}{l}
\blsq a_0, a_1,\ldots , a_{i}, x_{i+1}, \ldots ,x_{m} ; \blangle a_0, a_1,\ldots , a_i, y_{i+1}, \ldots , y_n \brangle \brsq \\[5pt]
\blsq a_1,\ldots , a_{i}, x_{i+1}, \ldots ,x_{m} ; \blangle a_1,\ldots , a_i, y_{i+1}, \ldots , y_n \brangle \brsq \\[5pt]
=\cdots =\blsq a_{i},x_{i+1}, \ldots , x_{m} ; \blangle a_i, y_{i+1}, \ldots , y_n \brangle \brsq. \\[5pt]
\end{array}
$$

\end{proof}

For a positive integer $n$ and for $a(=a_0), a_1, \ldots , a_n, b_1, \ldots , b_n \in X$, we set $z_i~(i\in \{0,1,\ldots ,n \})$ by 
\[
\left\{
\begin{array}{ll}
z_0=a,\\
z_1=b_1,\\
z_i=[z_{i-2}, z_{i-1}, z_{i-1}|_{b_{i-1} \mapsto b_i}] & (i\in\{2,\ldots, n \}), 
\end{array}
\right.
\]
we set   $w_i~(i\in \{1,\ldots ,n \})$ by 
\[
\left\{
\begin{array}{ll}
w_1=a_1,\\
w_2=\langle a,a_1,a_2\rangle,\\
w_i= w_{i-1} |_{a_{i-1} \mapsto \langle a_{i-2}, a_{i-1}, a_i \rangle}& (i\in \{3,\ldots , n\}),
\end{array}
\right.
\]
and we set $y_{(i,j)}~(i\in \{0,\cdots ,n-1\}, j\in \{1,\ldots ,n-1\})$ by 
\[
y_{(i,j)} = \left\{
\begin{array}{ll}
\langle a_{j-1},  a_j,  y_{(i,j+1)}\rangle  & (j<i),\\[5pt]
\langle a_{j-1},  a_j,  a_{j+1} \rangle & (j=i, i+1),\\[5pt]
\langle y_{(i,j-1)},  a_j,  a_{j+1} \rangle & (j>i+1).
\end{array}
\right.
\]
We then have the following lemmas.
\begin{lemma}\label{lem:zw}
\begin{enumerate}
\renewcommand{\labelenumi}{(\arabic{enumi})}
\item \label{formula2-1} When  $i\geq 2$, $z_i=\blsq z_0, z_1, \ldots , z_{i-1}; b_i \brsq$.
\item \label{formula2-2} When $j\geq 2$, $w_j=\blangle  a, a_1, \ldots , a_j \brangle$.
\end{enumerate}
\end{lemma}
\begin{proof}
(\ref{formula2-1}) We show this formula by induction for $i\in \{2,\ldots , n\}$. 
When $i=2$, we have 
$z_2= [z_0, z_1, z_1|_{b_1 \mapsto b_2}] = [z_0,z_1, b_2] = \blsq z_0,z_1; b_2 \brsq $. 
Assuming the formula holds  for some $i \geq 2$,
we have 
\[
\begin{array}{ll}
z_{i+1} &= \big[z_{i-1}, z_i, z_i|_{b_{i} \mapsto b_{i+1}}\big]\\[5pt]
& =\big[z_{i-1}, z_i, \blsq z_0, z_1, \ldots , z_{i-1}; b_i \brsq \Big|_{b_{i} \mapsto b_{i+1}}\big] \\[5pt]
& =\big[z_{i-1}, z_i, \blsq z_0, z_1, \ldots , z_{i-1}; b_{i+1} \brsq \big] \\[5pt]
& =\blsq z_0, z_1, \ldots , z_{i-1}, z_i ; b_{i+1} \brsq,
\end{array}
\]
where the first and last equalities follow from the definitions, and the second quality follows from the assumption. 

(\ref{formula2-2}) This  follows easily from  the definition of $w_j$. 
\end{proof}
\begin{lemma}\label{lem:y}
\begin{enumerate}
\renewcommand{\labelenumi}{(\arabic{enumi})}

\item \label{formula3-5} When  $j\leq i-1$, 
$$\big[ a_{j-1}, a_{j}, y_{(i-1,j)} \big] =
\left\{ \begin{array}{ll}
 y_{(i-1, j+1)} & (j<i-1), \\[3pt]
a_i & (j=i-1).
\end{array}
\right.
 $$

\item \label{formula3-2} When  $j\leq i-1$, 
$$\blangle a_{j-1}, a_{j},\ldots , a_{i} \brangle=y_{(i-1,j)}. $$
In particular, $w_i = y_{(i-1, 1)}$.

\item \label{formula3-1} When  $j\geq i+1$, 
$$\blangle a_{i-1}, a_{i},\ldots , a_{j} \brangle = \blangle a_{i-1}, y_{(i-1,i)}, y_{(i-1,i+1)},\ldots , y_{(i-1,j-1)} \brangle . $$
\item \label{formula3-3} When  $j\leq i-1$, 
$$ \Big[ a, w_j, w_i \Big] =\left\{
\begin{array}{ll}
\blangle y_{(i-1,1)}, y_{(i-1,2)},  \ldots , y_{(i-1,j+1)} \brangle & (j<i-1),\\[5pt]
\blangle y_{(i-1,1)}, y_{(i-1,2)},  \ldots , y_{(i-1,i-1)}, a_i \brangle & (j=i-1).
\end{array}
\right. $$

\item \label{formula3-4} When  $j\geq i+1$, 
$$ \Big[ a, w_i, w_j \Big] =
\blangle y_{(i-1,1)}, y_{(i-1,2)},  \ldots , y_{(i-1,i-1)}, a_i , a_{i+1}, \ldots , a_j \brangle.$$

\end{enumerate}
\end{lemma}
\begin{proof}
(\ref{formula3-5}) This follows from (2) of Lemma~\ref{lem:tribrackets1} as 
$$
\big[a_{j-1}, a_{j} , y_{(i-1,j)}   \big] =\big[a_{j-1}, a_{j} , \langle a_{j-1}, a_{j}, y_{(i-1,j+1)} \rangle \big]
=y_{(i-1,j+1)}.
$$

(\ref{formula3-2}) This  follows easily from  the definition of $y_{(i,j)}$.

(\ref{formula3-1}) 
For $i\leq k \leq j-3$,
\[
\begin{array}{l}
\blangle a_{i-1}, y_{(i-1, i)}, \ldots, y_{(i-1, k)},  a_{k+1}, \ldots , a_{j} \brangle \\[5pt]
=  \blangle a_{i-1}, y_{(i-1, i)}, \ldots, y_{(i-1, k)}, \blangle   y_{(i-1, k)}, a_{k+1}, \ldots , a_{j} \brangle \brangle\\[5pt]
=  \blangle a_{i-1}, y_{(i-1, i)}, \ldots, y_{(i-1, k)},  \blangle   y_{(i-1, k)}, \langle  y_{(i-1, k)} ,a_{k+1}, a_{k+2} \rangle, a_{k+2}, \ldots , a_{j} \brangle \brangle \\[5pt]
=  \blangle a_{i-1}, y_{(i-1, i)}, \ldots, y_{(i-1, k)},  \blangle   y_{(i-1, k)}, y_{(i-1, k+1)}, a_{k+2}, \ldots , a_{j} \brangle \brangle \\[5pt]
=\blangle a_{i-1}, y_{(i-1, i)}, \ldots, y_{(i-1, k+1)},  a_{k+2}, \ldots , a_{j} \brangle,
\end{array}
\]
where the first and fourth equalities follow from (\ref{formula1-1}) of Lemma~\ref{lem:tribracketfomula}, the second equality follows from  (\ref{formula1-2}) of Lemma~\ref{lem:tribracketfomula}, and 
the third equality  follows from the definition of $y_{(i,j)}$.

By repeating the above operation, we have   
\[
\begin{array}{l}
\blangle a_{i-1}, a_{i},\ldots , a_{j} \brangle \\[5pt]
= \blangle a_{i-1}, \langle a_{i-1}, a_{i}, a_{i+1} \rangle , a_{i+1}, \ldots , a_{j} \brangle\\[5pt]
= \blangle a_{i-1}, y_{(i-1, i)}, a_{i+1}, \ldots , a_{j} \brangle \\[5pt]
= \cdots =  \blangle a_{i-1}, y_{(i-1,i)}, y_{(i-1,i+1)},\ldots , y_{(i-1,j-2)}, a_{j-1}, a_j \brangle \\[5pt]
=  \blangle a_{i-1}, y_{(i-1,i)}, y_{(i-1,i+1)},\ldots , y_{(i-1,j-2)}, \langle y_{(i-1,j-2)}, a_{j-1}, a_j \rangle  \brangle \\[5pt]
=  \blangle a_{i-1}, y_{(i-1,i)}, y_{(i-1,i+1)},\ldots , y_{(i-1,j-2)}, y_{(i-1,j-1)}  \brangle , \\[5pt]
\end{array}
\]
where the first equality follows from (\ref{formula1-2}) of Lemma~\ref{lem:tribracketfomula}, the second and last equalities follow from  the definition of $y_{(i,j)}$, and the second equality from the last follows from  (\ref{formula1-1}) of Lemma~\ref{lem:tribracketfomula}.

(\ref{formula3-3}) For $k\leq j-3$, we have 
\[
\begin{array}{l}
\Big[ a_k , \blangle a_k, a_{k+1}, \ldots , a_j\brangle , y_{(i-1,k+1)} \Big]\\[5pt]
=\Big[ a_k , \big\langle a_k, a_{k+1} , \blangle a_{k+1}, a_{k+2}, \ldots , a_j\brangle  \big\rangle , y_{(i-1,k+1)} \Big]\\[5pt]
=\Big \langle y_{(i-1,k+1)} , \big[a_k, a_{k+1} , y_{(i-1,k+1)}   \big] ,  \Big[ a_{k+1} ,\blangle a_{k+1}, a_{k+2}, \ldots , a_j\brangle, \big[a_k, a_{k+1} , y_{(i-1,k+1)}   \big]  \Big]   \Big\rangle \\[5pt]
=\Big \langle y_{(i-1,k+1)} , y_{(i-1,k+2)},  \Big[ a_{k+1} ,\blangle a_{k+1}, a_{k+2}, \ldots , a_j\brangle, y_{(i-1,k+2)}\Big]   \Big\rangle , 
\end{array}
\]
where the first equality follows from the definition of the bold angle bracket, the second equality follows from (2) of  Lemma~\ref{lem:tribrackets2}, and the third equality follows from (\ref{formula3-5}) of Lemma~\ref{lem:y}.

Hence, by repeating the above operation, for $j<i-1$, we have 
\[
\begin{array}{l}
\Big[ a, w_j, w_i \Big] \\[5pt]
= \Big[ a,  \blangle a,  a_1, \ldots , a_j\brangle , y_{(i-1,1)} \Big]\\[5pt]
=\cdots 
= \blangle y_{(i-1,1)} , y_{(i-1,2)},  \ldots , y_{(i-1,j-1)} , \Big[ a_{j-2}, \blangle a_{j-2}, a_{j-1} , a_j \brangle , y_{(i-1,j-1)} \Big]   \brangle  \\[5pt]
= \blangle y_{(i-1,1)} , y_{(i-1,2)},  \ldots , y_{(i-1,j)} , y_{(i-1,j+1) }   \brangle, 
\end{array}
\]
where the first equality follows from (\ref{formula2-2}) of Lemma~\ref{lem:zw} and (2) of Lemma~\ref{lem:y}, and the last equality follows from the definition of the bold angle bracket, (2) of  Lemma~\ref{lem:tribrackets2},    and (\ref{formula3-5}) of Lemma~\ref{lem:y} 
as 
\[
\begin{array}{l}
\Big[ a_{j-2}, \blangle a_{j-2}, a_{j-1} , a_j \brangle , y_{(i-1,j-1)} \Big]  \\[5pt]
=\Big[ a_{j-2}, \big \langle a_{j-2}, a_{j-1} , a_j \big \rangle , y_{(i-1,j-1)} \Big]  \\[5pt]
=  \Big\langle  y_{(i-1,j-1)},  \big[a_{j-2}, a_{j-1}, y_{(i-1,j-1)}\big], \big[a_{j-1}, a_j, [a_{j-2}, a_{j-1}, y_{(i-1,j-1)}] \big] \Big\rangle   \\[5pt]
=  \Big\langle  y_{(i-1,j-1)}, y_{(i-1,j)} , \big[a_{j-1}, a_j, y_{(i-1,j)} \big] \Big\rangle   \\[5pt]
=  \Big\langle  y_{(i-1,j-1)}, y_{(i-1,j)} , y_{(i-1,j+1)} \Big\rangle.
\end{array}
\]
Similarly, for $j=i-1$, we have 
\[
\begin{array}{l}
\Big[ a, w_{i-1}, w_i \Big] \\[5pt]
=\cdots 
= \blangle y_{(i-1,1)} , y_{(i-1,2)},  \ldots , y_{(i-1,i-2)} , \Big[ a_{i-3}, \blangle a_{i-3}, a_{i-2} , a_{i-1} \brangle , y_{(i-1,i-2)} \Big]   \brangle  \\[5pt]
= \blangle y_{(i-1,1)} , y_{(i-1,2)},  \ldots , y_{(i-1,i-2)} ,   \Big\langle  y_{(i-1,i-2)}, y_{(i-1,i-1)} , \big[a_{i-2}, a_{i-1}, y_{(i-1,i-1)} \big] \Big\rangle   \brangle  \\[5pt]
= \blangle y_{(i-1,1)} , y_{(i-1,2)},  \ldots , y_{(i-1,i-1)} , a_i   \brangle,  
\end{array}
\]
where the second equality from the last follows from  (2) of  Lemma~\ref{lem:tribrackets2}, (1) of Lemma~\ref{lem:y}, and the last equality follows from  (\ref{formula3-5}) of Lemma~\ref{lem:y}.

(\ref{formula3-4}) For $k\leq i-3$, we have 
\[
\begin{array}{l}
\Big[ a_k  , y_{(i-1, k+1)} , \blangle a_k, a_{k+1}, \ldots , a_j \brangle  \Big]\\[5pt]
= \Big[ a_k , \big\langle a_k, a_{k+1}, y_{(i-1, k+2)}  \big\rangle  , \blangle a_k, a_{k+1}, \ldots , a_j \brangle   \Big]\\[5pt]
= \Big\langle  y_{(i-1, k+1)} ,  y_{(i-1, k+2)}, \big[a_{k+1}, y_{(i-1,k+2)}, [a_{k}, a_{k+1}, \blangle a_k, a_{k+1}, \ldots , a_j \brangle ]\big] \Big\rangle\\[6pt]
= \bigg\langle  y_{(i-1, k+1)} ,  y_{(i-1, k+2)}, \Big[a_{k+1}, y_{(i-1,k+2)}, \big[a_{k}, a_{k+1},  \langle a_{k}, a_{k+1},  \blangle a_{k+1}, \ldots , a_j \brangle \rangle \big]\Big] \bigg\rangle\\[8pt]

= \blangle  y_{(i-1, k+1)} ,  y_{(i-1, k+2)}, \big[a_{k+1}, y_{(i-1,k+2)}, \blangle a_{k+1}, \ldots , a_j \brangle \big] \brangle,\\[5pt]
\end{array}
\]
where the first and third equalities  follow from the definitions, the second equality follows from (2) of Lemma~\ref{lem:tribrackets2}, the definition of $y_{(i,j)}$ and the fourth equality follows from (2) of Lemma~\ref{lem:tribrackets1}.

Hence, by repeating the above operation, we have 
\[
\begin{array}{l}
\Big[ a, w_i, w_j \Big] \\[5pt]
= \Big[ a, y_{(i-1, 1)} , \blangle a, a_1, \ldots , a_j \brangle  \Big]\\[5pt]
=\cdots= \blangle  y_{(i-1, 1)} ,  \ldots ,y_{(i-1, i-1)}, \big[a_{i-2}, y_{(i-1,i-1)}, \blangle a_{i-2}, \ldots , a_j \brangle \big] \brangle\\[5pt]
= \blangle  y_{(i-1, 1)} ,  \ldots ,y_{(i-1, i-1)}, \Big[a_{i-2}, \langle a_{i-2},a_{i-1}, a_{i} \rangle, \blangle a_{i-2}, \ldots , a_j \brangle \Big] \brangle\\[5pt]
= \blangle  y_{(i-1, 1)} ,  \ldots ,y_{(i-1, i-1)}, \Big\langle y_{(i-1,i-1)}, a_i, \big[a_{i-1}, a_i, [ a_{i-2}, a_{i-1}, \blangle a_{i-2}, \ldots , a_j \brangle ] \big] \Big\rangle \brangle\\[5pt]
= \blangle  y_{(i-1, 1)} ,  \ldots ,y_{(i-1, i-1)}, \Big\langle y_{(i-1,i-1)}, a_i, \big[a_{i-1}, a_i, [ a_{i-2}, a_{i-1}, \langle a_{i-2}, a_{i-1}, \blangle a_{i-1}, \ldots , a_j \brangle\rangle  ] \big] \Big\rangle \brangle\\[5pt]

= \blangle  y_{(i-1, 1)} ,  \ldots ,y_{(i-1, i-1)}, \Big\langle y_{(i-1,i-1)}, a_i, \blangle a_{i}, \ldots , a_j \brangle   \Big\rangle \brangle\\[5pt]
= \blangle  y_{(i-1, 1)} ,  \ldots ,y_{(i-1, i-1)}, a_i, \ldots , a_j\brangle,\\[5pt]
\end{array}
\]
where the first equality follows from (2) of Lemma~\ref{lem:zw} and  (\ref{formula3-2}) of Lemma~\ref{lem:y},
the fifth  and third equalities from the last follow from the definitions, 
the fourth equality from the last follows from (2) of  Lemma~\ref{lem:tribrackets2}, the second equality  from the last follows from (2) of Lemma~\ref{lem:tribrackets1} and the last equality follows from (\ref{formula1-1}) of Lemma~\ref{lem:tribracketfomula}.

\end{proof}

We recall that in Subsection~\ref{subsec:correspondence1}, the homomorphism 
$\varphi_n: C_n^{\rm LB}(X) \to C_{n-1}^{\rm N}(X)$ was defined by 
\[
\varphi_n \Big(\big((a,b_1), (a, b_2), \ldots , (a ,b_n) \big)\Big) = (z_0, z_1, \ldots , z_n)
\]
if $n\geq 1$, and $\varphi_n=0$ otherwise.
The homomorphism  
$\psi_n:  C_{n-1}^{\rm N}(X) \to C_{n}^{\rm LB}(X)$ was defined by
\[
\psi_n \big((a, a_1, a_2, \ldots , a_n)\big) = \big((a,w_1), (a,w_2), \ldots , (a, w_n)\big)
\]
if $n\geq 1$, and $\psi_n=0$ otherwise
\begin{lemma}\label{lem:welldefined}
\begin{itemize}
\item[(1)] $\varphi_n$ is well-defined.
\item[(2)] $\psi_n$ is well-defined.
\end{itemize}
\end{lemma}
\begin{proof}
(1) We regard $\varphi_n$ as a homomorphism from $C_n(X)$ to $C_{n-1}^{\rm Nie}(X)$.
It suffices to show that $\varphi_{n} \big(D_{n}(X)\big) \subset D_{n-1}^{\rm Nie}(X)$.
 
For $\big((a,b_1), (a,b_2), \ldots , (a,b_n)\big) \in \bigcup_{a\in X} (\{a\} \times X)^n $, suppose $b_{i}=b_{i+1}$ for some $i$ ($i \in \{1, 2, \cdots n\} $).
We then have 
\[
\begin{array}{ll}
&\varphi_{n}\Big(\big((a,b_1), (a,b_2), \ldots , (a,b_n)\big)\Big)\\[5pt]
&=\big( z_0, z_1, \cdots, z_{i-1}, z_i, [z_{i-1}, z_i, z_i |_{b_{i} \mapsto b_{i+1}}], z_{i+2},  \cdots,   z_n \big)  \\[5pt]
&=\big( z_0, z_1, \cdots, z_{i-1}, z_i, [z_{i-1}, z_i, z_i ],  \cdots,   z_n \big) . 
\end{array}
\]
Here we have $\langle z_{i-1}, z_i, [z_{i-1}, z_i, z_i ] \rangle=z_i$ by (1) of Lemma~\ref{lem:tribrackets1}, which implies that $\varphi_{n} \Big(\big((a,b_1), (a,b_2), \ldots , (a,b_n)\big)\Big) \in D_{n-1}^{\rm Nie}(X)$, i.e. $\varphi_{n} \big(D_{n}(X)\big) \subset D_{n-1}^{\rm Nie}(X)$.

(2) We regard $\psi_n$ as a homomorphism from $C_{n-1}^{\rm Nie}(X)$ to $C_n(X)$.
It suffices to show that $\psi_{n} \big(D_{n-1}^{\rm Nie}(X)\big) \subset D_{n}(X)$.

For $(a, a_1, a_2, \ldots , a_n)\in X^{n+1}$, suppose $\langle a_{j-1}, a_j, a_{j+1}\rangle =a_j$ for some $j$ $(j \in \{ 2, 3,  \cdots, n-1\})$.
We then have
\[
\begin{array}{ll}
&\psi_n \big((a, a_1, a_2, \ldots , a_n)\big)\\[5pt]
&= \big((a,w_1), (a,w_2), \ldots ,(a,w_j), (a,w_{j+1}=w_j) , (a, w_{j+2}),  \ldots , (a, w_n)\big)
\end{array}
\]
since $w_{j+1}=w_{j}|_{a_{j} \mapsto \langle a_{j-1}, a_j, a_{j+1}\rangle } =w_{j}|_{a_{j} \mapsto a_j }=w_{j}$. Hence there is a pair of neighboring components which are equal. It follows that $\psi_n \big((a, a_1, a_2, \ldots , a_n)\big) \in D_{n}(X)$,  i.e. $\psi_{n} \big(D_{n-1}^{\rm Nie}(X) \big)  \subset D_{n}(X)$.
\end{proof}

\begin{lemma}\label{lem:inverses}
$\varphi_n$ and $\psi_n$ are the inverses of each other.
\end{lemma}
\begin{proof}
We first show that $\psi_n \circ \varphi_n ={\rm id}$. We have
\[
\begin{array}{ll}
&\psi_n \circ \varphi_n \Big(\big((a,b_1), (a,b_2), \ldots , (a,b_n)\big)\Big)\\[5pt]
&=\psi_n \Big( (a, z_1, \ldots , z_n) \Big)\\[5pt]
&=\big((a,w'_1), (a,w'_2), \ldots , (a, w'_n)\big),\\[5pt]
\end{array}
\]
where $w_1'=z_1 = b_1$, and for $i\geq2$, 
\[
\begin{array}{ll}
w'_i&=\blangle  a, z_1, \ldots , z_i \brangle\\[5pt]
&=\blangle  a, z_1, \ldots , z_{i-1}, \blsq a, z_1, \ldots , z_{i-1}; b_i \brsq \brangle\\[5pt]
&=b_i, \\[5pt]
\end{array}
\]
where the first and second equalities respectively  follow from (\ref{formula2-2}) and (\ref{formula2-1}) of Lemma~\ref{lem:zw}, and the third equality follows from (\ref{formula1-4}) of Lemma~\ref{lem:tribracketfomula}.

We next show that $\varphi_n \circ \psi_n ={\rm id}$. We have
\[
\begin{array}{ll}
&\varphi_n \circ \psi_n \big((a, a_1, a_2, \ldots , a_n)\big)\\[5pt]
&=\varphi_n  \Big( \big((a,w_1), (a,w_2), \ldots , (a, w_n)\big) \Big)\\[5pt]
&=(a, z'_1, \ldots , z'_n). \\[5pt]
\end{array}
\]
Here we show that $z'_{i}=a_{i}$ for $i \in \{1, 2, \cdots, n\}$ by induction. When $i=1$, $z'_{1}=w_{1}=a_{1}$.
Assuming that $z'_{k}=a_{k}$ for $k \leq {i-1}$, we have
 \[
\begin{array}{ll}
z'_{i}&=\blsq a, z'_1, \ldots , z'_{i-1}; w_{i} \brsq \\[5pt]
&=\blsq a, a_1, \ldots , a_{i-1}; \blangle  a, a_1, \ldots , a_{i} \brangle \brsq \\[5pt]
&=a_{i}, \\[5pt]
\end{array}
\]
where the first equality follows from (\ref{formula2-1}) of  Lemma~\ref{lem:zw}, the second equality follows from the assumption and  (\ref{formula2-2}) of Lemma~\ref{lem:zw},  and the third equality follows from  (\ref{formula1-4}) of Lemma~\ref{lem:tribracketfomula}.
\end{proof}

Now we consider about the composition $\varphi_{n-1} \circ \partial_n^{\rm LB} \circ \psi_n$.
For an integer $n\geq 2$, we have 
\begin{align*}
&\varphi_{n-1} \circ \partial_n^{\rm LB} \circ \psi_n \big( (a, a_1, a_2, \ldots , a_n) \big) \notag \\[3pt] 
&= \varphi_{n-1} \circ \partial_n^{\rm LB}\Big( \big((a,w_1), (a,w_2), \ldots , (a, w_n)\big) \Big)\notag \\[3pt]
&= \varphi_{n-1}\Big(\sum_{i=1}^{n} (-1)^i \big\{ \big( (a, w_1 ), \ldots, (a, w_{i-1}), (a, w_{i+1}), \ldots  , (a,w_n) \big) \notag \\[3pt]
&\hspace{2.5cm}- \big( (w_i, [a, w_1, w_i]=:b_{(i,1)} ), \ldots, (w_i, [a, w_{i-1}, w_i]=:b_{(i,i-1)}), \notag    \\[3pt]
&\hspace{3cm} (w_i, [a, w_i, w_{i+1}]=:b_{(i,i+1)}) ,\ldots  , (w_i, [a, w_i, w_{n}]=:b_{(i,n)})\big) \big\} \Big) \notag \\[3pt]
&=\sum_{i=1}^{n} (-1)^i (u_{(i,0)}, u_{(i,1)}, \ldots,  u_{(i,i-1)}, u_{(i,i+1)}, \ldots , u_{(i,n)})
\\[3pt]
&\hspace{1cm}+\sum_{i=1}^{n} (-1)^{i+1}  (v_{(i,0)}, v_{(i,1)}, \ldots,  v_{(i,i-1)}, v_{(i,i+1)}, \ldots , u_{(i,n)}), 
\end{align*}
where 
\[
u_{(1,j)}=\left\{
\begin{array}{ll}
a & (j=0),\\
w_2 & (j=2),\\
\big[ u_{(1,0)} , u_{(1,2)} , u_{(1,2)} |_{w_2 \mapsto w_3}\big] & (j=3),\\[3pt]
\big[ u_{(1,j-2)}, u_{(1,j-1)} , u_{(1,j-1)}|_{w_{j-1}\mapsto w_j}\big] & (j\geq 4),
\end{array}
\right.
\]
and for $i\geq 2$, 
\[
u_{(i,j)}=\left\{
\begin{array}{ll}
a & (j=0),\\
w_1 & (j=1),\\
\big[ u_{(i,i-2)} , u_{(i,i-1)} , u_{(i,i-1)} |_{w_{i-1} \mapsto w_{i+1}}\big] & (j=i+1),\\[3pt]
\big[ u_{(i,i-1)} , u_{(i,i+1)} , u_{(i,i+1)} |_{w_{i+1} \mapsto w_{i+2}}\big] & (j=i+2),\\[3pt]
\big[ u_{(i,j-2)}, u_{(i,j-1)} , u_{(i,j-1)}|_{w_{j-1}\mapsto w_j}\big] & (\mbox{otherwise}),
\end{array}
\right.
\]
 and where 
\[
v_{(1,j)}=\left\{
\begin{array}{ll}
w_i & (j=0),\\
b_{(1,2)} & (j=2),\\
\big[ v_{(1,0)} , v_{(1,2)} , v_{(1,2)} |_{b_{(1,2)} \mapsto b_{(1,3)}}\big] & (j=3),\\[3pt]
\big[ v_{(1,j-2)}, v_{(1,j-1)} , v_{(1,j-1)}|_{b_{(1,j-1)}\mapsto b_{(1,j)}}\big] & (j\geq 4),
\end{array}
\right.
\]
and for $i\geq 2$, 
\[
v_{(i,j)}=\left\{
\begin{array}{ll}
w_i & (j=0),\\
b_{(i,1)} & (j=1),\\
\big[ v_{(i,i-2)} , v_{(i,i-1)} , v_{(i,i-1)} |_{b_{(i, i-1)} \mapsto b_{(i, i+1)}}\big] & (j=i+1),\\[3pt]
\big[ v_{(i,i-1)} , v_{(i,i+1)} , v_{(i,i+1)} |_{b_{(i, i+1)} \mapsto b_{(i, i+2)}}\big] & (j=i+2),\\[3pt]
\big[ v_{(i,j-2)}, v_{(i,j-1)} , v_{(i,j-1)}|_{b_{(i, j-1)}\mapsto b_{(i,j)}}\big] & (\mbox{otherwise}).
\end{array}
\right.
\]
We recall that in Subsection~\ref{subsec:NHom}, a homomorphism $\partial_{n-1}^{\rm N} : C_{n-1}^{\rm N} (X) \to C_{n-2}^{\rm N} (X)$ was defined by 
\begin{align}
&\partial_{n-1}^{\rm N} \big( ( a, a_1, \ldots , a_{n} ) \big)  \notag \\
&= \sum_{i=0}^{n-1} (-1)^i \big( y_{(i,1)}, y_{(i,2)}, \ldots , y_{(i,i)}, a_{i+1}, a_{i+2}, \ldots , a_{n} \big)\notag  \\
&\hspace{1cm}+ \sum_{i=0}^{n-1} (-1)^{i+1} \big( a, a_1, \ldots , a_i, y_{(i,i+1)},y_{(i,i+2)},\ldots , y_{(i,n-1)} \big) \notag 
\end{align}
if $n> 0$, and $\partial_n^{\rm N}=0$ otherwise.

\begin{lemma}\label{lem:chainmap}
$\varphi_n$ and $\psi_n$ are chain maps between  the chain complexes $C_*^{\rm LB} (X)$ and $C_*^{\rm N} (X)$.
\end{lemma}
\begin{proof}
We may show that $\varphi_{n-1} \circ \partial_n^{\rm LB} \circ \psi_n = \partial_{n-1}^{\rm N}$ for $n\geq 2$.
We will show that 
\begin{subequations}
\begin{align}
&(u_{(i,0)}, u_{(i,1)}, \ldots,  u_{(i,i-1)}, u_{(i,i+1)}, \ldots , u_{(i,n)})\notag \\
 &= \big( a, a_1, \ldots , a_{i-1}, y_{(i-1,i)},y_{(i-1,i+1)},\ldots , y_{(i-1,n-1)} \big) \label{eq:1a}
\end{align}
\end{subequations}
and
\begin{subequations}
\begin{align} 
& (v_{(i,0)}, v_{(i,1)}, \ldots,  v_{(i,i-1)}, v_{(i,i+1)}, \ldots , v_{(i,n)})\notag \\
&=\big( y_{({i-1},1)}, y_{(i-1,2)}, \ldots , y_{(i-1,i-1)}, a_{i}, a_{i+1}, \ldots , a_{n} \big). \label{eq:1b}
\end{align}
\end{subequations}
 
First we  consider the above (\ref{eq:1a}) in the case of $i \geq 2$.
For the first $i$ components,  we have 
\[
\begin{array}{l}
(u_{(i,0)}, u_{(i,1)}, \ldots,  u_{(i,i-1)}) =\varphi_{i-1}\circ \psi_{i-1} \big( (a, a_1,\ldots , a_{i-1}) \big)\\[3pt]
={\rm id} \big((a, a_1,\ldots , a_{i-1}) \big)= (a, a_1,\ldots , a_{i-1}).
\end{array}
\] 
For $u_{(i,i+1)}$, we have
 \[
\begin{array}{ll}
u_{(i,i+1)}&= \blsq  u_{(i,0)}, u_{(i,1)}, \ldots , u_{(i,i-1)}; w_{i+1} \brsq \\[5pt]
&= \blsq a, a_1,\ldots , a_{i-1} ; \blangle  a, a_1, \ldots , a_{i+1} \brangle \brsq\\[5pt]
&= \blangle  a_{i-1}, a_{i} , a_{i+1} \brangle\\[5pt]
&= \langle  a_{i-1}, a_{i} , a_{i+1} \rangle\\[5pt]
&=y_{(i-1, i)},
\end{array}
\] 
where the first and second equalities respectively  follow from (1) and (2) of  Lemma~\ref{lem:zw}, the third equality follows from (\ref{formula1-4}) of  Lemma~\ref{lem:tribracketfomula}, and the fourth and fifth equalities follow from the definitions.

For $j \geq i+2$, we show that $u_{(i, j)}=y_{(i-1, j-1)}$ by induction for $j$. Assuming that $u_{(i, k)}=y_{(i-1, k-1)}$ for $i+1 \leq k \leq j-1$, we have
\[
\begin{array}{ll}
u_{(i,j)}&= \blsq  u_{(i,0)}, u_{(i,1)}, \ldots , u_{(i,i-1)},  u_{(i,i+1)}, \ldots , u_{(i,j-1)}; w_{j} \brsq \\[5pt]
&= \blsq a, a_1,\ldots , a_{i-1}, y_{(i-1, i)}, \ldots , y_{(i-1, j-2)} ; \blangle  a, a_1, \ldots , a_{j} \brangle \brsq\\[5pt]
&= \blsq  a_{i-1}, y_{(i-1, i)}, \ldots , y_{(i-1, j-2)} ; \blangle  a_{i-1}, \ldots , a_{j} \brangle \brsq\\[5pt]
&= \blsq  a_{i-1}, y_{(i-1, i)}, \ldots , y_{(i-1, j-2)} ; \blangle  a_{i-1}, y_{(i-1, i)}, \ldots , y_{(i-1, j-2)} , y_{(i-1, j-1)} \brangle \brsq\\[5pt]
&= y_{(i-1, j-1)},\\[5pt]
\end{array}
\] 
where the first equality follows from (1) of Lemma~\ref{lem:zw},  the second equality follows from the assumption, the third and fifth equalities follow from (4) of Lemma~\ref{lem:tribracketfomula}, and the fourth equality follows from (3) of  Lemma~\ref{lem:y}.

Next we consider the equality (\ref{eq:1b})  in the case of $i \geq 2$. For $0 \leq j \leq i-2$, we show that $v_{(i, j)}=y_{(i-1, j+1)}$ by induction for $j$. When $j=0$, we have 
\[
\begin{array}{ll}
v_{(i,0)}&=w_i\\[5pt]
&=  \blangle  a, a_1, \ldots , a_{i} \brangle\\[5pt]
&= y_{(i-1, 1)},\\[5pt]
\end{array}
\] 
where the first equality follows from the definition, the second equality follows from (2) of Lemma~\ref{lem:zw}, and  the third equality follows from (2) of Lemma~\ref{lem:y}.
When $j=1$, we have 
\[
\begin{array}{ll}
v_{(i,1)}&=b_{(i, 1)}\\[5pt]
&= [a, w_1, w_{i}]\\[5pt]
&= \blsq a, a_1; \blangle  a, a_1, \ldots , a_{i} \brangle \brsq \\[5pt]
&= \blangle  a_1, a_2, \ldots , a_{i} \brangle\\[5pt]
&= y_{(i-1, 2)},\\[5pt]
\end{array}
\] 
where the first and second equalities follow from the definitions, the third equality follows from (2) of Lemma~\ref{lem:zw} and the definition, the fourth equality follows from (4) of Lemma~\ref{lem:tribracketfomula}, and the fifth equality follows from (2) of Lemma~\ref{lem:y}.
Assuming that $v_{(i, k)}=y_{(i-1, k+1)}$ for $k \leq j-1$, we have
\[
\begin{array}{ll}
v_{(i,j)}&= \blsq  v_{(i,0)}, v_{(i,1)}, \ldots , v_{(i,j-1)}; b_{(i,j)} \brsq \\[5pt]
&= \blsq y_{(i-1,1)}, y_{(i-1,2)}, \ldots , y_{(i-1, j)} ; [ a, w_j, w_i ]  \brsq \\[5pt]
&= \blsq y_{(i-1,1)}, y_{(i-1,2)}, \ldots , y_{(i-1, j)} ; \blangle  y_{(i-1,1)}, y_{(i-1,2)}, \ldots , y_{(i-1, j+1)} \brangle \brsq\\[5pt]
&= y_{(i-1, j+1)},\\[5pt]
\end{array}
\] 
where the first equality follows from (1) of  Lemma~\ref{lem:zw}, the second equality follows from the assumption and the definition, the third equality follows from (4) of Lemma~\ref{lem:y}, and the fourth equality follows from (4) of  Lemma~\ref{lem:tribracketfomula}.

When $j=i-1$, we have
\[
\begin{array}{ll}
v_{(i,i-1)}&= \blsq  v_{(i,0)}, v_{(i,1)}, \ldots , v_{(i,i-2)}; b_{(i,i-1)} \brsq \\[5pt]
&= \blsq y_{(i-1,1)}, y_{(i-1,2)}, \ldots , y_{(i-1, i-1)} ;  [ a, w_{i-1}, w_i ] \brsq \\[5pt]
&= \blsq y_{(i-1,1)}, y_{(i-1,2)}, \ldots , y_{(i-1, i-1)} ; \blangle  y_{(i-1,1)}, y_{(i-1,2)}, \ldots , y_{(i-1, i-1)}, a_i \brangle \brsq\\[5pt]
&= a_i ,\\[5pt]
\end{array}
\] 
where the first equality follows from (1) of  Lemma~\ref{lem:zw}, the second equality follows from the assumption and the definition, the third equality follows from (4) of Lemma~\ref{lem:y}, and the fourth equality follows from (4) of  Lemma~\ref{lem:tribracketfomula}.

For $j \geq i+1$, we show that $v_{(i,j)}=a_j$ by induction for $j$. When $j=i+1$, we have 
\[
\begin{array}{ll}
v_{(i,i+1)}&= \blsq  v_{(i,0)},  \ldots ,v_{(i,i-2)}, v_{(i,i-1)}; b_{(i,i+1)} \brsq \\[5pt]
&= \blsq y_{(i-1,1)},  \ldots , y_{(i-1, i-1)}, a_i ;  [a, w_i, w_{i+1}] \brsq \\[5pt]
&= \blsq y_{(i-1,1)},  \ldots , y_{(i-1, i-1)}, a_i ; \blangle  y_{(i-1,1)},  \ldots , y_{(i-1, i-1)}, a_i, a_{i+1} \brangle  \brsq \\[5pt]
&= a_{i+1} ,\\[5pt]
\end{array}
\] 
where the first equality follows from (1) of Lemma~\ref{lem:zw}, the second equality follows from the assumption and the definition, the third equality follows from (5) of Lemma~\ref{lem:y}, and the fourth equality follows from (4) of  Lemma~\ref{lem:tribracketfomula}. Assuming that $v_{(i,k)}=a_k$ for $i+1 \leq k \leq j-1$, we have
\[
\begin{array}{ll}
v_{(i,j)}&= \blsq  v_{(i,0)},  \ldots  ,v_{(i,i-2)}, v_{(i,i-1)}, v_{(i,i+1)},  \ldots , v_{(i,j-1)}; b_{(i,j)} \brsq \\[5pt]
&= \blsq y_{(i-1,1)},  \ldots , y_{(i-1, i-1)}, a_i, a_{i+1}, \ldots , a_{j-1}; [a, w_i, w_{j}]  \brsq \\[5pt]
&= \blsq y_{(i-1,1)},  \ldots , y_{(i-1, i-1)}, a_i, a_{i+1}, \ldots , a_{j-1};\\
& \hspace{50pt} \blangle  y_{(i-1,1)}, \ldots , y_{(i-1, i-1)}, a_i, a_{i+1}, \ldots , a_j \brangle  \brsq \\[5pt]
&= a_{j} ,\\[5pt]
\end{array}
\] 
where the first equality follows from (1) of  Lemma~\ref{lem:zw}, the second equality follows from the assumption and the definition, the third equality follows from (5) of Lemma~\ref{lem:y}, and the fourth equality follows from (4) of  Lemma~\ref{lem:tribracketfomula}.

Therefore we have 
\begin{align*}
&\varphi_{n-1} \circ \partial_n^{\rm LB} \circ \psi_n \big( (a, a_1, a_2, \ldots , a_n) \big) \notag \\
&=\sum_{i=1}^{n} (-1)^i (u_{(i,0)}, u_{(i,1)}, \ldots,  u_{(i,i-1)}, u_{(i,i+1)}, \ldots , u_{(i,n)}) 
\\
&\hspace{1cm}+\sum_{i=1}^{n} (-1)^{i+1}  (v_{(i,0)}, v_{(i,1)}, \ldots,  v_{(i,i-1)}, v_{(i,i+1)}, \ldots , u_{(i,n)}) \\ 
&= \sum_{i=1}^{n} (-1)^i  \big( a, a_1, \ldots , a_{i-1}, y_{(i-1,i)},y_{(i-1,i+1)},\ldots , y_{(i-1,n-1)} \big) \notag  \\
&\hspace{1cm}+ \sum_{i=1}^{n} (-1)^{i+1} \big( y_{({i-1},1)}, y_{(i-1,2)}, \ldots , y_{(i-1,i-1)}, a_{i}, a_{i+1}, \ldots , a_{n} \big)\notag \\
&= \sum_{i=0}^{n-1} (-1)^{i+1} \big( a, a_1, \ldots , a_i, y_{(i,i+1)},y_{(i,i+2)},\ldots , y_{(i,n-1)} \big) \notag  \\
&\hspace{1cm}+ \sum_{i=0}^{n-1} (-1)^i \big( y_{(i,1)}, y_{(i,2)}, \ldots , y_{(i,i)}, a_{i+1}, a_{i+2}, \ldots , a_{n} \big) \notag \\
&=\partial_{n-1}^{\rm N} \big( ( a, a_1, \ldots , a_{n} ) \big).
\end{align*}

We leave the proof in the case of $i=1$ to the reader.

\end{proof}

The next theorem follows from Lemmas~\ref{lem:welldefined}, \ref{lem:inverses} and \ref{lem:chainmap}.
\begin{theorem}
$\varphi_n$ and $\psi_n$ are chain maps, and they are the inverses of each other.  
\end{theorem}

\section{Examples}
\label{Examples}

\begin{example} (Alexander local biquandles)
Let $X=R$ be a commutative ring with identity. Then for any 
choice of two units $x,y\in R$, we have a horizontal-tribracket known as an 
\textit{Alexander tribracket} defined by
\[[a,b,c]=xb+yc-xya.\]
This determines a local biquandle structure on $X^2$ by the maps
\begin{eqnarray*}
(a,b)\uline{\star}(a,c) & = & (c,xb+yc-xya)\\
(a,b)\oline{\star}(a,c) & = & (c,xc+yb-xya).
\end{eqnarray*}

For example, in the Alexander local biquandle on $X=\mathbb{Z}_5$ with
$x=3$ and $y=2$,
 we have
\[(1,3)\uline{\star}(1,4) = (4,3(3)+2(4)-(3)(2)1)=(4,9+8-6)=(4,1)\]
and
\[(1,3)\oline{\star}(1,4) = (4,3(4)+2(3)-(3)(2)1)=(4,12+6-6)=(4,2).\]
\end{example}

\begin{example} 

Let $X=G$ be a group. Then we have a horizontal-tribracket known as an 
\textit{Dehn tribracket} defined by
\[[a,b,c]=ba^{-1}c.\]
This determines a local biquandle structure on $X^2$ by the maps
\begin{eqnarray*}
(a,b)\uline{\star}(a,c) & = & (c,ba^{-1}c)\\
(a,b)\oline{\star}(a,c) & = & (c,ca^{-1}b).
\end{eqnarray*}

\end{example}

\begin{example}
More generally, we can represent a tribracket on a finite set 
$X=\{1,2,\dots, n\}$ with an operation $3$-tensor,
i.e. a list of $n$ $n\times n$ matrices, with the rule that the
element $[i,j,k]$ is the entry in the $i$th matrix's $j$th row and
$k$th column. Then \texttt{python} calculations show that there are
two horizontal-tribrackets with two elements, given by
\[\left[\left[\begin{array}{rr} 1& 2 \\2 & 1\end{array}\right],
\left[\begin{array}{rr}2 & 1 \\ 1 & 2\end{array}\right]\right]
\]
and
\[\left[
\left[\begin{array}{rr}2 & 1 \\ 1 & 2\end{array}\right],
\left[\begin{array}{rr} 1& 2 \\2 & 1\end{array}\right]
\right]
\]
and twelve tribrackets on the set of three elements.
\end{example}

Theorem~\ref{mainthm1} says that there is an isomorphism between the $(n-1)$st 
Niebrzydowski's (co)homology of a vertical-tribracket on $X$ and its corresponding
$n$th local biquandle (co)homology. We can represent an element
of $C_{\rm LB}^n(X)$ with an $(n+1)$-tensor whose entry in position 
$(z_0,z_1,\dots, z_n)$ is the coefficient of $\chi_{(z_0,z_1,\dots, z_n)}$.

\begin{example}
In \cite{NeedellNelson16}, an algebraic structure called \textit{biquasile}
was introduced, consisting of a pair of quasigroup operations 
$\ast,\cdot:X\times X\to X$ satisfying the conditions
\[\begin{array}{rcl}
a\ast(x\cdot (y\ast(a\cdot b))) & = & (a\ast(x\cdot y))\ast(x\cdot (y\ast((a\ast(x\cdot y))\cdot b))) \\
y\ast((a\ast (x\cdot y))\cdot b) & = & (y\ast(a\cdot b))\ast((a\ast (x\cdot (y\ast(a\cdot b))))\cdot b).
\end{array}\]
A biquasile determines a vertical-tribracket by 
\[\langle a,b,c\rangle =b\ast(a\cdot c).\]
For example, the biquasile structure on $X=\{1,2\}$ given by
\[\begin{array}{r|rr}
\ast & 1 & 2 \\ \hline
1 & 1 & 2 \\
2 & 2 & 1\\
\end{array} \quad
\begin{array}{r|rr}
\cdot & 1 & 2 \\ \hline
1 & 2 & 1 \\
2 & 1 & 2\\
\end{array}
\]
yields the tribracket with operation tensor
\[\left[\left[\begin{array}{rr}
2 & 1\\
1 & 2
\end{array}\right],
\left[\begin{array}{rr}
1 & 2\\
2 & 1
\end{array}\right]\right].\]
In \cite{ChoiNeedellNelson17}, the notion of \textit{Boltzmann enhancements}
for biquasile colorings of oriented knots and links was considered and conditions
for a function $f:X\times X\times X\to X$ to define a Boltzmann weight were 
identified. These correspond to local biquandle 2-cocycles in our present
notation. 
For example, the Boltzmann weight $\phi$ with $\mathbb{Z}_5$
coefficients in Example 5 in \cite{ChoiNeedellNelson17} corresponds to the local 
biquandle $2$-cocycle specified by the 3-tensor
\[\phi=\left[\left[\begin{array}{rr}
0 & 2\\
3 & 0
\end{array}\right],
\left[\begin{array}{rr}
0 & 4\\
3 & 0
\end{array}\right]\right]\]
or alternatively
\[\phi=2\chi_{((1,1),(1,2))}+3\chi_{((1,2),(1,1))}+4\chi_{((2,1),(2,2))}+3\chi_{((2,2),(2,1))}\]
where 
\[\chi_{\vec{x}}(\vec{y})=\left\{\begin{array}{ll}
1 & \vec{x}=\vec{y}\\
0 & \vec{x}\ne\vec{y}
\end{array}\right.\]
for $\vec{x},\vec{y}\in\{\left((a,b),(a,c)\right)\ |\ a,b,c\in X\}$.
\end{example}

\begin{remark}
We have defined the cocycle invariant as a multiset of diagram weights over the
set of colorings of a diagram. It is common in the literature
(see \cite{CarterJelsovskyKamadaLangfordSaito03, ElhamdadiNelson} etc.) to encode a multiset in a polynomial format by
summing a formal variable $u$ to the power of each element in the multset,
effectively rewriting multiplicities as coefficients and elements as exponents.
For example, the multiset $\{0,0,2,2,2,4\}$ becomes
\[u^0+u^0+u^2+u^2+u^2+u^4=2+3u^3+u^4.\]
This notation is convenient since it is easier to compare polynomials
visually than to compare multisets visually, and evaluation of the polynomial
at $u=1$ yields the number of colorings.
\end{remark}

\begin{example}
Let $X=\{1,2,3\}$ and consider the horizontal-tribracket defined by the 3-tensor
\[\left[
\left[\begin{array}{rrr}
2& 3& 1 \\
3& 1& 2 \\
1& 2& 3
\end{array}\right],\left[\begin{array}{rrr}
1& 2& 3 \\
2& 3& 1 \\
3& 1& 2
\end{array}\right],\left[\begin{array}{rrr}
3& 1& 2 \\
1& 2& 3 \\
2& 3& 1
\end{array}\right]
\right].\]
Our \texttt{python} computations reveal $3$-cocycles in $H^3_{\rm LB}(X;\mathbb{Z}_5)$
including
\[\theta=
\left[
\left[\begin{array}{rrr}
0 & 0 & 0 \\
1 & 0 & 3 \\
1 & 3 & 0
\end{array}\right],\left[\begin{array}{rrr}
0 & 0 & 0 \\
1 & 0 & 0 \\
1 & 1 & 0
\end{array}\right],\left[\begin{array}{rrr}
0 & 3 & 3 \\
3 & 0 & 2 \\
3 & 4 & 0
\end{array}\right]\right].
\]
Let us illustrate the computation of the cocycle enhancement invariant
for the Hopf link using this cocycle. There are 27 semi-arc $X^2$- (and region $X$-)colorings of
the Hopf link; for each, we must compute the diagram weight. Writing the
colorings of the diagrams in terms of 2-parallels, we have the following 
weight contributions at each crossing:
\[\includegraphics{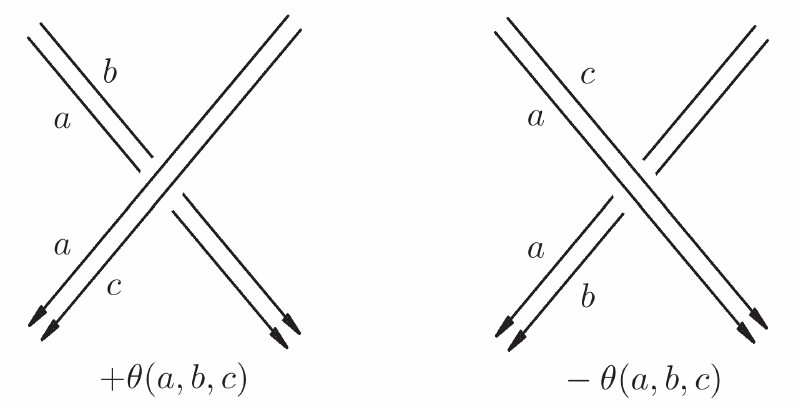}\]
Then for example the coloring below has diagram weight 
$\theta(1,1,1)+\theta(1,1,1)=0+0=0$,
\[\includegraphics{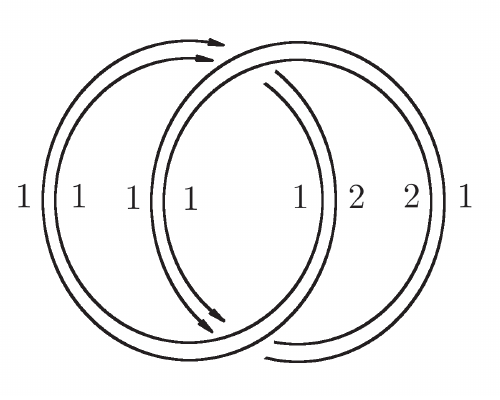}\]
while the coloring
\[\includegraphics{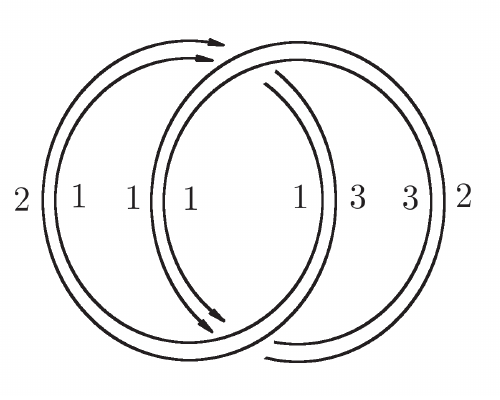}\]
has diagram weight
$\theta(1,2,1)+\theta(1,1,2)=1+0=1$.
Repeating for all 27 colorings, we obtain the multiset invariant value
\[\Phi_{\theta}(L2a1)=
\{0,0,0,0,0,0,0,0,0,
 1,1,1,1,1,1,1,1,1,
 1,1,1,1,1,1,1,1,1\}\]
or in polynomial form, $\Phi_{\theta}(L2a1)=9+18u$.

The invariant takes the following values on links with
up to seven crossings:
\[
\begin{array}{r|l}
\Phi_{\theta}(L) & L \\ \hline
9 & L6a4 \\
27 & L5a1, L7a1, L7a3, L7a4 \\ 
9+18u & L2a1, L7a5, L7a6 \\
9+18u^2 & L4a1, L6a1, L7a2 \\
9+18u^3 & L6a2, L6a3, L7n1 \\
9+54u^2+18u^3 & L6a5, L6n1 \\
45+18u+18u^2 & L7a7 \\
14+3u+2u^2+4u^3+4u^4 & L7n2 \\
\end{array}
\]
\end{example}

\begin{example}
In  Example 4.17 of \cite{Niebrzydowski2}, a coloring of the trefoil knot by the
Alexander tribracket with $[a,b,c]=a+b-c$, i.e. given by the 3-tensor
\[
\left[
\left[\begin{array}{rrr}
1 & 3 & 2 \\
2 & 1 & 3 \\
3 & 2 & 1
\end{array}\right], \left[\begin{array}{rrr}
2 & 1 & 3 \\
3 & 2 & 1 \\
1 & 3 & 2 
\end{array}\right], \left[\begin{array}{rrr}
3 & 2 & 1 \\
1 & 3 & 2 \\
2 & 1 & 3
\end{array}\right]
\right]
\]
was observed to have a nontrivial weight sum with respect to the 
cocycle which in our notation is
\[\theta=
\left[
\left[\begin{array}{rrr}
0 & 1 & 0 \\
0 & 0 & 0 \\
0 & 1 & 0 
\end{array}\right], \left[\begin{array}{rrr}
0 & 0 & 1 \\
0 & 0 & 0 \\
1 & 0 & 0 
\end{array}\right], \left[\begin{array}{rrr}
0 & 2 & 0 \\
1 & 0 & 0 \\
2 & 0 & 0
\end{array}\right]
\right].\]

Our \texttt{python} computations reveal that the right-handed trefoil $3_1$
knot has $\Phi_{\theta}(3_1)=9+18u$ while the left-handed trefoil $\overline{3_1}$
has $\Phi_{\theta}(3_1)=9+18u^2$, distinguishing the two and showing that
these invariants can distinguish mirror images. We note also that the
square knot and granny knot are distinguished by this invariant with
$\Phi_{\theta}(3_1\#3_1)=45+18u+18u^2$ and
$\Phi_{\theta}(3_1\#\overline{3_1})=9+36u+36u^2$.
For prime knots with eight or fewer crossings, the nontrivial values are listed 
in the table.
\[
\begin{array}{r|l}
\Phi_{\theta}(K) & K \\ \hline
27 & 6_1, 8_{10}, \\
9+18u & 3_1, 7_4, 7_7, 8_{15}, 8_{21} \\
9+18u^2 & 8_5, 8_{19} \\
15+6u+6u^2 & 8_{11} \\
18+6u+3u^2 & 8_{20} \\
33+24u+24u^2 & 8_{18} \\
\end{array}
\]
\end{example}

\begin{example}
Let $X$ be the set $\{1,2,3\}$ with horizontal-tribracket given by the 3-tensor
\[\left[
\left[\begin{array}{rrr}
2 & 3 & 1 \\
1 & 2 & 3 \\
3 & 1 & 2
\end{array}\right],\left[\begin{array}{rrr}
3 & 1 & 2 \\
2 & 3 & 1 \\
1 & 2 & 3
\end{array}\right],\left[\begin{array}{rrr}
1 & 2 & 3 \\
3 & 1 & 2 \\
2 & 3 & 1
\end{array}\right]
\right].\]
With the cocycle with $\mathbb{Z}_3$ coefficients specified by
\[\theta=\left[
\left[\begin{array}{rrr}
0 & 0 & 0 \\
2 & 0 & 2 \\
2 & 2 & 0
\end{array}\right],\left[\begin{array}{rrr}
0 & 1 & 2 \\
2 & 0 & 2 \\
1 & 0 & 0
\end{array}\right],\left[\begin{array}{rrr}
0 & 0 & 1 \\
2 & 0 & 0 \\
1 & 1 & 0
\end{array}\right]
\right], \]
we compute the following invariant values for prime knots with up to 8 crossings
and prime links with up to seven crossings. Note in particular that since the 
trivial invariant values for links of two and three components are 27 and 81 
respectively, this example detects the non-triviality of every link on the list.
\[\begin{array}{r|l}
\Phi_{\theta}(K) & K \\\hline
9 & 4_1,5_1,5_2,6_2,6_3,7_1,7_2,7_3,7_5,7_6, 8_1, 8_2, 8_3,8_4,8_6,8_7,8_8,8_9,
8_{12}, 8_{13}, 8_{14}, 8_{16}, 8_{17} \\
27 & 6_1, 8_{10}\\
9+18u & 3_1, 7_4, 7_7, 8_{15}, 8_{21} \\
9+18u^2 & 8_5, 8_{19} \\
13+4u+10u^2 & 8_{11} \\
16+4u+7u^2& 8_{20} \\
25+28u+28u^2 & 8_{18}
\end{array}\]
\[\begin{array}{r|l}
\Phi_{\theta}(L) & L \\\hline
9 & L2a1, L4a1, L5a1, L6a2, L6a4, L6n1, L7a2,L7a3, L7a4, L7a6, L7a7,L7n1, L7n2 \\
9+18u & L6a1 \\
9+18u^2 & L6a3, L6a5, L7a1 \\
13+10u+4u^2 & L7a5
\end{array}\]
\end{example}

\section*{Acknowledgments}

The first author was supported by Simons Foundation collaboration grant 316709.
The second author was supported by JSPS KAKENHI Grant Number 16K17600.


\begin{thebibliography}{000}


\bibitem{CarterElhamdadiSaito04}
S.~Carter, M.~Elhamdadi and M.~Saito,
\textit{Homology theory for the set-theoretic Yang-Baxter equation and knot invariants from generalizations of quandles},
Fund. Math. \textbf{184} (2004), 31--54.


\bibitem{CarterJelsovskyKamadaLangfordSaito03}
J.~S.~Carter, D.~Jelsovsky, S.~Kamada, L.~Langford and M.~Saito,
\textit{Quandle cohomology and state-sum invariants of knotted curves and surfaces},
Trans. Amer. Math. Soc. \textbf{355} (2003), 3947--3989.

\bibitem{CarterKamadaSaito01} 
J.~S.~Carter, S.~Kamada, and M.~Saito, 
\textit{Geometric interpretations of quandle homology}, 
J. Knot Theory Ramifications \textbf{10} (2001), 345--386.

\bibitem{CenicerosElhamdadiGreenNelson14}
J.~Ceniceros, M.~Elhamdadi, M.~Green and S.~Nelson
\textit{Augmented biracks and their homology},
Internat. J. Math. \textbf{25} (2014), 1450087, 19 pp.

\bibitem{ChoiNeedellNelson17}
W.~Choi, D.~Needell and  S.~Nelson,
\textit{Boltzmann enhancements of biquasile counting invariants},
arXiv:1704.02555.

\bibitem{ElhamdadiNelson}
M.~Elhamdadi and S.~Nelson. 
\textit{Quandles: An introduction to the algebra of knots,} 
volume {\bf 74} of Student
Mathematical Library. American Mathematical Society, Providence, RI, 2015.


\bibitem{FennRourke92}
R.~Fenn and C.~Rourke,
\textit{Racks and links in codimension two},
J. Knot Theory Ramifications {\bf 1} (1992), 343--406.

\bibitem{FennRourkeSanderson92}
R.~Fenn, C.~Rourke and  B.~Sanderson, 
\textit{An introduction to species and the rack space},
Topics in knot theory (Erzurum, 1992), 33--55,
NATO Adv. Sci. Inst. Ser. C Math. Phys. Sci. {\bf 399}, Kluwer Acad. Publ., Dordrecht, 1993.

\bibitem{FennRourkeSanderson95}
R.~Fenn, C.~Rourke and B.~Sanderson,
\textit{Trunks and classifying spaces},
Appl. Categ. Structures \textbf{3} (1995), 321--356.



\bibitem{GravesNelsonTamagawa}
P.~Graves, S.~Nelson and S.~Tamagawa
\textit{Niebrzydowski Algebras and Trivalent Spatial Graphs},
arXiv:1805.00104.


\bibitem{Joyce82}
D.~Joyce,
\textit{A classifying invariant of knots, the knot quandle},
J.~Pure Appl. Alg. \textbf{23} (1982), 37--65.

\bibitem{KR02}
L.~H.~Kauffman and D.~Radford,
\textit{Bi-oriented quantum algebras, and generalized Alexander polynomial for virtual links},
AMS. Contemp. Math. {\bf 318} (2002), 113--140.


\bibitem{KimNelson} 
J.~Kim and S.~Nelson. 
\textit{Biquasile colorings of oriented surface-links},
Topology and its Applications, \textbf{236} (2018) 64--76.


\bibitem{Matveev82}
S.~V.~Matveev,
\textit{Distributive groupoids in knot theory},
Mat. Sb. (N.S.) \textbf{119(161)} (1982), 78--88.



\bibitem{NeedellNelson16}
D.~Needell and S.~Nelson, 
\textit{Biquasiles and dual graph diagrams}, 
J. Knot Theory Ramifications \textbf{26(8)} (2017), 1750048, 18 pp. 

\bibitem{NeedellNelsonShi}
D.~Needell, S.~Nelson and Y.~Shi,
\textit{Tribracket Modules},
arXiv:1808.04421.

\bibitem{Niebrzydowski0}
M.~Niebrzydowski,
\textit{On some ternary operations in knot theory.}
Fund. Math. \textbf{225(1)} (2014), 259--276. 

\bibitem{Niebrzydowski1}
M.~Niebrzydowski,
\textit{Ternary quasigroups in knot theory},
arXiv:1708.05330.

\bibitem{Niebrzydowski2}
M.~Niebrzydowski,
\textit{Homology of ternary quasigroups yielding invariants of knots and knotted surfaces},
arXiv:1706.04307.

\bibitem{PicoNelson}
S.~Pico and S.~Nelson,
\textit{Virtual Tribrackets},
arXiv:1803.03210.

\bibitem{Roseman}
D.~Roseman, 
\textit{Reidemeister-type moves for surfaces in four-dimensional space}, 
Knot theory(Warsaw, 1995), 347--380, Banach Center Publ., {\bf 42}, Polish Acad. Sci., Warsaw (1998).




\end{thebibliography}
\end{document}